\documentclass[12pt,reqno]{amsart}
\AtBeginDocument{\addtocontents{toc}{\protect\setlength{\parskip}{0pt}}}

\usepackage{amsfonts}
\usepackage{dsfont}
\usepackage{amsmath}
\usepackage{amssymb}
\usepackage{mathrsfs}
\usepackage{url}
\usepackage{caption}
\usepackage{enumitem}
\usepackage{graphicx, subfig}
\usepackage[breaklinks]{hyperref}
\usepackage{tikz-cd}
\usetikzlibrary{calc}
\newtheorem{theorem}{Theorem}[section]
\newtheorem{lemma}[theorem]{Lemma}
\newtheorem{proposition}[theorem]{Proposition}

\theoremstyle{definition}
\newtheorem{definition}[theorem]{Definition}
\newtheorem{example}[theorem]{Example}
\newtheorem{corollary}[theorem]{Corollary}

\theoremstyle{remark}
\newtheorem{remark}[theorem]{Remark}

\numberwithin{equation}{section}

\begin{document}

\title[anisotropic FUP and open quantum map]
{Anisotropic 2D FUP and  quantum open baker's map}

\author{Long Jin, An Zhang and Hong Zhang}

\address{Yau Mathematical Sciences Center, Jingzhai, Tsinghua University, Haidian District, Beijing, 100084, P.R.China}
\email{jinlong@mail.tsinghua.edu.cn}
\email{hong-zha23@mails.tsinghua.edu.cn}
\address{School of Mathematical Sciences, Beihang University, 37 Xueyuan road, Haidian district, Beijing, 100191, P.R.China}
\email{anzhang@buaa.edu.cn}

\date{\today. \ \emph{2020 AMS MSC.}   28A80, 42B10, 43A32, 58J50, 58J40, 81Q50, 81Q12. }
\maketitle
\begin{abstract}
We prove an \emph{essential spectral gap} for 2D \emph{anisotropic} quantum open baker's map. This extends the 1D results of Dyatlov--Jin \cite{dyatlov2017resonances} and the isotropic 2D results of Cohen \cite{Cohen}. The key ingredient is the \emph{anisotropic discrete fractal uncertainty principle (FUP)} associated with a 2D anisotropic fractal set called
the \emph{Bedford–McMullen carpet}.  We also study the relation between our anisotropic discrete FUP and its continuous counterpart in the spirit of Dyatlov--Jin \cite{dyatlov2018dolgopyat} and \cite{Cohen}. In particular, we prove \emph{continuous FUP} for 2D \emph{anisotropic porous} sets, extending the (high-dimensional) isotropic results of Cohen \cite{MR4927737}. To the best of our knowledge, the anisotropic (line) porosity condition---a variant of Cohen's line porosity and stronger than ball porosity---appears to be new to the literature. 
\end{abstract}

\tableofcontents  

\section{Introduction}

The study of quantum chaos seeks to understand how chaotic classical dynamics manifest in quantum mechanical properties. For open quantum chaotic systems, it is expected that the uniform hyperbolicity of the trapped set leads to a positive essential spectral gap for the imaginary part of the scattering resonances (see, e.g., \cite[Conjecture 3]{zworski2017mathematical}). This has been confirmed for convex cocompact hyperbolic surfaces by Bourgain--Dyatlov \cite{bourgain2018spectral}, for two-dimensional Gaspard--Rice system (see \cite{gaspard1989scattering}) by Vacossin \cite{MR4736525}, and for negatively curved asymptotically hyperbolic surfaces 
by Tao \cite{tao2024spectral}. These results rely on the one-dimensional fractal uncertainty principle (FUP), which was proposed by Dyatlov--Zahl \cite{dyatlov2016spectral}. For higher-dimensional models, the recent breakthrough by Cohen \cite{MR4927737} on the FUP in higher dimensions implies a positive essential spectral gap for a large family of convex cocompact hyperbolic manifolds  (see Tao \cite{Tao2025PhD} and the authors \cite{jin2026quantitative}). 
However, in general, an open quantum chaotic system in higher dimensions (e.g., 3D Gaspard--Rice model) exhibits anisotropic properties, meaning that there are different Lyapunov exponents    along different expanding directions. It seems that previous works on FUP are not directly applicable to such systems. 

In this paper, we prove an essential spectral gap for a discrete model, namely the 2D anisotropic quantum open baker's map whose trapped set is the Bedford--McMullen carpet \cite{MR771063}, which is the simplest anisotropic fractal set, having different scaling properties in different directions. This extends the one-dimensional results of Dyatlov--Jin \cite{dyatlov2017resonances} and the isotropic version of Cohen \cite{Cohen} in two dimensions.

\vskip2mm
\noindent\textbf{Notation.} For $M\in\mathbb N$, $\mathbb Z_M=\mathbb Z/M\mathbb Z=\{0,\dots,M-1\}$.
For vectors $\mathbf M=(M_1,M_2)$ and $\mathbf N=(N_1,N_2)$ in $\mathbb R^2$, we use the following notations:
\[\mathbf M\mathbf N=(M_1N_1, M_2N_2),\quad \mathbf M/\mathbf N=(M_1/N_1, M_2/N_2),\quad \mathbf M\cdot\mathbf N=M_1N_1+M_2N_2.\]
If $\mathbf N\in\mathbb N^2$, set $\mathbb Z_{\mathbf N}:=\mathbb Z_{N_1}\times\mathbb Z_{N_2}$. Then $\ell_{\mathbf N}^2=\ell^2(\mathbb Z_{\mathbf N})$ denotes the space of functions $u:\mathbb Z_{\mathbf N}\to\mathbb C$ for which the Hilbert norm
\[\|u\|_{\ell_{\mathbf N}^2}^2=\sum_{\mathbf j\in\mathbb Z_{\mathbf N}}|u(\mathbf j)|^2<\infty.\]

We freely use the Hardy, Landau, Poincar\'e and Vinogradov  notations: $X\le C_kY$, $X=O_k(Y)$, $X\ll Y$ and $X\lesssim_k Y$ (and also $X\sim_k Y$).

\subsection{Anisotropic quantum baker's map}
Let $\mathbf M=(M_1, M_2)\in \mathbb N^2$ be a \emph{base} with $M_1>M_2\geq 2$, and let $\mathcal{A}\subset \mathbb Z_{\mathbf M}$ be a non-empty set, which we call the \emph{alphabet}. 
We consider the sequence of scales $\{\mathbf{N}=(N_1,N_2)\}_{k}$, where $N_j=M_j^k$ for $j=1,2$ and $k\in\mathbb N$, so that $N_j\to\infty$ as $k\to\infty$. 
The two scales $N_1$ and $N_2$ are related by
\[
    N_2 = N_1^\alpha, \qquad \alpha := \frac{\log M_2}{\log M_1} \in (0,1).
\]


Given a fixed cutoff function $\chi\in C_c^\infty((0,1)^2;[0,1])$, we define its discretization $\chi_{\mathbf{N}}$ as a function on ${\mathbb Z_{\mathbf N}}$ by
\[
    \chi_{\mathbf{N}}(\mathbf j)=
    \chi\left({\mathbf j}/{\mathbf N}\right),
    \quad \mathbf j\in {\mathbb Z_{\mathbf N}}.
\]
Then the quantum open baker's map $B_{\mathbf{N}}$ is an operator on ${\ell^2_{\mathbf N}}$ defined by
\begin{equation}\label{eq:BN}
        B_{\mathbf{N}}:=\sum_{\mathbf{a}\in \mathcal{A}} B_{\mathbf{N}}^{\mathbf{a}}, \qquad
        B_{\mathbf{N}}^{\mathbf{a}}:=\mathcal{F}_{\mathbf{N}}^*\, \Pi_{\mathbf{N},\mathbf{a}}^* \,
        \chi_{\mathbf{N}/\mathbf{M}} \,\mathcal{F}_{\mathbf{N}/\mathbf M} \, 
        \chi_{\mathbf{N}/\mathbf{M}}\, \Pi_{\mathbf{N},\mathbf{a}} ,
\end{equation}
where $\mathcal{F}_{\mathbf{N}}$ denotes the discrete Fourier transform \eqref{eq:DFT} on ${\mathbb Z_{\mathbf N}}$, and $\Pi_{\mathbf{N},\mathbf{a}}$ is the projection operator defined by
\[        \Pi_{\mathbf{N},\mathbf{a}} : {\ell^2_{\mathbf N}} \to 
        \ell^2_{\mathbf N/\mathbf M}, \qquad 
        \Pi_{\mathbf{N},\mathbf{a}}\,u(\mathbf j) =u\left( \mathbf j+\mathbf a\,\frac{\mathbf N}{\mathbf M}\right).
\]
This quantum map corresponds to the Fourier integral operator associated with the classical map, which is a symplectic relation on $(\mathbb T^2)^2$
\[\kappa:  (\mathbf x,\mathbf y)\in \left(\frac{a_1}{M_1},\frac{a_1+1}{M_1}\right) \times \left(\frac{a_2}{M_2},\frac{a_2+1}{M_2}\right) \times \mathbb T^2 ~\mapsto~ \left(\mathbf M \mathbf x-\mathbf a, \frac{\mathbf y+\mathbf a}{\mathbf M}\right),   \qquad  \mathbf a\in \mathcal A. \]
Since there are no interactions between the projection operators $\Pi_{\mathbf{N},\mathbf{a}}$ for distinct $\mathbf{a}$, we have the bound
\[
\|{B}_{\mathbf N}\|_{\ell^2\to \ell^2}\leq 1.
\]
Consequently, the \emph{spectrum} $\operatorname{Sp}(B_{\mathbf N})$ is contained in the unit disk.

Our first main result concerns the existence of a spectral gap for the anisotropic quantum open baker's map.

\begin{theorem}[Spectral gap]\label{thm:spectralgap}
    Suppose that the alphabet $\mathcal A$ has no full row or no full column. Then there exists some $\beta>0$ such that
    \begin{equation}\label{eq:spectralgap}
        \limsup_{N\to \infty} \max\{|\lambda|:\lambda\in \operatorname{Sp}(B_\mathbf{N})\} \le M_2^{-\beta},
    \end{equation}
    where $B_\mathbf{N}$ is the quantum open baker's map defined in \eqref{eq:BN}.
\end{theorem}

The proof follows the basic strategies of \cite{dyatlov2017resonances} and \cite{Cohen}, and uses the following anisotropic discrete fractal uncertainty principle.

\subsection{Anisotropic  FUP}
With the base $\mathbf M$ and alphabet $\mathcal A\subset \mathbb Z_{\mathbf M}$ as above, we define the iterated \emph{discrete Cantor sets} $\mathcal X_k$ and their closed \emph{drawings} $\mathbf X_k$ as follows:
\begin{align}\label{eq:dcantorset}
    \mathcal{X}_k &= \bigl\{(a_1^{(0)}+\cdots+a_1^{(k-1)}M_1^{k-1},\; a_2^{(0)}+\cdots+a_2^{(k-1)}M_2^{k-1}) : (a_1^{(j)},a_2^{(j)})\in\mathcal{A}\bigr\} \subset \mathbb Z_{\mathbf N}, \nonumber\\[3pt]
    \mathbf{X}_{k} &= \overline{\bigl\{(x,y)\in \mathbb{T}^2 : (\lfloor M_1^k x \rfloor,\lfloor M_2^k y \rfloor)\in \mathcal{X}_k\bigr\}} \subset \mathbb{T}^2.
\end{align}
We write $\mathbf{X}=\bigcap_k \mathbf{X}_k\subset \mathbb T^2$ for the \emph{limiting Cantor set}.
Intuitively, we use the projection
 \[  \mathbf j=(j_1,j_2)\in \mathbb Z_{\mathbf N} \;\longmapsto\; \mathbf j/\mathbf N=(j_1/N_1,j_2/N_2)\in \mathbb{T}^2
\]
to relate the discrete setting $\mathbb Z_{\mathbf N}$ to the continuous setting $\mathbb{T}^2$. 
The limiting Cantor set $\mathbf{X}\subset \mathbb{R}^2$ is called a \textit{Bedford–McMullen carpet}. It is a self‑affine fractal generalization of the well‑known self‑similar Sierpiński carpet. 
See, for example, McMullen \cite{MR771063} for the Hausdorff dimension of Bedford–McMullen carpets. Figure~\ref{fig:aniscantor} shows an example of an anisotropic Cantor iterate model. 
\newcommand{\drawfractal}[1]{%
    \ifnum#1=0
        \fill[black!60] (0,0) rectangle (1,1);
    \else
        \pgfmathtruncatemacro{\newdepth}{#1-1}
        \begin{scope}[shift={(0.666667, 0.5)}, xscale=0.333333, yscale=0.5]
            \drawfractal{\newdepth}
        \end{scope}
        \begin{scope}[shift={(0.333333, 0)}, xscale=0.333333, yscale=0.5]
            \drawfractal{\newdepth}
        \end{scope}
        \begin{scope}[shift={(0, 0.5)}, xscale=0.333333, yscale=0.5]
            \drawfractal{\newdepth}
        \end{scope}
    \fi
}

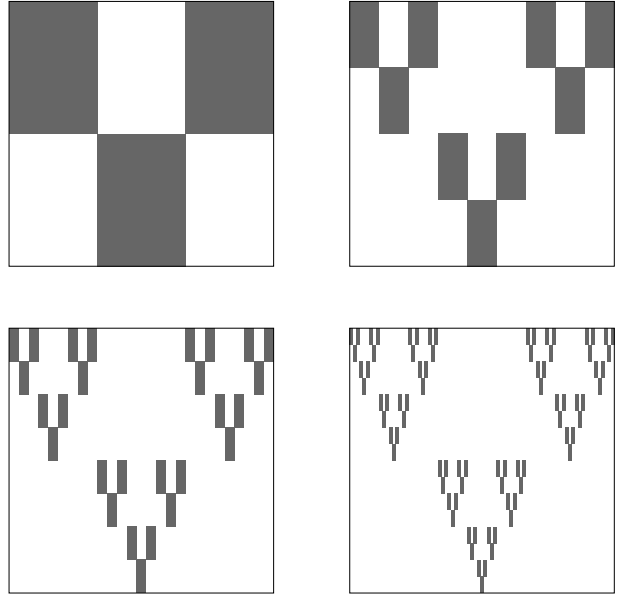
\begin{figure}[htbp]
\centering
\setlength{\tabcolsep}{0.5cm}   
\renewcommand{\arraystretch}{5} 
\begin{tabular}{c c}
    \begin{tikzpicture}[scale=3.5]
        \drawfractal{1}
        \draw[very thin] (0,0) rectangle (1,1);
    \end{tikzpicture}
    &
    \begin{tikzpicture}[scale=3.5]
        \drawfractal{2}
        \draw[very thin] (0,0) rectangle (1,1);
    \end{tikzpicture}
    \\
    \begin{tikzpicture}[scale=3.5]
        \drawfractal{3}
        \draw[very thin] (0,0) rectangle (1,1);
    \end{tikzpicture}
    &
    \begin{tikzpicture}[scale=3.5]
        \drawfractal{4}
        \draw[very thin] (0,0) rectangle (1,1);
    \end{tikzpicture}
    \\[0.5ex] 
\end{tabular}
\caption{Anisotropic Cantor iterates, to a Bedford-Mcmullen carpet. Here we choose $M_1=3,M_2=2$ and $\mathcal{A}=\{(0,1),(1,0),(2,1)\}$.}
\label{fig:aniscantor}
\end{figure}

Let $\{\mathcal X_k\}, \{\mathcal Y_k\}\subset \mathbb Z_{\mathbf N}$ be two sequences of discrete Cantor sets, and let $\mathbf X, \mathbf Y \subset \mathbb T^2$ be their respective limiting Cantor sets, all constructed from a common base $\mathbf M$ and from two alphabets $\mathcal A$ and $\mathcal B$ (respectively) via \eqref{eq:dcantorset}.

The following theorem, our second main result, completely characterizes the conditions on the alphabets $\mathcal A$ and $\mathcal B$ under which the anisotropic discrete FUP holds.

\begin{theorem}[DFUP]\label{thm:DFUP}
Suppose the alphabets $\mathcal A,\mathcal B$ satisfy one of the following conditions:
\begin{itemize}
    \item $\mathcal A$ or $\mathcal B$ has no full row and no full column;
    \item both $\mathcal A$ and $\mathcal B$ have no full row, or both have no full column.
\end{itemize}
Then there exist $\beta>0$ and $C>0$, independent of $k$, such that
\begin{equation}\label{eq:dfup}
    \|\mathds{1}_{\mathcal{Y}_k}\mathcal{F}_{\mathbf N}\mathds{1}_{\mathcal{X}_k}\|_{{\ell^2_{\mathbf N}}\to {\ell^2_{\mathbf N}}}\leq C N_2^{-\beta},
\end{equation}
where $\mathcal{F}_{\mathbf N}:{\ell^2_{\mathbf N}}\to {\ell^2_{\mathbf N}}$ is the {anisotropic discrete Fourier transform}
\begin{equation}\label{eq:DFT}
    \mathcal{F}_{\mathbf N}u(\mathbf j)=\frac{1}{\sqrt{N_1N_2}} \sum_{\mathbf l\in \mathbb Z_{\mathbf N}} \exp\left(-2\pi i \,\frac{\mathbf j}{\mathbf N}\cdot \mathbf l\right)u(\mathbf l).
\end{equation}
Otherwise, for any $k\geq 0$,
\[
\|\mathds{1}_{\mathcal{Y}_k}\mathcal{F}_{\mathbf N}\mathds{1}_{\mathcal{X}_k}\|_{{\ell^2_{\mathbf N}}\to {\ell^2_{\mathbf N}}}=1.
\]
\end{theorem}

\begin{remark} In an alphabet $\mathcal A$, a \emph{row} corresponds to fixing the second coordinate, and a \emph{column} corresponds to fixing the first coordinate. 
The condition on alphabets in Theorem \ref{thm:DFUP} excludes the situation where one alphabet contains a full row while the other contains a full column. This is an anisotropic strengthening of Cohen's isotropic condition \cite{Cohen}, which requires that there be no two orthogonal lines, one in each of the two limiting Cantor sets.
\end{remark}

\begin{remark}
Typically, the discrete FUP is weaker than its continuous counterpart (see Theorem~\ref{thm:C-D}). We are unable to prove this continuous counterpart, which involves an anisotropic semiclassical Fourier transform. However, as another main theorem, we obtain an anisotropic continuous FUP for the (isotropic) semiclassical Fourier transform (see Theorem~\ref{thm:cfup}), following the idea of \cite{dyatlov2024semiclassical}. Here we present a version for Bedford--McMullen carpets.
\end{remark}

\begin{theorem}[CFUP for BM carpet]\label{thm:cfup-bm}
Suppose that both $\mathcal A$ and $\mathcal B$ have no full row. 
Then there exist $\beta>0$ and $C>0$, independent of $h\in(0,1)$,  such that
\begin{equation}\label{eq:cfup-BD}
\|\mathds{1}_{\mathbf X+[-h,h]\times[-h^\alpha,h^\alpha]} \mathcal{F}_h \mathds{1}_{\mathbf Y+[-h,h]\times[-h^\alpha,h^\alpha]}\|_{L^2 \to L^2} \leq C h^\beta,
\end{equation}
where  $\mathcal{F}_h$ is the (isotropic) semiclassical Fourier transform
\begin{equation}\label{eq:semi-F}
\mathcal{F}_h f(\xi, \eta) = h^{-1} \int_{\mathbb{R}^2} e^{-2\pi i (x\xi/h + y\eta/h)} f(x, y) \, dx \, dy.
\end{equation}
\end{theorem}

As a corollary of Theorem \ref{thm:cfup-dual} and Corollary \ref{cor:cfup-ani}, we also obtain the following reflected version of the anisotropic continuous FUP for the anisotropic semiclassical Fourier transform.
\begin{theorem}[CFUP involving $\mathcal{F}_h^{\operatorname{ani}}$]\label{thm:cfup-bm-ani}
Suppose that $\mathcal A$ has no full row and $\mathcal B$ has no full column. Then there exist $\beta>0$ and $C>0$, independent of $h\in(0,1)$, such that
\[
\|\mathds{1}_{\mathbf{Y}+[-h,h]\times[-h^\alpha,h^\alpha]}\mathcal{F}_h^{\operatorname{ani}} \mathds{1}_{\operatorname{Ref}(\mathbf{X}+[-h,h]\times[-h^\alpha,h^\alpha])}\| \leq C h^\beta,
\]
  where $\operatorname{Ref}$ denotes the reflection $(x,y) \in\mathbb{R}^2 \mapsto (y,x)\in \mathbb{R}^2$, and $\mathcal{F}_h^{\operatorname{ani}}$ is the anisotropic semiclassical Fourier transform
\begin{equation}\label{eq:anis-semi-F}
 \mathcal{F}_h^{\operatorname{ani}}f(\xi,\eta) = 
    \frac{1}{h^{(1+\alpha)/2}} \int_{\mathbb{R}^2} e^{-2\pi i(x\xi/h + y\eta/h^{\alpha})} f(x,y) \, dx \, dy.
\end{equation}
\end{theorem}

\subsection{Background and related works}

Earlier studies of essential spectral gap for open quantum chaotic systems usually assume the pressure condition,
which means the trapped set is ``small" in some sense, see e.g. Nonnenmacher-Zworski \cite{nonnenmacher2009quantum}. The first breakthrough was given by Dyatlov–Zahl \cite{dyatlov2016spectral} who first proved positive essential spectral gap for some specific convex cocompact hyperbolic surfaces for which the pressure conditions fails. They reduced the spectral gap problem for convex cocompact hyperbolic manifolds to a \emph{fractal uncertainty principle} (FUP):  a nonzero function cannot be supported on one fractal set and have its Fourier transform supported on another fractal set.  A subsequent milestone paper by Bourgain–Dyatlov \cite{bourgain2018spectral} proved the 
first and now-famous continuous FUP in one dimension, which was used to obtain the spectral gap for all convex cocompact hyperbolic surfaces. This continuous FUP has found another striking application in closed quantum chaotic systems, see, e.g., \cite{MR3849286}, \cite{MR4374954} and \cite{athreya2025semiclassical}.

The high-dimensional theory of FUP is more challenging. Han–Schlag \cite{MR4085124} 
proved a continuous FUP where the physical set is $\delta$-regular and the Fourier support is a special 
Cartesian product of regular sets, using damping functions. Cladek–Tao \cite{CladekTao} proved an FUP for 
regular sets in odd dimensions with $\delta$ near $d/2$, using additive energy estimates. 
Backus–Leng–Tao \cite{MR4930594} proved a high-dimensional FUP via Dolgopyat’s method for regular sets with $\delta < d/2$, 
avoiding certain orthogonal cases; this extends the one-dimensional result of \cite{dyatlov2018dolgopyat}. 
Recently, Cohen \cite{MR4927737} established a general FUP in high dimensions for sets exhibiting porosity along lines, 
using a quantitative higher-dimensional Beurling–Malliavin theorem. 

Open quantum maps are discrete-time models of open quantum chaotic systems, which are attractive due to their simplicity in both theoretical and numerical aspects (see Nonnenmacher's review paper \cite{nonnenmacher2011spectral} for mathematical results in open quantum chaos, 
and Novaes \cite{novaes2013resonances} for the physics literature on open quantum maps). 
Nonnenmacher--Zworski \cite{nonnenmacher2007distribution} first studied the essential spectral gap problem for Walsh-quantized open baker's maps. 
Dyatlov--Jin \cite{dyatlov2017resonances} gave the first systematic theoretical analysis of Fourier-quantized open baker's map using the FUP. In particular, they established the discrete FUP in 1D and proved a positive quantitative essential spectral gap for all parameters. For further studies of 1D discrete FUP, see \cite{MR4779150}, \cite{han2024fractal}, \cite{lai2025fractal}, 
and \cite{kangabire2024bounds}. Dyatlov--Jin \cite{dyatlov2017resonances} also established the fractal Weyl upper bound for the number of resonances, 
which was improved further by Li \cite{li2023weyl} and Cunningham \cite{cunningham2026improved}. 
Recently, Cohen \cite{Cohen} studied the open quantum baker's map in higher dimensions. In particular, he established that the isotropic FUP for two-dimensional discrete Cantor sets holds if and only if there is no pair of orthogonal lines in the limiting Cantor sets.

\subsection{Structure of the paper}

Section~\ref{sec:spectralgap} is devoted to the proof of Theorem~\ref{thm:spectralgap}, establishing the existence of a spectral gap $\beta>0$ for the anisotropic quantum open baker's map. 
We use propagation of singularities to reduce the spectral gap problem to an anisotropic discrete FUP, which is the main Theorem~\ref{thm:DFUP}. 
Section~\ref{sec:dfup} is devoted to the anisotropic discrete FUP under the condition that the two alphabets do not simultaneously have a full row in one alphabet and a full column in the other. 
A key role is played by an anisotropic version of a line support lemma derived from a number-theoretic result counting cyclotomic points on algebraic curves, due to Ruppert and Beukers–Smyth. 
We prove Theorem~\ref{thm:DFUP} by combining the line support lemma with the key observation that, due to its anisotropy, the limiting Cantor set $\mathbf X$ either contains no line (implying a separation from lines) or contains only horizontal or vertical lines.

In the final Section~\ref{sec:cfup}, we focus on the anisotropic continuous FUPs. 
The first subsection~\ref{subsec:implication} introduces a continuous FUP where both the fractal sets and the Fourier transform are anisotropic; if such FUP holds, it implies Theorem~\ref{thm:DFUP} (see Theorem~\ref{thm:C-D}). 
The key tool for this implication is that the discrete FUP can be controlled by the standard Fourier transform on $\mathbb R^2$ for translated separated sets, with arbitrarily small errors. It is unfortunate that 
we can not prove such continuous FUP, and we do not 
even know whether this FUP is true in general.
Section \ref{subsec:cfup-isoFourierTransform} provides the proof of a continuous FUP for anisotropic fractal sets but using the (isotropic) semiclassical Fourier transform. 
We introduce a novel anisotropic porosity (Definition~\ref{def:porous}), use the Weyl quantization and a metaplectic operator to perform normalization, and apply symbol calculus associated to a coisotropic space to prove the almost-orthogonality of piecewise cutoff operators. 
The theorem is then proved by combining these ingredients with the porosity of the projection set and a one-dimensional FUP. 
This follows the work of \cite{dyatlov2024semiclassical}, which considers semiclassical measures for quantum cat maps in high dimensions. 
In Section \ref{subsec:rescaled}, we prove a rescaled FUP which implies a reflected FUP involving anisotropic semiclassical Fourier transform and Bedford-McMullen carpets. 

\section{Reduction principle for spectral gap}\label{sec:spectralgap}
We aim to prove Theorem~\ref{thm:spectralgap}, which establishes the spectral gap for anisotropic quantum open baker's maps. In addition to the discrete FUP (to be proved in Section~\ref{sec:dfup}), we require the propagation of singularities for the quantum maps. We first introduce some notation.
\begin{definition} 
For two points $\mathbf x=(x_1, x_2),\, \mathbf x'=(x_1', x_2')\in \mathbb T^2$ and two sets $V,W\subset \mathbb T^2$,  we introduce the \emph{$\alpha$-distance on $\mathbb{T}^2$}: 
\[
  d_{\alpha,\mathbb{T}^2}(\mathbf x,\mathbf x') := \max\left(d_{\mathbb{T}}(x_1,x_1')^{\alpha}, d_{\mathbb{T}}(x_2,x_2')\right),  \qquad d_{\alpha,\mathbb{T}^2}(V,W):=\inf_{x\in V,\, y\in W} d_{\alpha,\mathbb{T}^2}(x,y),
\] 
where $d_{\mathbb{T}}$ is the metric on $\mathbb T$ defined by
\[
d_{\mathbb{T}}(x,x'):=\min_{a\in \mathbb{Z}} |x-x'-a|, \quad x,x'\in \mathbb T.
\]
\end{definition}

Note that $d_{\alpha,\mathbb{T}^2}$ is a metric on $\mathbb{T}^2$ and induces the same topology as the standard metric on $\mathbb{T}^2$. 
One can easily verify the following basic property.
\begin{lemma}\label{lem:distance_lowerbound}
For every pair of points\, $\mathbf x=(x_1, x_2),\, \mathbf x'=(x_1', x_2')\in \mathbb T^2$, we have
\begin{equation}\label{eq:dis-basic}
    d_{\alpha,\mathbb{T}^2}(\mathbf x,\mathbf x') = d_{\alpha,\mathbb{T}^2}(\mathbf x-\mathbf x', 0) \geq \min\left(\max(|x_1-x_1'|^{\alpha},|x_2-x_2'|),\, d_{\alpha,\mathbb{T}^2}(\mathbf x', 0_L)\right),
\end{equation}
where $0_L:=\{0\}\times [0,1]\cup [0,1]\times \{0\}$ is the lower boundary of $\mathbb T^2$. 
\end{lemma}

\begin{proof}
Since 
\[
  d_{\alpha,\mathbb{T}^2}(\mathbf x,\mathbf x') := \max\left(d_{\mathbb{T}}(x_1,x_1')^{\alpha}, d_{\mathbb{T}}(x_2,x_2')\right), \]
  and 
  \begin{align*}
  & d_{\mathbb{T}}(x_1,x_1')^{\alpha}\ge \min(|x_1-x_1'|^\alpha, d_{\mathbb T}(x_1',0)^\alpha)\ge \min(|x_1-x_1'|^\alpha, d_{\alpha,\mathbb{T}^2}(\mathbf x', 0_L)), \\[4pt]
 & d_{\mathbb{T}}(x_2,x_2')\ge \min(|x_2-x_2'|, d_{\mathbb T}(x_2',0))\ge \min(|x_2-x_2'|, d_{\alpha,\mathbb{T}^2}(\mathbf x', 0_L))\end{align*} 
 the desired inequality follows by checking case to case. 
\end{proof}

\begin{definition}\label{def:Phi}
For a base $\mathbf M = (M_1, M_2)$ and an alphabet $\mathcal A$, define the \emph{expanding map} as
\begin{align*}
    \Phi: \ \bigcup_{\mathbf{a}=(a_1,a_2)\in \mathcal{A}} \left(\frac{a_1}{M_1},\frac{a_1+1}{M_1}\right)
    \times \left(\frac{a_2}{M_2},\frac{a_2+1}{M_2}\right)
    &\longrightarrow (0,1)^2 \nonumber\\[3pt]
    (x_1,x_2)\in\left(\frac{a_1}{M_1},\frac{a_1+1}{M_1}\right)
    \times \left(\frac{a_2}{M_2},\frac{a_2+1}{M_2}\right)
    &\longmapsto (M_1x_1-a_1,M_2x_2-a_2).
\end{align*}
\end{definition}
The expanding map satisfies the following distance inequality, 
which is analogous to Lemma 2.1 in \cite{dyatlov2017resonances}.
\begin{lemma} 
For $\mathbf x \in \mathbb T^2$ and $\mathbf y$ in the domain of $\Phi$,
\begin{equation}\label{distance-xPhiy}    d_{\alpha,\mathbb{T}^2}(\mathbf x,\Phi(\mathbf y))\geq \min\bigl(M_2\, d_{\alpha,\mathbb{T}^2}(\Phi^{-1}(\mathbf x),\mathbf y),\,d_{\alpha,\mathbb{T}^2}(\Phi(\mathbf y),0_L)\bigr)
\end{equation}
where $0_L$ is the lower boundary of $\mathbb T^2$ as in Lemma \ref{lem:distance_lowerbound}.
\end{lemma}
\begin{proof}
    We write $\mathbf{y}=(y_1,y_2)$ and $\Phi(\mathbf{y})=\mathbf{w}=(w_1,w_2)$. 
There exists $\mathbf a=(a_1,a_2)\in \mathcal{A}$ 
    such that 
    \[y_1 \in \left(\frac{a_1}{M_1},\frac{a_1+1}{M_1}\right) \quad \text{and}\quad y_2 \in \left(\frac{a_2}{M_2},\frac{a_2+1}{M_2}\right).\]
    Choose the unique $\mathbf{z}=(z_1,z_2)\in \Phi^{-1}(\mathbf{x})$ such that 
\[        z_1\in \left(\frac{a_1}{M_1},\frac{a_1+1}{M_1}\right) \quad \text{and}\quad  z_2\in \left(\frac{a_2}{M_2},\frac{a_2+1}{M_2}\right).
\]
Then 
    \[ \frac{1}{M_2}|x_1-w_1|^\alpha=|z_1-y_1|^\alpha<\frac{1}{M_2}\leq \frac{1}{2},\qquad
            \frac{1}{M_2}|x_2-w_2|=|z_2-y_2|<\frac{1}{M_2}\leq \frac{1}{2}.
    \]
 By Lemma \ref{lem:distance_lowerbound},  
 \begin{align*}d_{\alpha,\mathbb{T}^2}(\mathbf x,\Phi(\mathbf y)) 
 & \geq \min(M_2\, d_{\alpha,\mathbb{T}^2}(\mathbf{z},\mathbf y),\, d_{\alpha,\mathbb{T}^2}(\Phi(\mathbf y), 0_L))\\
 & \geq \min\bigl(M_2\, d_{\alpha,\mathbb{T}^2}(\Phi^{-1}(\mathbf x),\mathbf y),\,d_{\alpha,\mathbb{T}^2}(\Phi(\mathbf y),0_L)\bigr).\end{align*}
\end{proof}

Finally, we can write the quantum baker's map in \eqref{eq:BN} explicitly as follows.
\begin{lemma}\label{lem:BaN}  
For $\mathbf a\in \mathcal A$, the $\mathbf a$-piece quantum open baker's map given in \eqref{eq:BN} is 
\begin{multline*}
{B_{\mathbf{N}}^{\mathbf{a}}u(\mathbf{j})}={ \frac{\sqrt{M_1M_2}}{N_1N_2}\,}
       \sum_{\mathbf{m},\mathbf l\in\mathbb Z_{{\mathbf N}/{\mathbf M}}}
        \exp\left[2\pi i\, \left(\mathbf m\cdot \frac{\mathbf j-\mathbf l\mathbf M}{\mathbf N}+\mathbf a \cdot \frac{\mathbf j}{\mathbf M}\right)\right] \\ \cdot\,
\chi\left(\frac{\mathbf l\mathbf M}{\mathbf N}\right)\chi\left(\frac{\mathbf m\mathbf M}{\mathbf N}\right)  u\left(\mathbf{l}+\mathbf{a}\,\frac{\mathbf{N}}{\mathbf{M}}\right).
\end{multline*}
\end{lemma}

\subsection{Propagation of singularities}
We now establish the propagation of semiclassical singularities in position and frequency space under quantum evolution. This can be seen as a weak Egorov's theorem and guarantees the localization without any smoothness condition on the observables. For $\varphi: \mathbb T^2\to \mathbb R$, define its \emph{discretized function} $\varphi_{\mathbf N}\in \ell^2_\mathbf N$ and \emph{discretized Fourier multiplier} 
$\varphi_\mathbf N^\mathcal F: \ell^2\to \ell^2$
such that 
\[\varphi_{\mathbf N}(\mathbf j):=\varphi(\mathbf j/\mathbf N), \qquad  \varphi_\mathbf N^\mathcal F:=\mathcal F_\mathbf N^*\,\varphi_\mathbf N\,\mathcal F_\mathbf N,\] 
where $\varphi_\mathbf N$ is interpreted as a multiplication operator in the second formula.

\begin{proposition}[One-time propagation]\label{prop:one-time-propagation}
    Assume that $\varphi, \psi : \mathbb T^2 \to [0, 1]$ satisfy, for some $c > 0$ and 
    $0 \le \rho < 1$,
\begin{equation}\label{eq:condition-propagation}
    d_{\alpha,\mathbb{T}^2}(\Phi(\operatorname{supp}\psi \cap \Phi^{-1}(\operatorname{supp}\chi)), \operatorname{supp}\varphi) \ge c\,N_2^{-\rho}.
\end{equation}
Then
\begin{align}
    \|\varphi_{\mathbf{N}}\, B_{\mathbf{N}}\, \psi_{\mathbf{N}}\|_{{\ell_\mathbf N^2} \to {\ell_\mathbf N^2}} &= 
    \mathcal{O}_{c,\rho,\alpha,\chi}(N_2^{-\infty}), 
    \label{eq:Prop-phys} \\
    \|\psi_{\mathbf{N}}^{\mathcal{F}}\, B_{\mathbf{N}} \,\varphi_{\mathbf{N}}^{\mathcal{F}}\|_{{\ell_\mathbf N^2} \to {\ell_\mathbf N^2}} &= 
    \mathcal{O}_{c,\rho,\alpha,\chi}(N_2^{-\infty}),
    \label{eq:Prop-freq}
\end{align}
where $\mathcal{O}_{c,\rho,\alpha,\chi}(N^{-\infty})$ denotes a quantity that decays faster than any polynomial in $N$, and the implicit constants depend only on $c$, $\rho$, $\alpha$ and $\chi$. In particular, \eqref{eq:Prop-phys} and \eqref{eq:Prop-freq} hold when
\begin{equation}\label{eq:Distance condition psiphi for propagation of singularities}
    d_{\alpha,\mathbb{T}^2}(\operatorname{supp}\psi, \Phi^{-1}(\operatorname{supp}\varphi)) \ge cN_2^{-\rho}.
\end{equation}
\end{proposition}

\begin{proof}
By Lemma \ref{lem:BaN} and direct computation, for all $\mathbf j\in \mathbb Z_{\mathbf N}$ we have
    \[ (\varphi_{\mathbf{N}} B_{\mathbf{N}} \psi_{\mathbf{N}}u)(\mathbf{j})=\sum_{\mathbf{a}\in \mathcal{A}}\,\sum_{\mathbf{l}\in \mathbb Z_{\mathbf N/\mathbf M}}         
 A_{\mathbf{j}\mathbf{l}}^{\mathbf{a}}\ u\left(\mathbf{l}+\mathbf{a}\,\frac{\mathbf{N}}{\mathbf{M}}\right), \]
 where 
  \begin{align*}   
            A_{\mathbf{j}\mathbf{l}}^{\mathbf{a}}&:=\frac{\sqrt{M_1M_2}}{N_1N_2}\, \varphi\left(\frac{\mathbf{j}}{\mathbf{N}}\right)\,
            \exp\left(2\pi i\,\mathbf a\cdot\frac{\mathbf j}{\mathbf M}
            \right)\,\chi\left(\frac{\mathbf{l}\mathbf{M}}{\mathbf{N}}\right)\,\psi\left(\frac{\mathbf{l}}{\mathbf{N}}+
            \frac{\mathbf{a}}{\mathbf{M}}\right)\,\tilde{A}_{\mathbf{j}\mathbf{l}}, \\[6pt]
            \tilde{A}_{\mathbf{j}\mathbf{l}}&:=\sum_{\mathbf{m}\in \mathbb Z_{\mathbf N/\mathbf M}}
            \exp\left(2\pi i\,\mathbf m\cdot \frac{\mathbf j-\mathbf l\mathbf M}{\mathbf N}  \right)\,
           \chi\left(\frac{\mathbf{m}\mathbf{M}}{\mathbf{N}}\right).
        \end{align*}
Note that, by the Definition \ref{def:Phi} of $\Phi$, for all $\mathbf{j}\in {\mathbb Z_{\mathbf N}}$ and  
$\mathbf{l}\in {\mathbb Z_{\mathbf N/\mathbf M}}$, the non-vanishing condition $A_{\mathbf{j}\mathbf{l}}^{\mathbf{a}}\neq 0$ forces
\begin{equation}\label{eq:supp}\frac{\mathbf j}{\mathbf N}\in \operatorname{supp} \varphi,
\quad \frac{\mathbf l}{\mathbf N}+\frac{\mathbf a}{\mathbf M} \in \operatorname{supp} \psi,
\quad \frac{\mathbf l\mathbf M}{\mathbf N}=\Phi\left(\frac{\mathbf l}{\mathbf N}+\frac{\mathbf a}{\mathbf M}\right) \in \operatorname{supp} \chi.
\end{equation}
Set $\mathbf{b}:=\mathbf{j}-\mathbf{l}\mathbf{M}$. Using the fact 
that $\operatorname{supp} \chi\subset (0,1)^2$, we can extend 
the summation in $\tilde{A}_{\mathbf{j}\mathbf{l}}$ to all $\mathbf{m}\in \mathbb{Z}^2$:
\[
    \tilde{A}_{\mathbf{j}\mathbf{l}}= \tilde{A}_{\mathbf b}=\sum_{\mathbf{m}\in \mathbb{Z}^2}
    \exp\left(2\pi i\,  \frac{\mathbf b\mathbf m}{\mathbf N}\right)\chi\left(\frac{\mathbf{m}\mathbf{M}}{\mathbf{N}}\right).
\]
To apply Lemma \ref{lem:nonstationary}, we need to verify the condition
\[d_{\alpha,\mathbb{T}^2} \left(\frac{\mathbf b}{\mathbf N},0\right) 
    \ge cN_2^{-\rho}.\]
Assume the support-separation condition \eqref{eq:condition-propagation} holds. By  \eqref{eq:dis-basic},  \eqref{eq:condition-propagation} and \eqref{eq:supp},
for any $\mathbf{j},\mathbf{l}$ with $A_{\mathbf{j}\mathbf{l}}^{\mathbf{a}}\neq 0$, we have 
\[
d_{\alpha,\mathbb{T}^2}\left(\frac{\mathbf b}{\mathbf N},0\right)=    d_{\alpha,\mathbb{T}^2}\left(\frac{\mathbf j}{\mathbf N},\frac{\mathbf l\mathbf M}{\mathbf N}\right)
    \geq  d_{\alpha,\mathbb{T}^2}\left(\Phi(\operatorname{supp}\psi \cap \Phi^{-1}(\operatorname{supp}\chi)), \operatorname{supp}\varphi\right) \ge c\,N_2^{-\rho}.
\]
Hence, by Lemma \ref{lem:nonstationary}, we obtain 
\[
 \sup_{\mathbf j, \mathbf l, \mathbf a}|A_{\mathbf{j}\mathbf{l}}^{\mathbf a}|\le \sup_{\mathbf j, \mathbf l} |\tilde{A}_{\mathbf{j}\mathbf{l}}|=\mathcal{O}_{c,\rho,\alpha,\chi}(N_2^{-\infty}),\] and thus 
\eqref{eq:Prop-phys} holds.
By conjugation and duality, \eqref{eq:Prop-freq} follows from \eqref{eq:Prop-phys}; see the proof in \cite[Proposition 2.3]{dyatlov2017resonances}.

It remains to prove \eqref{eq:Prop-phys} under the assumption \eqref{eq:Distance condition psiphi for propagation of singularities}.
We next show that \eqref{eq:Distance condition psiphi for propagation of singularities} implies 
\eqref{eq:condition-propagation} for some new $c>0$. Actually, 
for $\mathbf x\in  \operatorname{supp} \varphi$ and $\mathbf y\in \operatorname{supp} \psi\cap \Phi^{-1}(\operatorname{supp}\chi)$, 
since $\chi\in C_c^\infty((0,1)^2)$, there exists $c_1>0$ depending only on 
$\operatorname{supp} \chi$ such that
\begin{equation}\label{eq:d-chi-0}
    d_{\alpha,\mathbb{T}^2}(\Phi(\mathbf y),0_L)\geq c_1>0,
\end{equation}
and therefore, by \eqref{distance-xPhiy}, we have
\[
    \begin{aligned}
        d_{\alpha,\mathbb{T}^2}(\mathbf x,\Phi(\mathbf y))&\geq 
    \min(d_{\alpha,\mathbb{T}^2}(\Phi(\mathbf y),0_L),
    M_2d_{\alpha,\mathbb{T}^2}(\Phi^{-1}(\mathbf x),\mathbf y))\\
    &\geq \min(c_1, M_2 cN_2^{-\rho})\geq c'N_2^{-\rho}
    \end{aligned}
\]
for some new $c'>0$. This completes the proof of the proposition.
\end{proof}

\begin{lemma}[Nonstationary phase]\label{lem:nonstationary}
If $d_{\alpha,\mathbb{T}^2}\left(\frac{\mathbf b}{\mathbf N},0\right) \ge cN_2^{-\rho}$ for some constants $c>0$ and $0\le \rho<1$, then
\[\tilde{A}_{\mathbf b}=\sum_{\mathbf{m}\in \mathbb{Z}^2}
\exp\left(2\pi i\, \frac{\mathbf b\mathbf m}{\mathbf N}\right)\chi\left(\frac{\mathbf{m}\mathbf{M}}{\mathbf{N}}\right)=\mathcal{O}_{c,\rho,\alpha,\chi}(N_2^{-\infty}).
\]
\end{lemma}

\begin{proof}
Applying the Poisson summation formula to the summand considered as a function of $\mathbf m$,
\[
\exp\left(2\pi i\, \frac{\mathbf b\mathbf m}{\mathbf N}\right)\chi\left(\frac{\mathbf{m}\mathbf{M}}{\mathbf{N}}\right),
\]
we obtain
\[
\tilde{A}_{\mathbf b}=\frac{N_1N_2}{M_1M_2}\,\sum_{\mathbf{l}\in \mathbf{Z}^2}\widehat{\chi}\left(\frac{\mathbf N\mathbf l-\mathbf b}{\mathbf M}\right).
\]
By the Schwartz decay of $\widehat\chi$ and the relation $N_2=N_1^\alpha$, we have, for an arbitrarily large integer $L$ which may change in the proof, the estimate
\[
|\tilde{A}_{\mathbf b}| \le \frac{N_1N_2}{M_1M_2}\, \sum_{\mathbf{l}\in \mathbf{Z}^2} \left|\frac{\mathbf N}{\mathbf M}  \left(\mathbf l-\frac{\mathbf b}{\mathbf N}\right)\right|^{-L} \lesssim_{\alpha,\chi} N_2^{-L} +
\frac{N_1N_2}{M_1M_2}\,\sum_{l_1\le 1, l_2\le 1} \left|\frac{\mathbf N}{\mathbf M}  \left(\mathbf l-\frac{\mathbf b}{\mathbf N}\right)\right|^{-L}.
\]
Assume that
\[
d_\mathbb T\left(\frac{b_2}{N_2},0\right)=d_{\alpha,\mathbb{T}^2} \left(\frac{\mathbf b}{\mathbf N},0\right) 
    \ge cN_2^{-\rho}
\]
for some constants $c>0$ and $0\le \rho<1$. Then for arbitratily large $L$
\[
|\tilde{A}_{\mathbf b}| 
\lesssim N_2^{-L} +
\frac{N_1N_2}{M_1M_2}\,
(cN_2^{1-\rho}/M_2)^{-L}\lesssim_{c,\rho,\alpha,\chi} N_2^{-L},
\]
which means that
\[
\tilde{A}_{\mathbf b}=\mathcal{O}_{c,\alpha,\rho}(N_2^{-\infty}).
\]
The proof is the same for the case
\[
d_\mathbb T\left(\frac{b_1}{N_1},0\right)^\alpha=d_{\alpha,\mathbb{T}^2} \left(\frac{\mathbf b}{\mathbf N},0\right) 
    \ge cN_2^{-\rho}.
\]
\end{proof}

Now we turn to the case where we iterate the quantum map up to (almost) twice the \emph{Ehrenfest time} $k$. The classical map has expansion rates $\mathbf M = (M_1, M_2)$, and the semiclassical small parameter is $h = (2\pi \mathbf M^k)^{-1} = ((2\pi M_1^k)^{-1}, (2\pi M_2^k)^{-1})$. Therefore, propagation up to twice the Ehrenfest time corresponds to raising the quantum map $B_{\mathbf N}$ to the power $k$.

\begin{proposition}[Long-times propagation]\label{prop:Propagation of singularities for long times}
Assume that $\varphi, \psi : \mathbb T^2 \to [0, 1]$ satisfy that there exist $c > 0$ and $0 \le \rho < 1$ such that for every integer $\tilde{k} \in [1, k]$,
\begin{equation}\label{eq:2.18 in Dyatlov jin}
    d_{\alpha,\mathbb{T}^2}(\operatorname{supp}\psi, \Phi^{-\tilde{k}}
    (\operatorname{supp}\varphi)) \ge cN_2^{-\rho}.
\end{equation}
Then
\begin{align}
    \|\varphi_{\mathbf{N}}(B_\mathbf{N})^{\tilde{k}}\psi_{\mathbf{N}}\|_{{\ell_\mathbf N^2} \to {\ell_\mathbf N^2}} &= \mathcal{O}(N_2^{-\infty}),
     \label{eq:2.19 in Dyatlov jin} \\
    \|\psi_{\mathbf{N}}^{\mathcal{F}}(B_{\mathbf{N}})^{\tilde{k}}\varphi_{\mathbf{N}}^{\mathcal{F}}\|_{{\ell_\mathbf N^2} \to {\ell_\mathbf N^2}} &= 
    \mathcal{O}(N_2^{-\infty}),
    \label{eq:2.20 in Dyatlov jin}
\end{align}
where the implicit constants in $\mathcal{O}(N_2^{-\infty})$ depend only on $c$, $\rho$, $\alpha$ and $\chi$.
\end{proposition}

\begin{proof}
\emph{Step 1.} It suffices to construct a sequence of functions (see \cite[Figure 6]{dyatlov2017resonances} for a one-dimensional picture; the main difference here is that we use the new $\alpha$-distance)
\[
    \varphi^{(j)} : \mathbb T^2 \to [0, 1], \quad \psi^{(j)} := 1 - \varphi^{(j)}, \quad j = 0, \dots, \tilde{k},
\]
such that 
\begin{equation}\label{eq:2.22 in Dyatlov jin} 
 \varphi^{(0)}\varphi = \varphi, \quad \psi^{(0)}\varphi = 0, \quad %
    \psi^{(\tilde{k})}\psi = \psi,  \quad \varphi^{(\tilde{k})}\psi = 0,\end{equation} 
and, for some $c' > 0$ depending only on $c$ and $\chi$,
 \begin{equation} \label{eq:2.23 in Dyatlov jin} 
  d_{\alpha,\mathbb{T}^2}(\Phi(\operatorname{supp}\psi^{(j+1)} \cap \Phi^{-1}(\operatorname{supp}\chi)), \operatorname{supp}\varphi^{(j)}) 
    \ge c'N^{-\rho}, 
    \quad j = 0, \dots, \tilde{k}-1.
\end{equation}
Actually, if we have this sequence of functions, then by inserting the identity $1 = \varphi_{\mathbf N}^{(j)} + \psi_{\mathbf N}^{(j)}$ after each $j$-th factor $B_{\mathbf N}$ and using \eqref{eq:2.22 in Dyatlov jin}, we obtain
\begin{gather*}
    \varphi_{\mathbf{N}}(B_\mathbf{N})^{\tilde{k}}\psi_{\mathbf{N}} = \sum_{j=0}^{\tilde{k}-1} A_j'(\varphi_{\mathbf{N}}^{(j)}
    B_\mathbf{N}\psi_\mathbf{N}^{(j+1)})A_j'', \\
    A_j' = \varphi_\mathbf{N}(B_\mathbf{N})^j, \quad A_j'' = (B_\mathbf{N}\psi_\mathbf{N}^{(j+2)})\dots(B_\mathbf{N}\psi_\mathbf{N}^{(\tilde{k})})
    \psi_\mathbf{N}.
\end{gather*}
Clearly,
\[
    \|A_j'\|_{{\ell_\mathbf N^2} \to {\ell_\mathbf N^2}}, \|A_j''\|_{{\ell_\mathbf N^2} \to 
    {\ell_\mathbf N^2}} \le 1.
\]
Applying the one-time propagation result \eqref{eq:Prop-phys} together with \eqref{eq:2.23 in Dyatlov jin}, we get
\[
    \|\varphi_{\mathbf{N}}^{(j)} B_N \psi_\mathbf{N}^{(j+1)}\|_{{\ell_\mathbf N^2} \to {\ell_\mathbf N^2}} = \mathcal{O}(N_2^{-\infty}).
\]
This concludes the proof of \eqref{eq:2.19 in Dyatlov jin} since $\tilde{k}\leq k=\mathcal{O}(\log N_2)$. 
The other estimate \eqref{eq:2.20 in Dyatlov jin} follows from \eqref{eq:Prop-freq} 
by the same argument.

\emph{Step 2.} We now construct the sequence of functions $(\varphi^{(j)})_{j=0}^{\tilde k}$. 
Fix $c' > 0$, depending only on $c$ and $\chi$, to be chosen later in \eqref{eq:2.26 in Dyatlov jin} and \eqref{eq:2.27 in Dyatlov jin} below. Define
\[
    \varphi^{(0)}(\mathbf x) = 
    \begin{cases} 
        1, & \mathbf x \in \operatorname{supp} \varphi, \\ 
        0, & \text{otherwise}. 
    \end{cases}
\]
Then $\psi^{(0)}\varphi = 0$. 
For $j = 1, \dots, \tilde{k}$, we define inductively
\[
    \varphi^{(j)}(\mathbf y) = 
    \begin{cases} 
        1, & \mathbf y \in \Phi^{-1}(\operatorname{supp} \chi) \quad \text{and} \quad d_{\alpha,\mathbb{T}^2}(\Phi(\mathbf y), 
        \operatorname{supp} \varphi^{(j-1)}) \le c'N_2^{-\rho}, \\ 
        0, & \text{otherwise}. 
    \end{cases}
\]
Then \eqref{eq:2.23 in Dyatlov jin} holds, so it remains to prove $\psi^{(\tilde{k})}\psi = \psi$ in \eqref{eq:2.22 in Dyatlov jin}. The latter follows from the fact that
\begin{equation}\label{eq:2.24 in Dyatlov jin}
    \operatorname{supp} \varphi^{(\tilde{k})} \cap \operatorname{supp} \psi = \emptyset.
\end{equation}
To show \eqref{eq:2.24 in Dyatlov jin}, we note that for any point $\mathbf{p}_{\tilde{k}} \in \operatorname{supp} \varphi^{(\tilde{k})} \subset \Phi^{-1}(\operatorname{supp} \chi)$, there exists a sequence of points $\mathbf{p}_0, \dots, \mathbf{p}_{\tilde{k}} \in \mathbb T^2$ such that
\begin{equation}\label{eq:pjsequence}
    \mathbf{p}_0 \in \operatorname{supp} \varphi; \quad \mathbf{p}_j \in 
    \Phi^{-1}(\operatorname{supp} \chi), \quad 
    d_{\alpha,\mathbb{T}^2}(\Phi(\mathbf{p}_j), \mathbf{p}_{j-1}) \le c'N_2^{-\rho}, 
    \quad j = 1, \dots, \tilde{k}.
\end{equation}
By \eqref{eq:2.18 in Dyatlov jin} it then suffices to prove that
\begin{equation}\label{eq:2.25 in Dyatlov jin}
    d_{\alpha,\mathbb{T}^2}(\mathbf{p}_{\tilde{k}}, \Phi^{-\tilde{k}}(\mathbf{p}_0)) 
    < cN_2^{-\rho}.
\end{equation}
Since $\chi\in C_c^\infty$, there exists $c_1>0$ such that  
\[
    c_1 = d_{\alpha,\mathbb{T}^2}(\operatorname{supp} \chi, 0_L) > 0.
\]
This has been used in the proof of Proposition \ref{prop:one-time-propagation} (see \eqref{eq:d-chi-0}). 
In view of \eqref{distance-xPhiy} and \eqref{eq:pjsequence} we have for each $j = 1, \dots, \tilde{k}$,
\begin{align}
    \min \{c_1, M_2\, d_{\alpha,\mathbb{T}^2}
    (\mathbf{p}_j, \Phi^{-j}(\mathbf{p}_0))\} 
    \le &\min 
    \{d_{\alpha,\mathbb{T}^2}(\Phi(\mathbf{p}_j), 0_L), 
    M_2\,  d_{\alpha,\mathbb{T}^2}(\mathbf{p}_j, \Phi^{-j}(\mathbf{p}_0))\} \nonumber\\
    \le &d_{\alpha,\mathbb{T}^2}(\mathbf{p}_{j-1}, \Phi^{-(j-1)}(\mathbf{p}_0)) + c'N_2^{-\rho}. \label{eq:induction-onj}
\end{align}
By \eqref{eq:induction-onj} and induction on $j$, we see that if $c'$ is small enough such that
\begin{equation}\label{eq:2.26 in Dyatlov jin}
    \frac{c'M_2}{M_2-1} \ll  c_1,
\end{equation}
then for all $j = 0, \dots, \tilde{k}$ we have
\[
    d_{\alpha,\mathbb{T}^2}(\mathbf{p}_j, \Phi^{-j}(\mathbf{p}_0)) 
    \le c'N_2^{-\rho} \cdot \frac{1 - M_2^{-j}}{M_2 - 1}
      \le N_2^{-\rho} \cdot \frac{c'}{M_2 - 1}.
\]
This will imply \eqref{eq:2.25 in Dyatlov jin} and complete the proof by choosing $c'$ sufficiently small so that
\begin{equation}\label{eq:2.27 in Dyatlov jin}
    \frac{c'}{M_2-1} \ll c.
\end{equation}
To see this, note that induction also shows that the minimum in \eqref{eq:induction-onj} is attained at the second term, i.e.,
\[
\min\{c_1, M_2\, d_{\alpha,\mathbb{T}^2}(\mathbf{p}_j, \Phi^{-j}(\mathbf{p}_0))\} = M_2\, d_{\alpha,\mathbb{T}^2}(\mathbf{p}_j, \Phi^{-j}(\mathbf{p}_0)).
\]
Otherwise, we would have
\[
c_1 \le c'N_2^{-\rho} \cdot \frac{M_2(1 - M_2^{-j})}{M_2 - 1} \ll \frac{c'M_2}{M_2 - 1} \ll c_1,
\]
a contradiction.

\end{proof}

\subsection{Reduction to FUP}
Fix an integer $k\gg 1$ and a parameter $\rho\in(0,1)$. Let
  \[  \tilde{k} := \lceil \rho k \rceil \in \{1, \dots, k\}.\]
Define
\[
    X_\rho := \{\mathbf{p} \in [0, 1]^2 : d_{\alpha,\mathbb{T}^2}(\mathbf{p}, \Phi^{-\tilde{k}}([0, 1]^2)) \le N_2^{-\rho}\}.
\]
Using the discrete Cantor set $\mathcal{X}_k$ defined by 
\[
    \mathcal{X}_k=\{(a_0+\cdots a_kM_1^{k-1},
        b_0+\cdots b_kM_2^{k-1}):(a_j,b_j)\in\mathcal{A}\}\subset 
        {\mathbb Z_{\mathbf N}},
\]
we also define
\[ \mathcal{X}_\rho := \bigcup_{\mathbf m}\{\mathcal{X}_k + \mathbf m: \mathbf m=(m_1, m_2) \in \mathbb{Z}^2, \max(|m_1|^\alpha,|m_2|) \le 2N_2^{1-\rho}\}
    \subset {\mathbb Z_{\mathbf N}},\]
where addition $+$ is carried out in the group ${\mathbb Z_{\mathbf N}}$. Then the following implication holds:
\begin{equation}\label{eq:2.30 in Dyatlov jin} (l_1,l_2) \in \mathbb Z_{\mathbf N}, 
    \quad (\frac{l_1}{N_1},\frac{l_2}{N_2}) \in 
    X_\rho \implies (l_1,l_2) \in \mathcal{X}_\rho.
\end{equation}
Indeed, we have
\[
    \begin{aligned}
        \Phi^{-\tilde{k}}([0, 1]) &= \bigcup_{\mathbf{j} \in \mathcal{X}_{\tilde{k}}} 
        \left( \frac{j_1}{M_1^{\tilde{k}}}, 
        \frac{j_1+1}{M_1^{\tilde{k}}} \right)\times \left( \frac{j_2}{M_2^{\tilde{k}}}, 
        \frac{j_2+1}{M_2^{\tilde{k}}} \right)\\
        &\subset \bigcup_{\mathbf{j} \in \mathcal{X}_k} 
        \left( \frac{j_1 - M_1^{k-\tilde{k}}}{N_1}, \frac{j_1 + M_1^{k-\tilde{k}}}{N_1} \right)
        \times \left( \frac{j_2 - M_2^{k-\tilde{k}}}{N_2}, \frac{j_2 + M_2^{k-\tilde{k}}}{N_2} \right).
    \end{aligned}
\]
Thus if $(\frac{l_1}{N_1},\frac{l_2}{N_2}) \in X_\rho$, we know 
\[  \frac{l_i}{N_i}\in \left(\frac{j_i-M_i^{k-\tilde{k}}}{N_i}-N_i^{-\rho},\,
 \frac{j_i+M_i^{k-\tilde{k}}}{N_i}+N_i^{-\rho}\right)\qquad i=1,2,
\]
and \eqref{eq:2.30 in Dyatlov jin} follows because $M_2^{k-\tilde{k}} \le N_2^{1-\rho}$ and $M_1^{k-\tilde{k}}=
M_2^{(k-\tilde{k})/\alpha}\leq N_1^{1-\rho}$.

Taking $\varphi \equiv 1$ and $\psi := 1 - \mathds{1}_{X_\rho}$ 
in Proposition \ref{prop:Propagation of singularities for long times} and using \eqref{eq:2.30 in Dyatlov jin},
we obtain
\begin{align}\label{eq:BNXrho}
    (B_\mathbf{N})^{\tilde{k}} &= (B_\mathbf{N})^{\tilde{k}} \mathds{1}_{\mathcal X_\rho} + 
    \mathcal{O}(N_2^{-\infty})_{{\ell_\mathbf N^2} \to {\ell_\mathbf N^2}} \nonumber\\
    (B_\mathbf{N})^{\tilde{k}} &= \mathcal{F}_{\mathbf{N}}^* \mathds{1}_{\mathcal X_\rho} \mathcal{F}_{\mathbf{N}} 
    (B_\mathbf{N})^{\tilde{k}} + 
    \mathcal{O}(N_2^{-\infty})_{{\ell_\mathbf N^2} \to {\ell_\mathbf N^2}},
\end{align}
where the constants in $\mathcal{O}(N_2^{-\infty})$ depend only on $\rho, \chi$.

If the discrete FUP \eqref{eq:dfup} holds with some $\beta>0$ (for $\mathcal B=\mathcal A$),
we consider the normalized eigenfunction $u$ associated with an eigenvalue 
$\lambda\in \operatorname{Sp}(B_\mathbf{N})$. If $|\lambda|\ge M_2^{-\beta}$, then by \eqref{eq:BNXrho}
\[\|u\|_{\ell^2_{\mathbf N}}=\|(B_{\mathbf N}^{\tilde k})(\lambda^{-\tilde k}u)\|_{\ell^2_{\mathbf N}} \le M_2^\beta \lambda^{-\rho k} \|\mathds{1}_{\mathcal X_\rho} \mathcal F_{\mathbf N}^*\mathds{1}_{\mathcal X_\rho} \mathcal F_{\mathbf N} u\|_{\ell^2_{\mathbf N}}+\mathcal O(N_2^{-\infty}).\]
By translation invariance, 
\[\|\mathds{1}_{\mathcal X_\rho} \mathcal F_{\mathbf N}^*\mathds{1}_{\mathcal X_\rho}\|_{\ell^2\to \ell^2}\lesssim N_2^{4(1-\rho)} \,r_k.\]
Then the spectral radius satisfies
\[
\max\{|\lambda|:\lambda\in \operatorname{Sp}(B_\mathbf{N})\} \le \max\bigl(M_2^{-\beta}, C^{1/(\rho k)} M_2^{(2-2\rho-\beta)/\rho}\bigr) \to M_2^{-\beta}
\]
as $k\to\infty$ and $\rho\to 1$.
Hence, we have proved that the discrete FUP implies a spectral gap for the quantum map.

\begin{theorem}
If the discrete FUP \eqref{eq:dfup} holds for alphabets $\mathcal B=\mathcal A$ with some $\beta>0$, then we obtain the spectral gap estimate \eqref{eq:spectralgap} with the same $\beta$ for quantum maps $B_{\mathbf N}$ constructed from $\mathcal A$ by \eqref{eq:BN}. 
\end{theorem}

With this reduction principle, once we have Theorem \ref{thm:DFUP}, we also obtain Theorem \ref{thm:spectralgap}. In the following sections, we focus on anisotropic FUPs.

\section{Anisotropic discrete FUP}\label{sec:dfup}
We aim to prove Theorem~\ref{thm:DFUP}. 
We recall the anisotropic discrete Fourier transform defined by \eqref{eq:DFT} as
\[
\mathcal{F}_{\mathbf N}u(\xi_1,\xi_2)=\frac{1}{\sqrt{N_1N_2}}\sum_{a=0}^{N_1-1}\sum_{b=0}^{N_2-1}e^{-2\pi i(a\xi_1/N_1+b\xi_2/N_2)}u(a,b)~:~\ell^2_{\mathbf N}\to\ell^2_{\mathbf N}.
\] In this section, we also denote $\mathcal F_{\mathbf N}f$ by $\hat f$ for discrete functions.

Our main objective is to determine under which conditions on the alphabets $\mathcal{A},\mathcal{B}$ \eqref{eq:dfup} i.e. the following FUP estimate holds:
\[
\|\mathds{1}_{\mathcal{Y}_k}\mathcal{F}_{\mathbf N}\mathds{1}_{\mathcal{X}_k}\|_{{\ell^2_{\mathbf N}}\to{\ell^2_{\mathbf N}}}\leq C N_2^{-\beta}.
\]
This exhibits exponential decay in the iteration parameter $k$.

The self-similarity of the discrete Cantor sets implies the submultiplicativity of the FUP norm (\cite[Section~5.1]{Cohen}, see 
also Appendix \ref{app:submultiplicativity} for a detailed proof):
\begin{equation}\label{eq:submultiplicativity}
\|\mathds{1}_{\mathcal{Y}_{k_1+k_2}}\mathcal{F}_{\mathbf N}\mathds{1}_{\mathcal{X}_{k_1+k_2}}\|_{{\ell^2_{\mathbf N}}\to{\ell^2_{\mathbf N}}}
\le
\|\mathds{1}_{\mathcal{Y}_{k_1}}\mathcal{F}_{\mathbf N}\mathds{1}_{\mathcal{X}_{k_1}}\|_{{\ell^2_{\mathbf N}}\to{\ell^2_{\mathbf N}}}\,
\|\mathds{1}_{\mathcal{Y}_{k_2}}\mathcal{F}_{\mathbf N}\mathds{1}_{\mathcal{X}_{k_2}}\|_{{\ell^2_{\mathbf N}}\to{\ell^2_{\mathbf N}}}.
\end{equation}
 This reduces the problem to the existence of some $k_0$ such that
\[
\|\mathds{1}_{\mathcal{Y}_{k_0}}\mathcal{F}_{\mathbf N}\mathds{1}_{\mathcal{X}_{k_0}}\|_{{\ell^2_{\mathbf N}}\to{\ell^2_{\mathbf N}}}\le \lambda<1.
\]
Finally, it suffices to show that for some sufficiently large $k$ there exists no $f$ supported on $\mathcal{X}_k$ with $\operatorname{supp}\hat f\subset\mathcal{Y}_k$. Our argument is a modification of the argument in \cite{Cohen}.

\subsection{Cyclotomic points on algebraic curves} 
A classical theme in arithmetic geometry and Diophantine analysis is the study of algebraic curves that contain many points (finite or infinite) whose coordinates are roots of unity. Such points are called \emph{cyclotomic points}. The simplest examples are given by binomial equations of the form \(z^a w^b = \zeta\) or \(z^a = \zeta w^b\), where \(\zeta\) is a root of unity and \(a,b\) are coprime positive integers. Each such equation defines an algebraic curve that is isomorphic to the multiplicative group \(\mathbb{G}_m\) and contains a Zariski-dense set of cyclotomic points.

For a general polynomial \(F(z,w) \in \mathbb{C}[z,w]\), the set of cyclotomic points on the curve \(F(z,w)=0\) is typically finite. Bounding the number of such points in terms of the degree of \(F\) is a problem of combinatorial and arithmetic flavour.

The following fundamental theorem from number theory, due to Ruppert \cite[Corollary 5]{ruppert1993solving} 
and refined by Beukers and Smyth \cite[Theorem 4.1]{beukers2002cyclotomic}, provides a sharp quantitative and structural result about Lang's conjecture for polynomials with degree at most \(D\) in \(w\) and degree at most \(\lfloor D^{1/\alpha}\rfloor\) in \(z\).  See \cite[Theorem 5]{Cohen} for a modified isotropic version.

\begin{theorem}[Cyclotomic points]\label{thm:cyclotomic}
Let
 \[   F(z,w) = \sum_{0\le k\le \lfloor D^{1/\alpha}\rfloor,\;0\le l\le D} a_{kl}z^kw^l\]
be a polynomial in $\mathbb{C}[z,w]$ with degree at most $D\geq 1$ in $w$ 
and degree at most $\lfloor D^{1/\alpha}\rfloor$ in $z$. 
Then $F$ has either at most $30\lfloor D^{1+1/\alpha} \rfloor$ 
cyclotomic points, or infinitely many. In the latter case $F$ has an irreducible factor of the form
\begin{equation}\label{eq:Polynomial cutting out a line}
    z^a w^b - \zeta \quad \text{or} \quad z^a - \zeta w^b
\end{equation}
for some root of unity $\zeta$ and coprime integers $a,b$.
\end{theorem}
 
 \begin{remark}
The term {cyclotomic point} means a point $(z,w)\in \mathbb{C}^2$ such that $F(z,w)=0$ and both $z$ and $w$ are roots of unity.
The algebraic curve defined by $z^a w^b = \zeta$ (or $z^a = \zeta w^b$) is irreducible over $\mathbb{C}$ if and only if $a$ and $b$ are coprime, in which case the curve contains infinitely many cyclotomic points.
For any polynomial that does not have such an irreducible factor, the curve $F(z,w)=0$ has only finitely many cyclotomic points, and their number is bounded by $30\lfloor D^{1+1/\alpha}\rfloor$.
\end{remark}

\begin{remark}
    In the original statement of \cite[Theorem 4.1]{beukers2002cyclotomic}, the 
    upper bound on the number of cyclotomic points is given by $22$ times 
    the area of the \emph{Newton polytope} of $F$. In our case, the area of the Newton polytope of $F$ 
    is bounded by $\lfloor D^{1/\alpha}\rfloor D$.
\end{remark}

Recall that we can embed $\mathbb{T}^2$ into $\mathbb{C}^{2}$ via
\[
(x,y) \mapsto (e^{2\pi ix}, e^{2\pi iy}).
\]
The cyclotomic points in $\mathbb{C}^{2}$ associated to any polynomial are in the image of $(\mathbb{Q}/\mathbb{Z})^2$. For a polynomial $F(z,w)$, we define
\begin{align*}
Z(F) &= \{(x,y) \in \mathbb{T}^2 : F(e^{2\pi ix}, e^{2\pi iy}) = 0\}, \\
Z_{\mathbf N}(F) &= \{(x,y) \in {\mathbb Z_{\mathbf N}} : F(e^{\frac{2\pi i}{N_1}x}, e^{\frac{2\pi i}{N_2}y}) = 0\}. 
\end{align*}
If we view ${\mathbb Z_{\mathbf N}}$ as the subgroup of $\mathbb{T}^2$ via the identification 
\[
{\mathbb Z_{\mathbf N}} \cong \mathbb{T}_{\mathbf N}= \left\{\left(\frac{x}{N_1}, \frac{y}{N_2}\right) \in \mathbb{T}^2 \mid x,y \in \mathbb{Z} \right\},
\]
then $Z_{\mathbf N}(F) = Z(F) \cap \mathbb{T}_{\mathbf N}$. 

We say that a polynomial $F$ of the form \eqref{eq:Polynomial cutting out a line} \emph{cuts out a line}, because
\[
Z(F) = \{(x,y) \in \mathbb{T}^2 : ax\pm by = c\},
\]
with $a,b \ge 0$ coprime integers and $c \in \mathbb{Q}$, where $\zeta = e^{2\pi i c}$. 
This line is \emph{irreducible} since $a$ and $b$ are coprime (see Definition \ref{def:ir-line}). 


\subsection{Line support lemma}

In harmonic analysis and additive combinatorics over finite 
abelian groups, a fundamental theme is the trade-off between 
the support of a function and its Fourier support. 
In the study of uncertainty principles and applications 
such as detecting arithmetic progressions or linear structures, 
one often needs to decompose the support of a function 
along particular directions. To state the key line support lemma, 
we first introduce the notion of an {irreducible line} 
in $\mathbb Z_{\mathbf N}$. This notion is motivated by 
$Z_{\mathbf{N}}(F)$ when $F$ is a polynomial that cuts out a line.

\begin{definition}\label{def:ir-line}
    An \emph{irreducible line} $\ell$ in $\mathbb Z_{\mathbf N}$ is a non-empty set of the form    
    \begin{equation}\label{eq:line}
        \ell=\{(x,y)\in \mathbb Z_{\mathbf N} : \frac{ax}{N_1}+\frac{by}{N_2} \equiv c \pmod{1}\}
    \end{equation}
    for coprime pair $(a,b)\in \mathbb Z_{\mathbf N}$  and $c\in \mathbb{Q}/\mathbb Z$ (allowing $a=0$ or $b=0$, with the understanding that the other non-zero one is $\pm1$). 
    We will use the \emph{$\alpha$-norm} $\|\ell\|_\alpha$ of the line $\ell$ to denote the \emph{minimal} number $R$
    such that $\ell$ can be written as 
 \eqref{eq:line}   with $|a|^\alpha,|b|\le R$. 
\end{definition}

\begin{remark}
    This $\alpha$-norm characterizes the size of the dual normal vector of the irreducible lines. We note that when $\ell$ is either vertical or horizontal, $\|\ell\|_\alpha=1$. See Appendix \ref{app:irreducible-lines} for a description of the structure of \emph{lines} in $\mathbb Z_{\mathbf N}$. 
\end{remark}

Let $A\subset \mathbb Z_{\mathbf N}$. We define the (discrete  and half) ~\emph{$(R,\alpha)$-neighborhood} of $A$ in $\mathbb Z_{\mathbf N}$ by
\begin{equation}\label{eq:nb-R}
    \mathrm{N}_R^\alpha(A)=A+[0,\lfloor R^{1/\alpha} \rfloor)\times [0,R).
\end{equation}

\begin{lemma}[Line support]\label{lem:Cohen-Lemma11}
Let $f : {\mathbb Z_{\mathbf N}} \to \mathbb{C}$ be a nonzero function with $\operatorname{supp} f = S$. 
Let $R = \lfloor C_\alpha|S|^{\alpha/(1+\alpha)} \rfloor$, where $C_\alpha$ is a constant from Lemma \ref{lem:polysupp}.  
Then there exists an irreducible line $\ell$ with $\|\ell\|_\alpha \le R$ and a 
nonzero function $g$ with $\operatorname{supp} g \subset S \cap \ell$ 
and $\operatorname{supp} \hat{g} \subset \mathrm{N}_R^\alpha(\operatorname{supp} \hat{f})$.
\end{lemma}

\begin{remark}
This lemma provides a line support structure on $\mathbb Z_{\mathbf N}$ using Theorem \ref{thm:cyclotomic}. It asserts that for any nonzero function \(f\) with physical support \(S\), there exists an irreducible line \(\ell\) with \(\|\ell\|_{\alpha} \le R\lesssim |S|^{\alpha/(1+\alpha)}\) such that one can construct a nonzero function \(g\) whose physical support lies in \(S \cap \ell\) and whose Fourier support lies in an \(R\)-neighborhood of \(\operatorname{supp} \hat f\).
In other words, one can always find a low-complexity direction along which the restricted component of \(f\) remains nontrivial in the position domain while not spreading too far in the frequency domain. This property is crucial in density increment arguments in additive combinatorics, where it is used to locate linear structures and eventually derive the full linearity or periodicity of a set. Here this line support structure will be used to prove an anisotropic discrete FUP.
\end{remark}

\begin{remark}
Theorem \ref{thm:cyclotomic} is related to this key line support lemma, because any function $g : \mathbb Z_{\mathbf N} \to \mathbb C$ with $\operatorname{supp} \hat{g} \subset [0, \lfloor D^{1/\alpha}\rfloor] \times [0, D]$ can be expressed as a polynomial evaluated at roots of unity:
\[
    g(x,y) = \frac{1}{\sqrt{N_1N_2}} \sum_{0\le k\le \lfloor D^{1/\alpha}\rfloor,\, 0\le l\le D} 
    \hat{g}(k,l) z^k w^l, \qquad z = e^{\frac{2\pi i}{N_1}x},\, w = e^{\frac{2\pi i}{N_2}y}.
\]
The bound of the cyclotomic points contributes when the polynomial does not cut out a line; see \emph{Case 3} in the proof of Lemma \ref{lem:polysupp}.\end{remark}

We define the \emph{$\alpha$-degree} of a polynomial $F$ as
\begin{equation}\label{eq:alpha-deg}
    \deg_\alpha F = \max_{a_{kl}\neq 0} \max(|k|^\alpha, |l|), \qquad \text{where}\quad F(z,w) = \sum_{k,l} a_{kl}z^kw^l.
\end{equation}
By the inequality $(a+b)^\alpha \le a^\alpha + b^\alpha$ (which holds for $\alpha<1$), we have 
\begin{equation}\label{eq:relation-deg-poly}
\max\left(\deg_\alpha F,\, \deg_\alpha G\right)\le \deg_\alpha(FG)\le \deg_\alpha F + \deg_\alpha G.
\end{equation}

To prove Lemma \ref{lem:Cohen-Lemma11}, it suffices to prove the following lemma concerning line-supported polynomials, which modifies \cite[Lemma 20]{Cohen}.

Indeed, if Lemma \ref{lem:polysupp} holds, then for $f$, $S$, $R$ as in Lemma \ref{lem:Cohen-Lemma11}, there exists a polynomial $F^*$ with coefficients $a_{kl}$ such that $\deg_{\alpha} F^*\le C_\alpha|S|^{\alpha/(1+\alpha)}-1\le R$ and $S\setminus Z_{\mathbf N}(F^*)\neq\emptyset$ lies on an irreducible line $\ell$ satisfying $\|\ell\|_\alpha\le |S|^{\alpha/(1+\alpha)}\le R$.
Then the function 
\[
g=\frac1{\sqrt{N_1N_2}}\, F^*(e^{\frac{2\pi i}{N_1}x},\, e^{\frac{2\pi i}{N_2}y})\cdot f
\]
is nonzero and supported on $\ell\cap S$ with 
\[\operatorname{supp} \hat g \subset [0,\lfloor R^{1/\alpha} \rfloor)\times [0,R) + \operatorname{supp} \hat f \subset \mathrm{N}_R^\alpha(\operatorname{supp} \hat{f}).\]

\begin{lemma}[Polynomial version]\label{lem:polysupp}
 For all $\alpha\in(0,1)$ there exists a constant $C_\alpha \gg 1$, depending only on $\alpha$, such that for all sufficiently large $\mathbf N$ and every nonempty set $S\subset \mathbb Z_{\mathbf N}$, 
there exists a polynomial $F^*$ with $\deg_{\alpha} F^*\leq C_\alpha|S|^{\alpha/(1+\alpha)} - 1$ such that 
    $S \setminus Z_{\mathbf N}(F^*)$ is nonempty
    and lies on an irreducible line $\ell$ with $\|\ell\|_\alpha \le |S|^{\alpha/(1+\alpha)}$.
\end{lemma}

\begin{remark}
    The constant $C_\alpha$ determined in the proof can be taken as 
    \[C_\alpha=\max\left({30^\alpha}{(1-0.3^{\frac\alpha{1+\alpha}})^{-1-\alpha}},\, (1-0.99^{\frac\alpha{1+\alpha}})^{-1}\right), 
   \]  
which goes to $\infty$ as $\alpha\to0$.
\end{remark}

\begin{proof}
We present a recursive algorithm to find the polynomial $F^*$, based on induction on the size of $S$. Let $F$ be a polynomial of \emph{minimal} $\alpha$-degree $D$ such that $S\subset Z_{\mathbf N}(F)$. Existence of minimal degree comes from the definition \eqref{eq:alpha-deg}.  
By a modification of  \cite[Lemma 21]{Cohen}, we have $D \le |S|^{\alpha/(1+\alpha)}$. This uses the trivial linear algebra that we can always find a polynomial with $d$ coefficients, vanishing at arbitrary given $d-1$ distinct points.
If several such polynomials exist, choose one with the fewest irreducible factors.
 
 We now consider four cases. The first two are good and simple and serve as the base for the induction, while the last two involve induction on the size of $S$. In the fourth reducible case, the polynomial can be factored into two proper factors, and then we consider the non-zero point set of the dominant factor, which is essentially smaller than $S$. After further inductive steps, we eventually reach irreducible cases.
In particular, in the third case, when $F$ is irreducible but does not cut out a line, Theorem \ref{thm:cyclotomic} implies that it has not many cyclotomic points, i.e.
\[|S|\le |Z_{\mathbf N}(F)|\lesssim D^{1+1/\alpha}\lesssim |S|.\]
Then, by the trivial linear algebra, we obtain a polynomial of $\alpha$-degree $D-1$ whose zero point set is larger than $0.01|S|$. Induction then interprets that the set $S$ after removing these zero points is smaller than $0.99|S|$.

\emph{Case 1. $F$ is irreducible and cuts out an irreducible line $\ell$.} Then $F$ is of the form
\[
    F(z,w) = z^a w^b - \zeta \quad \text{or} \quad F(z,w) = z^a - \zeta w^b.
\]
Thus we can assume
\[
    Z(F)=\{(x,y)\in \mathbb{T}^2:ax+by=c\}
\]
for some coprime integers $a,b\in \mathbb{Z}$ and some $c\in \mathbb{Q}$, and consequently
\[
    Z_{\mathbf N}(F)=\left\{(x,y)\in {\mathbb Z_{\mathbf N}}:\frac{ax}{N_1}+\frac{by}{N_2} \equiv c \mod 1\right\}.
\]
Since $S\subset Z_{\mathbf N}$ is non-empty, we may take $F^*=1$ and the irreducible line $\ell=Z_{\mathbf N}(F)$.
Note that
\[\|\ell\|_\alpha = \deg_\alpha F=D \leq |S|^{\alpha/(1+\alpha)}.\]

\emph{Case 2.  The set $S$ is small.} Assume $|S|\le S_0(\alpha)$, where $S_0(\alpha)$ is given in the next case. We may assume first $|S|\leq 200$. Let $S = \{(x_k, y_k) \in {\mathbb Z_{\mathbf N}}\}$, and let $\{x_1,\dots,x_m\}$ be all distinct $x$-coordinates appearing in $S$. If $m=1$, we take as before the trivial polynomial $F^*=1$ and the vertical line $x=x_m$.
If $m>1$, set the polynomial
\[
    F^* = (z-e^{2\pi i x_1/N_1})\cdots (z-e^{2\pi i x_{m-1}/N_1}).
\]
Then its $\alpha$-degree is
\[\deg_\alpha F=(m-1)^\alpha \le 200^\alpha |S|^{\alpha/(1+\alpha)},\]
and $S\setminus Z_{\mathbf N}(F^*)$ is non-empty and lies on the vertical line $x = x_m$ in ${\mathbb Z_{\mathbf N}}$. Note that for all $S$ with $|S|\le S_0(\alpha)$, the result still holds and the constant $C_\alpha$ in the lemma can be taken uniformly.

\emph{Case 3. $F$ is irreducible but does not cut out a line.} In this case,
by Theorem \ref{thm:cyclotomic} we have $|S|\leq 30D^{1+1/\alpha}\le 30|S|$.
We can choose $|S|\ge S_0(\alpha)$ and then $D\ge D_0(\alpha)$ such that $(1-1/D)^{1+1/\alpha}>3/10$.
Then we can always choose a polynomial $G$ of $\alpha$-degree $D-1$ whose number of zero points in $S$ is at least
\[
    \min(|S|,(D-1)^{1+\frac{1}{\alpha}})\geq |S|/100.
\]
Let $A = S\cap Z_{\mathbf N}(G)$. Note that $S\setminus A$ is nonempty, i.e., $S\not\subset Z_{\mathbf N}(G)$, by the minimality of $D$.
Since $|S\setminus A| \leq 0.99|S|$, we can apply the inductive hypothesis 
(e.g., we can choose in the first inductive step $|S|/S_0\in (1,1/0.99)$ such that $|S \setminus A|\le S_0$) 
to find a polynomial $H$ whose non-zero point set in $S\setminus A$ is nonempty and contained in one irreducible line $\ell$ with $\deg_{\alpha} H\leq C_\alpha|S\setminus A|^{\alpha/(1+\alpha)}$. Set $F^*=GH$.
Then $S\setminus Z_{\mathbf N}(F^*)= (S\setminus A)\setminus Z_{\mathbf N}(H)$ also lies on the irreducible line $\ell$, and by \eqref{eq:relation-deg-poly}
\[\operatorname{deg}_\alpha F^*\leq \operatorname{deg}_\alpha G + \operatorname{deg}_\alpha H
     \leq D-1 + C_\alpha|S\setminus A|^{\alpha/(1+\alpha)} < C_\alpha|S|^{\alpha/(1+\alpha)} - 1,\]
provided we pick $C_\alpha$ large enough to ensure $C_\alpha(1-0.99^{\frac\alpha{1+\alpha}})\gg 1$.
The last inequality follows from $|S\setminus A|\leq 0.99|S|$ and $D\leq |S|^{\alpha/(1+\alpha)}$.

\emph{Case 4. $F$ is reducible.} Let $F = GH$ be a nontrivial factorization, where neither $G$ nor $H$ are scalars 
and $|S \cap Z_{\mathbf N}(G)|\le |S \cap Z_{\mathbf N}(H)|$. Let $T = S\setminus Z_{\mathbf N}(H)$.
Because $\operatorname{deg}_\alpha H \leq \operatorname{deg}_\alpha F$ and $H$ has fewer irreducible factors than $F$, 
we have $S \not\subset Z_{\mathbf N}(H)$. Hence $T$ is nonempty. 
From $|S \cap Z_{\mathbf N}(G)|\le |S \cap Z_{\mathbf N}(H)|$ we obtain $|T| \leq |S|/2$.
Using the inductive hypothesis we may find a polynomial $P$ whose non-zero point set in $T$ 
is nonempty and lies on one irreducible line $\ell$ with $\deg_{\alpha} P\leq C_\alpha|T|^{\alpha/(1+\alpha)}$.
Set $F^* = HP$; then $S\setminus Z_{\mathbf N}(F^*)=T\setminus Z_{\mathbf N}(P)$ is nonempty and lies on $\ell$. By \eqref{eq:relation-deg-poly}, the $\alpha$-degree satisfies
\[
        \operatorname{deg}_\alpha F^* \leq \operatorname{deg}_\alpha P + \operatorname{deg}_\alpha H
        \leq C_\alpha(|S|/2)^{\alpha/(1+\alpha)} + |S|^{\alpha/(1+\alpha)} \leq C_\alpha|S|^{\alpha/(1+\alpha)}-1,
\]
if we pick $C_\alpha$ sufficiently large.

The proof is complete; the interested reader is referred to \cite[Figure 5]{Cohen} for a pictorial interpretation of the isotropic case.
\end{proof}

\subsection{Line containment lemma}
We recall the following lemma from \cite[Lemma 14]{Cohen}. It implies that if, for a given direction, no line in that direction is entirely contained in the limiting set $\mathbf X$, then every line in that direction contains a point whose distance to $\mathbf X$ is bounded below by a positive constant depending only on $\mathbf X$.
\begin{lemma}[Line containment/separation]\label{lem:Lemma 4 in Cohen2025}
Suppose a set $\mathbf{X} \subsetneq \mathbb{T}^2$ is closed. Then there exists a constant $c_{\mathbf{X}} > 0$, depending only on $\mathbf X$, such that for every direction $\mathbf{v} \in \mathbb{R}^2 \setminus \{0\}$, either there exists some $\mathbf p\in \mathbb T ^2$ such that the coset $\mathbb{R}\mathbf{v} + \mathbf{p}$ lies entirely in $\mathbf{X}$, or
\[    \sup_{\mathbf x\in\mathbb{R}\mathbf{v}+\mathbf{p}} d_{\alpha, \mathbb T^2}(x,\mathbf{X}) \ge c_{\mathbf{X}} \qquad \forall\ \mathbf p\in \mathbb T^2.\]
\end{lemma}

\begin{remark}
In our construction, the limiting Cantor set $\mathbf X$ and therefore $c_{\mathbf X}$ depend only on the alphabet $\mathcal A$ and the base $\mathbf M$, and are independent of $k$ and $\mathbf N$.
\end{remark}

The following proposition shows that there are no oblique lines that lie entirely in the anisotropic Cantor sets $\mathbf{X}$ and $\mathbf{Y}$. This is the distinguishing feature of the anisotropy. See Figure \ref{fig:Anistropic_Cantor_set_and_the_intersection_with_lines} for an example.
 \begin{figure}[htbp]
    \centering
    \includegraphics[width=0.5\textwidth]{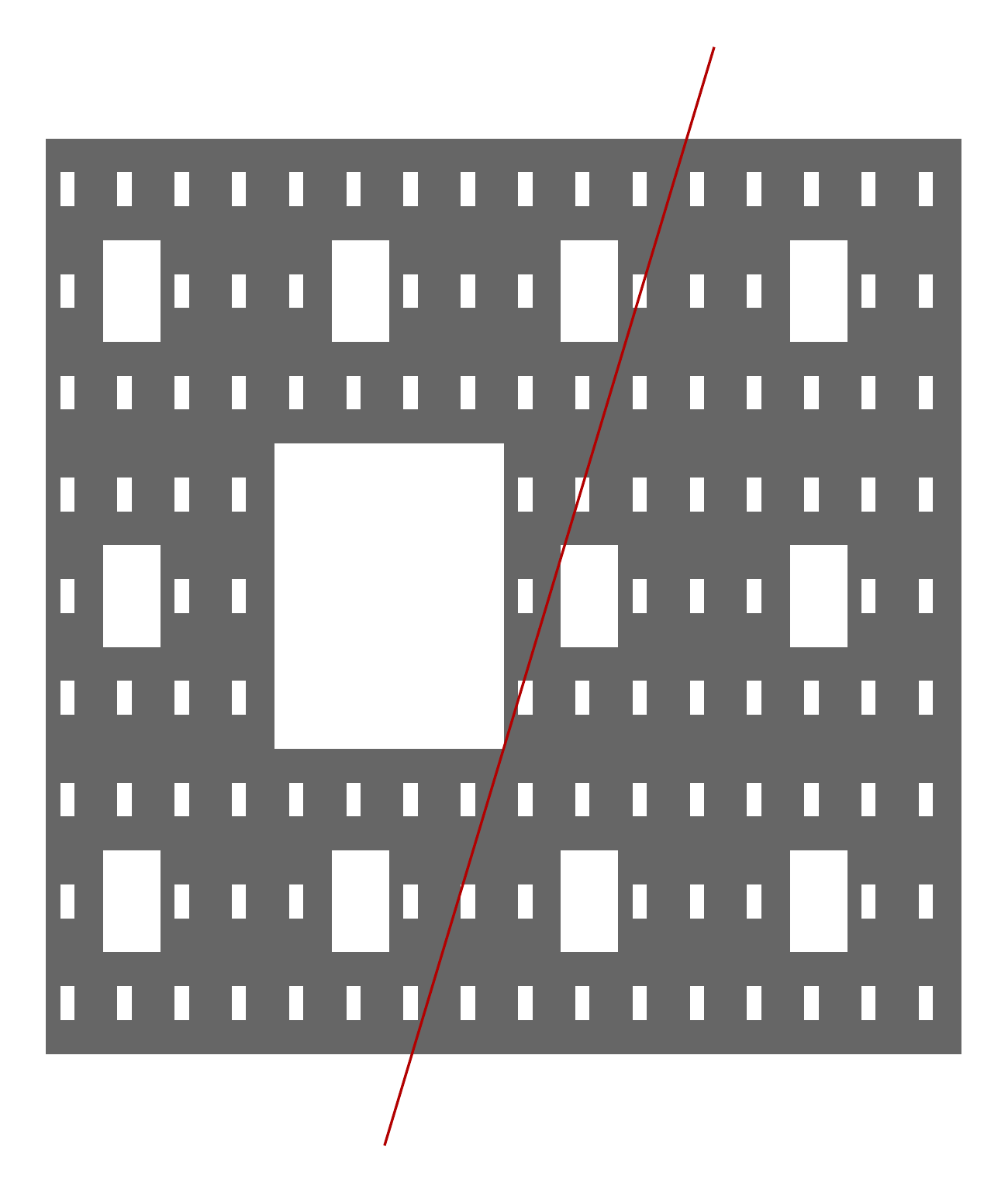}
\caption{Intersection of an oblique line with the \emph{anisotropic} limiting Cantor set $\mathbf{X}$ for parameters $\mathbf M=(4,3)$ and $\mathcal{A}=\mathbb{Z}_{\mathbf M}\setminus\{(1,1)\}$.}\label{fig:Anistropic_Cantor_set_and_the_intersection_with_lines}
\end{figure}

\begin{proposition}[Line containment]\label{prop:No oblique line lies entirely on the anisotropic Cantor set}
  Suppose the alphabet $\mathcal{A}\subset \mathbb Z_{\mathbf M}$ on base $\mathbf M$ is not full. If some open subset of a line $\mathbb{R}\mathbf{v} + \mathbf{p}$ lies entirely in the limiting set $\mathbf{X}$, then $\mathbf{v}$ is parallel to the coordinate axes.
Furthermore, if $\mathbf{v} = (1,0)$, then $\mathcal{A}$ is full for some row; if $\mathbf{v} = (0,1)$, then $\mathcal{A}$ is full for some column.
\end{proposition}

\begin{proof}   
    The statement for $\mathbf{v}$ parallel to the coordinate axes is trivial. 
    So we assume $\mathbf{v}$ is not parallel to any coordinate axis. 
    We need to show that any open subset of the line $L$ in $\mathbb{R}^2$ 
    defined as $\mathbb{R}\mathbf{v} + \mathbf{p}$ cannot lie entirely on $\mathbf{X}$.
    Without loss of generality, we assume $\mathbf{v}=(v_0,1)$ where $v_0>0$; 
    the case $v_0<0$ is similar. 
    It suffices to show that for large $k$ and any $n_0\in \mathbb{N}$ with $n_0\leq M_2^{k}-1$, 
    the open segment $I_{n_0,k}\subset L$ defined by
    \[
        I_{n_0,k}=\left\{\left(x_0+M_2^{-k}tv_0,\;M_2^{-k}(n_0+t)\right):t\in [0,1]\right\}
    \]
    does not lie entirely on $\mathbf{X}$, where $(x_0,M_2^{-k}n_0)$ is the unique point on $L$ 
    whose second coordinate equals $M_2^{-k}n_0$.
   
    We consider the rescaling $\Phi^k:\mathbb{R}^2\to \mathbb{R}^2$ defined by
    \[
        \Phi^k(x,y)= (M_1^kx,\;M_2^ky)
    \]
    and the rescaled interval
    \[
        \Phi^k(I_{n_0,k})=\{(M_1^kx_0+(M_1/M_2)^ktv_0,\;n_0+t):t\in [0,1]\}.
    \]
    Now we examine the structure of the rescaled Cantor set $\Phi^k(\mathbf{X})$.
    Since $\mathcal{A}$ is not full, we choose some $(a_0,b_0)$ not in the alphabet:
    \[
        (a_0,b_0)\in \mathbb Z_{\mathbf M}\setminus \mathcal{A}.
    \]
    The self-similarity of $\mathbf{X}$ implies that for any $(m,n)\in \mathbb{Z}^2\subset \mathbb{R}^2$
    we have
    \[
        \left[\left(m+\frac{a_0}{M_1},\,m+\frac{a_0+1}{M_1}\right)\times 
        \left(n+\frac{b_0}{M_2},\,n+\frac{b_0+1}{M_2}\right)\right]\cap \Phi^k(\mathbf{X})=\emptyset.
    \]
    In view of this ``hole'' in $\Phi^k(\mathbf{X})$, let $t_0\in (b_0/M_2,1)$ be the smallest number such that 
 \[
        M_1^kx_0+(M_1/M_2)^kt_0v_0\in \mathbb{Z}.
    \]   
Set $t_1=t_0+\frac{1}{(M_1/M_2)^kv_0}$.
    The assumption $M_1>M_2$ implies that for sufficiently large $k$,
        \begin{align}
            &t_0\leq \frac{b_0}{M_2}+\frac{10}{(M_1/M_2)^kv_0}, \nonumber \\[2pt]  
            &t_1\leq \frac{b_0}{M_2}+\frac{11}{(M_1/M_2)^kv_0}<\frac{b_0+1}{M_2}. \label{eq:t0_estimate}
        \end{align}
    Let 
    \[
        m_0=M_1^kx_0+(M_1/M_2)^kt_0v_0\in \mathbb{Z},
    \]
    then
    \[
        m_0+1=M_1^kx_0+(M_1/M_2)^kt_1v_0\in \mathbb{Z}.
    \]
    Therefore, \eqref{eq:t0_estimate} implies (see Figure \ref{fig:Rescaled_anisotropic_cantor_set_and_lines})
    \[
        \Phi^k(I_{n_0,k})\cap\left[\left(m_0+\frac{a_0}{M_1},\,m_0+\frac{a_0+1}{M_1}\right)\times 
        \left(n_0+\frac{b_0}{M_2},\,n_0+\frac{b_0+1}{M_2}\right)\right]\neq\emptyset.
    \]
    Hence $\Phi^k(I_{n_0,k})$ does not lie entirely in $\Phi^k(\mathbf{X})$, which completes the proof.
    
\end{proof}

\begin{figure}[htbp]
\centering   \includegraphics[width=0.8\textwidth]{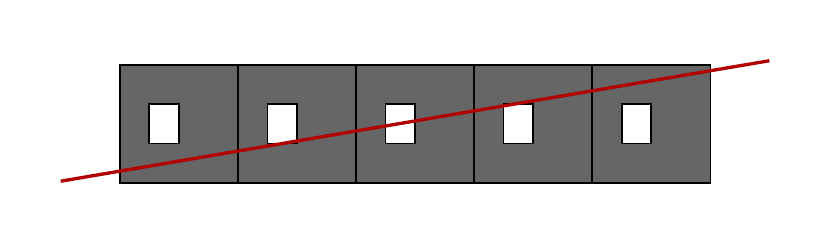}
 \caption{The intersection of $\Phi^k(I_{n_0,k})$ and $\Phi^k(\mathbf{X})$. 
The key point is that the line has very small but positive slope.}
   \label{fig:Rescaled_anisotropic_cantor_set_and_lines}
\end{figure}

\subsection{Proof of Theorem \ref{thm:DFUP}}
We now have all the ingredients to prove the anisotropic discrete FUP \eqref{eq:dfup} for the discrete Cantor sets $\mathcal{X}_k$ and $\mathcal{Y}_k$ defined in \eqref{eq:dcantorset}.

Recall that
\[
\begin{aligned}
\mathcal{X}_k &= \bigl\{(a_1^{(0)}+\cdots+a_1^{(k-1)}M_1^{k-1},\; a_2^{(0)}+\cdots+a_2^{(k-1)}M_2^{k-1}) : (a_1^{(j)},a_2^{(j)})\in\mathcal{A}\bigr\} \subset \mathbb Z_{\mathbf N},\\
\mathbf{X}_k &= \overline{\bigl\{ (x,y)\in \mathbb{T}^2 : (\lfloor M_1^k x\rfloor,\lfloor M_2^k y\rfloor)\in \mathcal{X}_k \bigr\}}  ~  \ ~ \ ~  \searrow~  \  ~  \    
\mathbf{X} = \bigcap_k \mathbf{X}_k \subset \mathbb{T}^2 ,
\end{aligned}
\]
where $\mathcal{X}_k$ is the discrete Cantor set on $\mathbb Z_{\mathbf N}$, $\mathbf{X}_k$ is its closed drawing, and $\mathbf{X}$ is the limiting Cantor set in $\mathbb{T}^2$. 
The sets $\mathbf{X}_k$ decrease to $\mathbf{X}$, which has zero Lebesgue measure: $|\mathbf X|=0$. 
Observe that a point $(x,y)\in \mathbb{T}^2$ belongs to $\mathbf{X}$ if and only if its expansion in the base $\mathbf M$ contains only digits from the alphabet $\mathcal{A}$.

We also denote the discretization of $\mathbf{X}$ by
\begin{equation}\label{eq:XN}
X_{\mathbf N} := \left\{ (x,y)\in \mathbb Z_{\mathbf N} :
\left[\frac{x}{N_1},\frac{x+1}{N_1}\right]\times \left[\frac{y}{N_2},\frac{y+1}{N_2}\right] \cap \mathbf{X} \neq \emptyset \right\} \supset \mathcal{X}_k.
\end{equation}
For $0<r=o(1)\ll 1$, let the (continuous and full) \emph{$(r,\alpha)$-neighborhood} of $\mathbf X$ in $\mathbb T^2$ be 
\[\operatorname{N}_r^\alpha(\mathbf X)=\mathbf X+[-r^{1/\alpha},r^{1/\alpha}]\times[-r,r].\]
Then $(x,y)\in X_{\mathbf N}$ if and only if $(x,y)/\mathbf N \in \operatorname{N}_{1/N_2}^\alpha(\mathbf{X}) \supset \mathbf{X}_k$. 
It is easy to verify that
\begin{equation}\label{eq:NXNinNX}
\operatorname{N}_R^\alpha(X_{\mathbf N}) \subset \mathbf N \cdot \operatorname{N}_{R/N_2}^\alpha(\mathbf X), \qquad\forall\, R=o(N_2), 
\end{equation}
where $\operatorname{N}_R^\alpha(X_{\mathbf N})$ is the discrete half neighborhood \eqref{eq:nb-R} in $\mathbb Z_{\mathbf N}$. 
Since $|\mathbf X|$=0, if we take $R=o(N_2)$ then $|\operatorname{N}_{R/N_2}^\alpha(\mathbf X)| \to 0$ as $k\to \infty$, which implies 
\begin{equation}\label{eq:|NXN|}
|X_{\mathbf N}|\le |\operatorname{N}_R^\alpha(X_{\mathbf N})|= o(N_1 N_2).\end{equation}
We use similar notation for the sets associated with another alphabet $\mathcal{B}$, namely $\mathcal{Y}_k$, $\mathbf{Y}_k$, $\mathbf{Y}$, $Y_{\mathbf N}$, $\operatorname{N}_R^\alpha(Y_{\mathbf N})$ and $\operatorname{N}_{r}^\alpha(\mathbf Y)$.

\begin{proof}[Proof of the Theorem \ref{thm:DFUP}]  
  At the beginning of this section we mentioned that, by submultiplicativity and unitarity, it suffices to show that for some large $k$ there is no $f$ supported on $\mathcal{X}_k$ with $\operatorname{supp} \hat f \subset \mathcal{Y}_k$.
Assume, for contradiction, that for every $k$ there exists a nonzero function $f:\mathbb{Z}_{\mathbf N}\to\mathbb{C}$ with 
\[
\operatorname{supp} \hat{f}\subset \mathcal Y_k \subset Y_{\mathbf N}, \qquad 
\operatorname{supp} f \subset \mathcal X_k \subset X_{\mathbf N},
\]
where $X_{\mathbf N}, Y_{\mathbf N}$ are the discretized sets \eqref{eq:XN}.
  
Applying the line support lemma (Lemma~\ref{lem:Cohen-Lemma11}) to $f$ and $S = \operatorname{supp} f$, we obtain a triple $(R, \ell, g)$ such that $0 < R \le C_\alpha |S|^{\alpha/(1+\alpha)}$, $\ell$ is a line with $\|\ell\|_\alpha \le R$, and $g$ is a nonzero function supported on $\ell \cap S$ with $\operatorname{supp} \hat{g} \subset \mathrm{N}_R^\alpha(Y_{\mathbf N})$. Set the vector $\mathbf{v} = (a/N_1, b/N_2)$, where $(a,b)$ is the coprime normal vector of $\ell$.

Since $\mathcal A$ is not full, by Proposition \ref{prop:No oblique line lies entirely on the anisotropic Cantor set}, $\mathbf v$ is horizontal, vertical, or (oblique but) satisfies the property that no line of the form $\mathbb{R}\mathbf{v} + \mathbf{p}$ lies in $\mathbf Y$.

\emph{Case 1. Assume $\mathbb{R}\mathbf{v} + \mathbf{p}$ does not lie entirely in $\mathbf{Y}$ for any $\mathbf p$.}  
Note that when $\mathbf v$ is neither vertical nor horizontal, this is always the case by Proposition \ref{prop:No oblique line lies entirely on the anisotropic Cantor set}.
We \emph{claim} that there exists a constant $\delta=\delta(\alpha,C_\alpha,c_{\mathbf{Y}})\in (0,1)$ such that
\begin{equation}\label{eq:Claim of the first case in the proof of the main thm}
|S| = |\operatorname{supp} f| \ge (\delta/C_\alpha)^{(1+\alpha)/\alpha} N_1N_2,
\end{equation}
where $c_{\mathbf{Y}}$ is the constant in Lemma \ref{lem:Lemma 4 in Cohen2025} and $C_\alpha$ is the constant in Lemma \ref{lem:Cohen-Lemma11}. 
This claim leads to a contradiction because $|S|\le |X_{\mathbf N}| = o(N_1N_2)$. 
We will take $\delta=c_{\mathbf{Y}}/2$, which is independent of $\mathbf N$.
 
For the claim, it suffices to show that $R/N_2 \ge \delta$ in view of $R\le C_\alpha|S|^{\alpha/(1+\alpha)}$. 
We now argue by contradiction.  
Suppose $R/N_2 < \delta$; we will show that $g=0$. 
        
Recall that $\mathbf{v} = (a/N_1,b/N_2)$ where $(a,b)$ is the normal vector of $\ell$. 
Because $g$ is supported on $\ell$ in the physical space, the dual function $\hat{g}$ has {constant magnitude} on the dual lines $\mathbb{Z}\cdot(a,b) + \mathbf{p}$ in the Fourier space. Indeed, let $(\xi,\eta) \in \mathbb{Z}_{\mathbf N}$ be arbitrary; for $n\in \mathbb{Z}$ we have
\[
\hat{g}(\xi+na,\eta+nb) = e^{2\pi i n c}\, \hat{g}(\xi,\eta),
\]
which is a rotation in the complex plane, where 
\[
\hat{g}(\xi,\eta) = \frac{1}{\sqrt{N_1N_2}} 
\sum_{\frac{ax}{N_1}+\frac{by}{N_2}\equiv c \mod 1} 
g(x,y)\,e^{-2\pi i(\frac{x\xi}{N_1}+\frac{y\eta}{N_2})}.
\]        
Take $\mathbf{p}=(\xi/N_1,\eta/N_2)\in \mathbb{T}^2$. By assumption, $\mathbb{R}\mathbf{v} + \mathbf{p}$ does not lie entirely in $\mathbf{Y}$. Then by Lemma \ref{lem:Lemma 4 in Cohen2025} there exists $t\in \mathbb{R}$ such that $d_{\alpha,\mathbb T^2}(t\mathbf{v} + \mathbf{p}, \mathbf{Y}) \ge c_{\mathbf{Y}}$. Let $n$ be the nearest integer to $t$. Then, with $\delta = c_{\mathbf{Y}}/2$, we have
\[
d_{\alpha,\mathbb T^2}(n\mathbf{v} + \mathbf{p}, \mathbf{Y}) \ge c_{\mathbf{Y}} - \max\left(\frac{|a|^\alpha}{N_1^\alpha}, \frac{|b|}{N_2}\right) \ge 
c_{\mathbf{Y}}- \frac{R}{N_2} > c_{\mathbf{Y}}-\delta = \delta > R/N_2.
\]
Since  \eqref{eq:NXNinNX} we have 
\[
\operatorname{supp} \hat{g} \subset \mathrm{N}_R^\alpha(Y_{\mathbf N}) \subset \mathbf N\cdot \operatorname N_{R/N_2}(\mathbf Y),
\]
then $(\xi+na,\eta+nb)=\mathbf N\, (n\mathbf v+\mathbf p)\notin \operatorname{supp}\hat{g}$, and therefore $\hat{g}(\xi+na,\eta+nb)=\hat{g}(\xi,\eta)=0$ for all $(\xi,\eta) \in \mathbb{Z}_{\mathbf N}$. 
This shows that $\hat{g}=g=0$, which contradicts the fact that $g$ is nonzero, and completes the proof of the claim \eqref{eq:Claim of the first case in the proof of the main thm}. 
Note that this case does not have a simple correspondence with any condition on the alphabets.          

\emph{Case 2. Assume $\mathbf v$ is vertical and there exists a (vertical)  line  $\mathbb R\mathbf v+\mathbf p$ for some $\mathbf p\in \mathbb T^2$ that lies entirely in $\mathbf Y$.} Then the line $\ell$ is horizontal, and in this case $\mathcal B$ has some full column. 
By the condition on the alphabets $\mathcal A, \mathcal B$ in the theorem we may assume that $\mathcal{A}$ is not full for each row.

Then we can write $\ell$ as $\{(t,y_0):t\in \mathbb{R}\}$ for some $y_0\in \mathbb{Z}_{N_2}$, and introduce a new function $\tilde{g}:\mathbb{Z}_{N_1}\to \mathbb{C}$ by $\tilde{g}(n):=g(n,y_0)$. Then
\[
\hat{g}(\xi,\eta) = \frac{1}{\sqrt{N_1N_2}} \sum_{n\in\mathbb{Z}_{N_1}} \tilde{g}(n)\,e^{-2\pi i(\frac{n\xi}{N_1}+\frac{y_0\eta}{N_2})}
= \frac{1}{\sqrt{N_2}}\,e^{-\frac{2\pi iy_0 \eta}{N_2}}\, \hat{\tilde{g}}(\xi),
\]
where $\hat{\tilde{g}}$ is the $\mathbb{Z}_{N_1}$-Fourier transform of $\tilde{g}$. It is easy to see that $|\operatorname{supp} \hat{g}| = N_2|\operatorname{supp} \hat{\tilde{g}}|$.
However, since $\mathcal{A}$ is not full for each row, there exists a constant $c_\mathcal{A}>0$ (depending only on $\mathcal{A}$) such that there is a open segment contained in $\ell\setminus X_{\mathbf N}$ of length at least $c_\mathcal{A} N_1$ (for sufficiently large $k$, say $k\ge 10$). 
Because $g$ is supported on $\ell\cap X_{\mathbf N}$, the one‑dimensional discrete uncertainty principle \cite[Proposition 9]{Cohen} implies
\[
|\operatorname{supp} \hat{g}| = N_2|\operatorname{supp} \hat{\tilde{g}}| \ge N_2 c_\mathcal{A} N_1,
\]
which contradicts $|\operatorname{supp} \hat{g}|\le |\mathrm{N}_R^\alpha(Y_{\mathbf N})| = o(N_1N_2)$ by \eqref{eq:|NXN|}.

\emph{Case 3. Assume $\mathbf v$ is horizontal and there exists a (horizontal)  line  $\mathbb R\mathbf v+\mathbf p$ for some $\mathbf p\in \mathbb T^2$ that lies entirely in $\mathbf Y$.} Then the line $\ell$ is vertical, and in this case $\mathcal B$ has some full row. 
By the condition on the alphabets $\mathcal A, \mathcal B$ in the theorem we may assume that $\mathcal{A}$ is not full for each column. 
The argument  is the same as that in \emph{Case 2}, once we swap the roles of rows and columns.
\end{proof}

\section{Anisotropic continuous FUP}\label{sec:cfup}
This section is devoted to the study of the anisotropic FUP for continuous Fourier transforms on anisotropic fractal sets. We consider both the (isotropic) semiclassical Fourier transform \eqref{eq:DFT} and its anisotropic counterpart \eqref{eq:anis-semi-F}. First, we investigate the relationship between the discrete and continuous versions of the FUP. Then, we prove an anisotropic FUP for the (isotropic) semiclassical Fourier transform.  A rescaled FUP for the standard Fourier transform will further imply a reflected version for Bedford-McMullen carpets and anisotropic  semiclassical Fourier transform.  

\subsection{Continuous FUP implies discrete FUP}\label{subsec:implication}
As shown in \cite[Proposition 5.8]{dyatlov2018dolgopyat} and 
\cite[Appendix B]{Cohen}, the discrete FUP for discrete Cantor sets $\mathcal{X}_k,\mathcal{Y}_k$ 
is a corollary of the corresponding continuous version of FUP for the limiting Cantor sets $\mathbf{X},\mathbf{Y}$. 
However, note that since our discrete Fourier transform is \textit{not isotropic}, i.e., 
the frequencies along the $x$-direction and the $y$-direction are different, the corresponding continuous version of FUP should also be anisotropic.
More precisely, under the following discrete-continuous correspondence 
\[ (x,y)\in {\mathbb Z_{\mathbf N}} \Longleftrightarrow \left(\frac{x}{N_1},\frac{y}{N_2}\right)\in \mathbb{T}^2\]
the anisotropic discrete Fourier transform $\mathcal{F}_{\mathbf N}$  corresponds to the anisotropic semiclassical Fourier transform 
\[ \mathcal{F}_h^{\operatorname{ani}}f(\xi,\eta) = 
    \frac{1}{h^{(1+\alpha)/2}} \int_{\mathbb{R}^2} e^{-2\pi i(x\xi/h + y\eta/h^{\alpha})} f(x,y) dx dy
\]
if we use the semiclassical parameter $h = 1/N_1\to 0$.
Therefore, we expect that the discrete anisotropic FUP  
\[
    \|\mathds{1}_{\mathcal{Y}_k}\mathcal{F}_{\mathbf{N}}\mathds{1}_{\mathcal{X}_k}\|_{{\ell^2_{\mathbf N}}\to 
    {\ell^2_{\mathbf N}}} \leq C N_1^{-\beta}
\]
should correspond to the following \emph{continuous version of  {anisotropic}  FUP}:
\[ \|\mathds{1}_{\mathbf{Y}_h}\mathcal{F}_h^{\operatorname{ani}}\mathds{1}_{\mathbf{X}_h}\|_{L^2(\mathbb{R}^2)\to L^2(\mathbb{R}^2)} \leq C h^{\beta},\]
where $\mathbf X_h$ is a $h$-scale neighborhood of $\mathbf X$ in $\mathbb R^2$.

Before stating the anisotropic continuous FUP, we need to introduce the \textit{$\alpha$-distance on $\mathbb{R}^2$}: for any two points $\mathbf x=(x_1,x_2), \mathbf y=(y_1,y_2) \in \mathbb R^2$ 
\[
d_{\alpha,\mathbb{R}^2}(\mathbf x, \mathbf y) = \max(|x_1-y_1|^\alpha,|x_2-y_2|).
\]
We denote the \emph{$\alpha$-ball} centered at $\mathbf{x}$ of radius $r$ with respect to the $\alpha$-distance $d_{\alpha,\mathbb{R}^2}$ by
\begin{equation}\label{eq:alpha-ball}
B^\alpha_r(\mathbf{x}) = \{ \mathbf y \in \mathbb{R}^2 : d_{\alpha,\mathbb{R}^2}(\mathbf y,\mathbf{x}) < r \},
\end{equation}
and we write $B^\alpha_r:=B^\alpha_r(0)$. Note that $B^\alpha_{h^\alpha}$ is a rectangle of dimensions $2h\times 2h^\alpha$. To distinguish it from the standard Euclidean ball, we use $B_r(x)=B(x,r)$ for the standard ball.

The following theorem shows that the continuous  anisotropic FUP implies the discrete anisotropic FUP, as in the isotropic case. 
\begin{theorem}[Implication]\label{thm:C-D}
    Let $\mathcal{A},\mathcal{B}\subset {\mathbb Z_{\mathbf N}}$ be two alphabets with base $\mathbf M$, and let $\mathcal{X}_k,\mathcal{Y}_k\subset \mathbb Z_{\mathbf N}$ and $\mathbf{X},\mathbf{Y}\subset \mathbb{T}^2$ be the associated discrete Cantor iterates and limiting Cantor sets as $k\to\infty$.
    Suppose there exist $\beta>0$ and $C>0$ independent of $h$ such that
    \begin{equation}\label{eq:ani-cont-fup}
        \|\mathds{1}_{\mathbf{Y}+B_{h^\alpha}^\alpha}\mathcal{F}_h^{\operatorname{ani}}\mathds{1}_{\mathbf{X}+B_{h^\alpha}^\alpha}\|_{L^2(\mathbb{R}^2)\to L^2(\mathbb{R}^2)} \leq C h^{\beta}.
    \end{equation}
Then for every $\beta'\in (0,\beta)$ we can find $C_{\beta'}>0$ independent of $k$ such that
    \begin{equation}\label{eq:discrete-fup-from-cfup}
        \|\mathds{1}_{\mathcal{Y}_k}\mathcal{F}_{\mathbf{N}}\mathds{1}_{\mathcal{X}_k}\|_{\ell^2_{\mathbf N}\to \ell^2_{\mathbf N}} \leq C_{\beta'} M_1^{-k\beta'}.
    \end{equation}
Here $\mathcal{F}_h^{\operatorname{ani}}$ and $\mathcal{F}_{\mathbf{N}}$ are the anisotropic semiclassical Fourier transform \eqref{eq:anis-semi-F} and the anisotropic discrete Fourier transform \eqref{eq:DFT}, respectively.
\end{theorem}
\begin{remark}
    The converse is generally not true; that is, the discrete FUP does not necessarily imply the continuous FUP.
\end{remark}

\begin{remark}
    The associated Cantor set $\mathbf{X}\subset \mathbb{R}^2$ (or $\mathbf{Y}$) is a Bedford--McMullen carpet; see Figure \ref{fig:aniscantor}. 
    A FUP for these sets with respect to the (isotropic) semiclassical Fourier transform $\mathcal{F}_h$ can be obtained using a one-dimensional FUP under the assumption that both $\mathcal A$ and $\mathcal B$ are not full for each row 
    (see Theorem \ref{thm:cfup} and \ref{thm:cfup-bm}). 
    However, here we involve a kind of FUP for the \textit{anisotropic} semiclassical Fourier transform. 
    After an anisotropic rescaling, the anisotropic semiclassical Fourier transform reduces to the standard Fourier transform.
    Indeed, we have
    \[
        \|\mathds{1}_{\mathbf{Y}+B_{h^\alpha}^\alpha}
        \mathcal{F}_h^{\operatorname{ani}}\mathds{1}_{\mathbf{X}+B_{h^\alpha}^\alpha}\|_{L^2\to L^2}
        =\|\mathds{1}_{\widetilde{\mathbf{Y}}+[-1,1]^2}
        \mathcal{F}\mathds{1}_{\mathbf{X}+B_{h^\alpha}^\alpha}\|_{L^2\to L^2}
    \]
    if we define the (anisotropic) rescaled set 
    \[
        \widetilde{\mathbf{Y}}=\{(x/h,y/h^\alpha):(x,y)\in \mathbf{Y}\}.
    \]
    The authors do not know how to prove this kind of FUP for rescaled sets.  
    Instead, we can prove the following \emph{reflected} version of the anisotropic FUP:
    \[
        \|\mathds{1}_{\mathbf{Y}+B_{h^\alpha}^\alpha}
        \mathcal{F}_h^{\operatorname{ani}}\mathds{1}_{\operatorname{Ref}(\mathbf{X}+B_{h^\alpha}^\alpha)}\|_{L^2\to L^2}\leq Ch^\beta
    \]
    under certain loose restrictions (orthogonal non-fullness) on the alphabets $\mathcal{A}$ and $\mathcal{B}$, 
    where $\operatorname{Ref}:(x,y) \in\mathbb{R}^2\mapsto (y,x)\in \mathbb{R}^2$ is the reflection across the diagonal. See Section 
    \ref{subsec:rescaled} and 
    Theorem \ref{thm:cfup-bm-ani}.
\end{remark}

Theorem \ref{thm:C-D} is a consequence of the following more general estimate.  We state in a dual version. Let $\mathcal{F}_h^{\operatorname{ani*}}$ and $\mathcal{F}_{\mathbf{N}}^*$ be the dual  operators.
\begin{proposition}[General implication]\label{prop:General-CD}
    For $\mathcal{X}, \mathcal{Y} \subset \mathbb{Z}_{\mathbf{N}}$, define in $\mathbb{R}^2$ the sets
    \[
    X = \mathcal{X} \subset [0,N_1)\times[0,N_2), \qquad Y = \mathcal{Y}/\mathbf{N} \subset [0,1)^2.
    \] 
    Then for all $r\in (10/N_2,1/10)$ and all $m>0$ we have 
    \begin{equation}\label{eq:CDgeneral}
        \|\mathds{1}_{\mathcal{Y}}\mathcal{F}_{\mathbf N}^*\mathds{1}_{\mathcal{X}}\|_{\ell^2\to\ell^2}
        \lesssim_{m,\alpha} (rN_2)^{\frac{1+\alpha}{2\alpha}}\,  \|\mathds{1}_{Y+B^\alpha_r}\mathcal{F}^* \mathds{1}_{X+B(0,1/4)}\|_{L^2\to L^2}+ N_2^{\frac{1+\alpha}{2\alpha}}(rN_2)^{-m}        \end{equation}
    where $\mathcal{F}f = \hat{f}$ is the standard Fourier transform on $L^2(\mathbb{R}^2)$.
\end{proposition}

The proof of Proposition \ref{prop:General-CD} uses the following pointwise inequality \eqref{eq:modified-pwIneq},
which roughly says that a function with compact Fourier support is bounded pointwise by its local $L^2$ mass times the size of its Fourier support.
This follows from the classical (quantitative) uncertainty principle, which states that a function with compact Fourier support is locally constant on the dual scale, up to an error. 
We remark that the corresponding inequality \cite[(33)]{Cohen} in Cohen's paper is not true; it misses the error term.
We believe that the loss of a polynomial in $(rN_2)$ in \eqref{eq:CDgeneral} is inevitable, perhaps up to a logarithmic factor. 
However, this loss does not harm the proof of 
Theorem \ref{thm:C-D}.
\begin{lemma}[Pointwise estimate]\label{lem:CD-pw-bound}
    Let $f\in L^2(\mathbb{R}^2)$ with 
    $\operatorname{supp} \hat{f}\subset [-N_1,N_1]\times [-N_2,N_2]$. 
    Then for all $r\in (10/N_2,1/10)$, all $m>0$, and all $\mathbf{x}\in \mathbb{R}^2$, we have
    \begin{equation}\label{eq:modified-pwIneq}
         |f(\mathbf{x})|^2\lesssim_{\alpha,m} 
        N_1N_2\bigl(\|f\|^2_{L^2(B_r^\alpha(\mathbf{x}))}+(N_2r)^{-m}\|f\|_{L^2({B_r^{\alpha}(\mathbf x)}^{\mathsf c})}^2\bigr).
    \end{equation}
\end{lemma}

\begin{remark}
    In the application to the proof of Theorem \ref{thm:C-D}, we will choose $r\sim N_2^{-1+\epsilon}$ for an arbitrarily small $\epsilon$ with $0<\epsilon\sim \beta-\beta'\ll 1$, and an arbitrarily large integer $m$ such that $m\gg \epsilon^{-1}$.
\end{remark}

\begin{proof}    
The key observation is that a function with compact Fourier support is locally constant at the inverse scale of the Fourier support. 
Quantitatively, we choose $\omega\in C^\infty(\mathbb{R}^2)$ 
to be a Schwartz function such that $\hat{\omega}$ is a bump function in $C_c^\infty((-2,2)^2)$, 
and $\hat{\omega}=1$ on $(-1,1)^2$. Then for a function $f\in L^2(\mathbb{R}^2)$ with $\hat{f}$ supported in $[-N_1,N_1]\times [-N_2,N_2]$ 
we have by the Fourier inversion formula
\[
f = f * \omega_{\mathbf N}, \qquad \omega_{\mathbf N}(\mathbf x) := N_1 N_2 \, \omega(\mathbf N \mathbf x).
\]
For every $r \in (10/N_2, 1/10)$, we split the convolution into two parts:
\[
f(\mathbf{x}) = \left( \int_{d_{\alpha,\mathbb{R}^2}(\mathbf{x},\mathbf{y})\le r} 
+ \int_{d_{\alpha,\mathbb{R}^2}(\mathbf{x},\mathbf{y}) > r} \right)
f(\mathbf{y})\, \omega_{\mathbf N}(\mathbf{x}-\mathbf{y}) \, d\mathbf{y}.
\]
By the Cauchy–Schwarz inequality we obtain
\[
|f(\mathbf{x})| \lesssim \|f\|_{L^2(B_r^\alpha(\mathbf{x}))}\|\omega_{\mathbf N}\|_{L^2(B_r^{\alpha})}
+ \|f\|_{L^2({B_r^{\alpha}(\mathbf x)}^{\mathsf c})}
\|\omega_{\mathbf N}\|_{L^2({B_r^{\alpha}}^{\mathsf c})}.
\]   
By Plancherel's theorem, it is easy to verify
\[
\|\omega_{\mathbf N}\|_{L^2} = (N_1 N_2)^{1/2} \|\omega\|_{L^2} \lesssim (N_1 N_2)^{1/2}.
\]
This bounds the first term. For the second term, we need a decay estimate for $\omega_{\mathbf N}$. 
Using the Schwartz decay of $\omega$ and a dyadic spherical shell decomposition, for arbitrarily large $m$ the integral of $\omega_{\mathbf N}$ outside the ball $B_r^\alpha$ can be estimated as
\[
\begin{aligned}
\|\omega_{\mathbf N}\|_{L^2({B_r^{\alpha}}^{\mathsf c})}^2
&= N_1 N_2 \int_{(B_{N_2 r}^\alpha)^{\mathsf c}} |\omega(\mathbf{y})|^2 \, d\mathbf{y}\\
&\lesssim_{\alpha} N_1 N_2 \sum_{k\ge 0} |B_{2^k N_2 r}^\alpha| \sup_{\mathbf{y}\in B_{2^k N_2 r}\setminus B_{2^{k-1} N_2 r}} 
|\omega(\mathbf{y})|^2\\
&\lesssim_{\alpha,m} N_1 N_2 \sum_{k\ge 0} 2^{-k(m-1-1/\alpha)} (N_2 r)^{-m+1+1/\alpha}\\
&\lesssim_{\alpha,m} N_1 N_2 (N_2 r)^{-m+1+1/\alpha}.
\end{aligned}
\]
The last inequality holds provided that $m > 1+1/\alpha$.
This completes the proof.
\end{proof}

\begin{proof}[Proof of Proposition \ref{prop:General-CD}] 

Now we are ready to prove the general control principle for discrete FUP using continuous FUP. This proof is a modification of \cite[Appendix B.4]{Cohen}, adapted to the anisotropic setting. We want to use Lemma \ref{lem:CD-pw-bound}, so we need to construct a function $f\in L^2(\mathbb{R}^2)$ with $\hat{f}$ supported near $\mathcal X$ in $[-N_1,N_1]\times [-N_2,N_2]$, which corresponds well to the discrete function whose Fourier transform is supported on $\mathcal X$.

 Let $u:{\mathbb Z_{\mathbf N}}\to \mathbb{C}$ satisfy $\operatorname{supp}\mathcal{F}_{\mathbf N} u\subset \mathcal{X}$ and $\|u\|_{\ell^2}=1$. We construct an auxiliary function $f\in L^2(\mathbb{R}^2)$ based on $u$ as follows. 
Let $\chi\in C_c^\infty(B(0,1/4))$ satisfy
\begin{equation}\label{eq:36 in Cohen}
\chi^\vee (\mathbf{x})\geq 1,\quad \mathbf{x}\in [-10,10]^2.
\end{equation}
Define
\[
\hat{f}(\mathbf{\xi})=\sum_{\mathbf{\xi}'\in \mathcal X} 
\mathcal{F}_{\mathbf N}u(\mathbf{\xi}')\, 
\chi(\mathbf{\xi}-\mathbf{\xi}') \qquad \forall\, \xi \in \mathbb R^2.
\]
Since $\chi$ has small support, essentially only one $\xi'$ in $\mathcal X$ takes effect in the sum. We note that $\hat{f}$ is supported in $[-N_1,N_1]\times [-N_2,N_2]$. 
By the unitarity of the Fourier transform,
\begin{equation}\label{eq:f L2 bounded by u l2}
\|f\|_{L^2}^2\lesssim \|u\|_{\ell^2}^2=1.
\end{equation}
For $\mathbf{x}\in Y$, the inverse Fourier transform on $\mathbb R^2$ and $\mathbb Z_{\mathbf N}$ gives
\[
f(\mathbf{x})=\sum_{\xi' \in \mathcal{X}}
\mathcal{F}_{\mathbf N}u(\xi') e^{2\pi i \mathbf{x} \cdot \xi'} \chi^\vee(\mathbf{x})= \sqrt{N_1N_2}\, u(\mathbf N \mathbf x) \,\chi^\vee(\mathbf{x}).
\]           
Using the pointwise bound \eqref{eq:36 in Cohen}, this implies
\[\|u\mathds{1}_{\mathcal{Y}}\|_{\ell^2}^2 \lesssim 
(N_1N_2)^{-1} \sum_{\mathbf{x} \in Y} |f(\mathbf{x})|^2.\]
By Lemma \ref{lem:CD-pw-bound}, for every $r\in (10/N_2,1/10)$ and every arbitrarily large $m\in \mathbb N$, we have the pointwise bound \eqref{eq:modified-pwIneq} for $f$. Summing over $\mathbf{x}\in Y$ yields
\begin{align}
\|u\mathds{1}_{\mathcal{Y}}\|_{\ell^2}^2 &\lesssim \sum_{\mathbf{x}\in Y} \|f\|^2_{L^2(B_r^\alpha(\mathbf{x}))}+ (N_2r)^{-m} \sum_{\mathbf{x}\in Y} \|f\|_{L^2({B_r^{\alpha}(\mathbf x)}^{\mathsf c})}^2\nonumber\\
&\lesssim \max_{\mathbf{x}\in \mathbb{R}^2}|B_r^\alpha(\mathbf x)\cap Y|\cdot\|f\mathds{1}_{Y+B_r^\alpha}\|_{L^2}^2+|Y|
(N_2r)^{-m}\|f\|_{L^2}^2\nonumber\\
&\lesssim (rN_2)^{1+1/\alpha}\|f\mathds{1}_{Y+B_r^\alpha}\|_{L^2}^2+
N_2^{1+1/\alpha}(N_2r)^{-m}\|f\|_{L^2}^2 \label{eq:u-bound-by-f}
\end{align}
where in the last line we use the fact that $Y$ is $1/\mathbf N$-separated and
\[
|Y\cap B^\alpha_{r}(\mathbf{x})| \lesssim_{\alpha} (rN_2)^{1+1/\alpha}.
\]    
Since $\operatorname{supp} \hat{f}\subset X+B(0,1/4)$, combining \eqref{eq:f L2 bounded by u l2} and \eqref{eq:u-bound-by-f} we obtain
\[
\begin{aligned}
\|u\mathds{1}_{\mathcal{Y}}\|_{\ell^2}^2 
&\lesssim_{\alpha,m} (rN_2)^{1+1/\alpha}\, \|\mathds{1}_{Y+B^\alpha_r}\mathcal{F}^*\mathds{1}_{X+B(0,1/4)}\|_{L^2(\mathbb{R}^2)\to L^2(\mathbb{R}^2)}^2 
+ N_2^{1+1/\alpha}(N_2r)^{-m},\\
\|u\mathds{1}_{\mathcal{Y}}\|_{\ell^2}
&\lesssim_{\alpha,m} (rN_2)^{\frac{1+\alpha}{2\alpha}}\, \|\mathds{1}_{Y+B^\alpha_r}\mathcal{F}^*\mathds{1}_{X+B(0,1/4)}\|_{L^2(\mathbb{R}^2)\to L^2(\mathbb{R}^2)}
+ N_2^{\frac{1+\alpha}{2\alpha}}(N_2r)^{-m}.
\end{aligned}
\]
This yields the general bound \eqref{eq:CDgeneral} as desired.
\end{proof}

The general result yields the discrete FUP from the continuous FUP estimates. Note that estimates \eqref{eq:ani-cont-fup} and \eqref{eq:discrete-fup-from-cfup} in Theorem \ref{thm:C-D} are equivalent to their dual forms for the dual operators $\mathcal{F}_h^{\operatorname{ani*}}$ and $\mathcal{F}_{\mathbf N}^*$.

\begin{proof}[Proof of Theorem \ref{thm:C-D}]

 We define the unitary map $\Phi^\alpha_h:L^2(\mathbb{R}^2)\to L^2(\mathbb{R}^2)$ by
    \[
        \Phi^\alpha_h f(x_1,x_2)=h^{-(1+\alpha)/2} f(x/h,y/h^\alpha).
    \]
    Then we have
    \[
        \begin{aligned}
            \|\mathds{1}_{\mathbf{Y}+B_{h^\alpha}^\alpha}\mathcal{F}_h^{\operatorname{ani*}}\mathds{1}_{\mathbf{X}+B_{h^\alpha}^\alpha}\|_{L^2\to L^2}
            &=\|\mathds{1}_{\mathbf{Y}+B_{h^\alpha}^\alpha}\mathcal{F}_h^{\operatorname{ani*}}\mathds{1}_{\mathbf{X}+B_{h^\alpha}^\alpha}
             \Phi^\alpha_h\|_{L^2\to L^2}\\
            &=\|\mathds{1}_{\mathbf{Y}+B_{h^\alpha}^\alpha}\mathcal{F}_h^{\operatorname{ani*}}\Phi^\alpha_h\,
            \mathds{1}_{\tilde{\mathbf{X}}_h}\|_{L^2\to L^2}\\
            &=\|\mathds{1}_{\mathbf{Y}+B_{h^\alpha}^\alpha}
            \mathcal{F}^*\mathds{1}_{\tilde{\mathbf{X}}_h}\|_{L^2\to L^2},
        \end{aligned}
    \]
    where $\mathcal{F}$ is the usual Fourier transform on $\mathbb R^2$, and $\tilde{\mathbf{X}}_h$ is a dilation and translation of $\mathbf X$ defined by
    \[
        \tilde{\mathbf{X}}_h=\mathbf X_h+B_1^{\alpha}, \quad \mathbf X_h={\mathbf X}/(h,h^\alpha).
    \]    
By the duality of translation and modulation in the dual spaces of position and frequency,
    our assumption of continuous FUP \eqref{eq:ani-cont-fup} implies that for 
    any $\mathbf{y}\in \mathbb{R}^2$,
    \begin{equation}\label{eq:ffup}
        \|\mathds{1}_{\mathbf{Y}+y+B_{h^\alpha}^\alpha}
            \mathcal{F}^*\mathds{1}_{{\mathbf{X}}_h+[-10,10]^2}\|_{L^2\to L^2}\leq Ch^\beta.
    \end{equation}

    As in Proposition \ref{prop:General-CD}, denote in $\mathbb R^2$ the corresponding discrete sets 
    \[
        X_k=\mathcal X_k,\quad Y_k=\mathcal Y_k/\mathbf N \subset \mathbb{R}^2.
    \]  
    Let $h:=1/N_1=M_1^{-k}\to 0$ as $k\to \infty$. Recall the set
    \[
        \mathbf{X}=\left\{\left(\sum_{j\geq 1} a_jM_1^{-j},\,\sum_{j\geq 1} b_jM_2^{-j}\right):
        (a_j,b_j)\in \mathcal{A}\right\}.
    \]
Thus we obtain
        \begin{align}\label{eq:XkboundbyXh}
            X_k+B(0,1/4)&\subset 
            \left\{\left(\sum_{j=1}^k a_jM_1^{k-j},\sum_{j=1}^k 
        b_jM_2^{k-j}\right):
        (a_j,b_j)\in \mathcal{A}\right\}+[-5,5]^2\nonumber\\
        &\subset 
        {X}_h+[-10,10]^2.
        \end{align}
    On the other hand, 
    we can choose $r=h^{\alpha(1-\epsilon)} > N_2^{-1}$ for any small $\epsilon>0$, and thus 
    \[
        Y_k+B_r^\alpha
        \subset \mathbf{Y}+B^\alpha_{1/N_2} + B_r^\alpha
        \subset \mathbf{Y}+B_{2r}^\alpha.
    \]
    We can choose a set $W\subset \mathbb{R}^2$ with $|W|\leq 100h^{-\epsilon (1+\alpha)}$ so that 
    \[
        \mathbf{Y}+B_{2r}^\alpha \subset \bigcup_{y\in W} \bigl(\mathbf{Y}+y+B_{h^\alpha}^\alpha\bigr),
    \]
    which gives 
  \begin{equation}\label{eq:YkboundbyY}
        Y_k+B_r^\alpha
        \subset \bigcup_{y\in W} \bigl(\mathbf{Y}+y+B_{h^\alpha}^\alpha\bigr).
    \end{equation}
    Thus \eqref{eq:ffup} together with \eqref{eq:XkboundbyXh} and \eqref{eq:YkboundbyY} implies that 
    \[
        \|\mathds{1}_{Y_k+B^\alpha_r}\mathcal{F}^*\mathds{1}_{X_k+B(0,1/4)}\|_{L^2\to L^2}
        \leq Ch^{\beta-\epsilon (1+\alpha)}
    \]
    for any small $\epsilon>0$.
    Therefore, the desired discrete FUP \eqref{eq:discrete-fup-from-cfup} with $\beta-\beta'\sim \epsilon$ follows from 
    \eqref{eq:CDgeneral} provided we choose $m\gg \epsilon^{-1}$ large enough, since
    \begin{align*}   
        \|\mathds{1}_{\mathcal{Y}_k}\mathcal{F}_{\mathbf N}^*\mathds{1}_{\mathcal{X}_k}\|_{\ell^2\to\ell^2}
        &\lesssim_{m,\alpha} (rN_2)^{\frac{1+\alpha}{2\alpha}} \|\mathds{1}_{Y_k+B^\alpha_r}\mathcal{F}^* \mathds{1}_{X_k+B(0,1/4)}\|_{L^2\to L^2}+   N_2^{\frac{1+\alpha}{2\alpha}}(N_2r)^{-m}  \\
        &\lesssim_\alpha h^{\beta-\epsilon\, \frac{3(1+\alpha)}2} +h^{-\epsilon  \alpha m-\frac{1+\alpha}2}     
        \lesssim_{\beta', \beta,\alpha} h^{\beta'}=M_1^{-k\beta'}.
    \end{align*}

\end{proof}

\subsection{Anisotropic continuous FUP for $\mathcal F_h$}\label{subsec:cfup-isoFourierTransform}
We use the method in \cite{dyatlov2024semiclassical} to prove an FUP for \emph{anisotropic porous} sets. We believe this property is a feature of a wide class of Bedford–McMullen carpets (see Example \ref{ex:BMcarpet-porous}).

\begin{definition}[Anisotropic porosity]\label{def:porous}
    Let $\alpha\in(0,1)$, $\gamma\in (0,1)$. Let $v_1, v_2\in \mathbb R^2$ be two linearly independent vectors. 
    For $X\subset \mathbb{R}^2$, we say $X$ is 
    \emph{$\alpha$-anisotropic $\gamma$-porous of scale $[l_1,l_2]$ along directions $(v_1, v_2)$} 
    if for each segment $I\subset \mathbb{R}^2$ of length $l\in [l_1,l_2]$ that is parallel to $v_1$, 
    we can always find a point $x\in I$ such that $X$ does not intersect the following parallelogram 
    centered at $x$:
    \[
    \operatorname{PG}_{\gamma,\alpha,l}(x, v_1,v_2) = \left\{x+t_1 v_1 + t_2 v_2 : t_1\in \left[-\frac{\gamma l}{2},\frac{\gamma l}{2}\right],\; t_2\in \left[-\frac{\gamma l^\alpha}{2},\frac{\gamma l^\alpha}{2}\right]\right\}.
    \]

\end{definition}

\begin{remark}
The definition applies to \emph{any positive} anisotropic parameter $\alpha$. However, as before, we always let $\alpha\in(0,1)$ and assume the $\alpha$-anisotropic porous set is 
of \emph{microscopic} scale, that is, $l_2\lesssim 1$ and $l_1\ll 1$;
simultaneously we will assume the $1/\alpha$-anisotropic porous set 
is of \emph{macroscopic} scale, that is, 
$l_1\gtrsim 1$ and $l_2\gg 1$. This assumption guarantees that 
the parallelogram in the definition has a larger extent along $v_2$ than along $v_1$.
\end{remark}

\begin{figure}[htbp]
    \centering
    \includegraphics[width=0.8\textwidth]{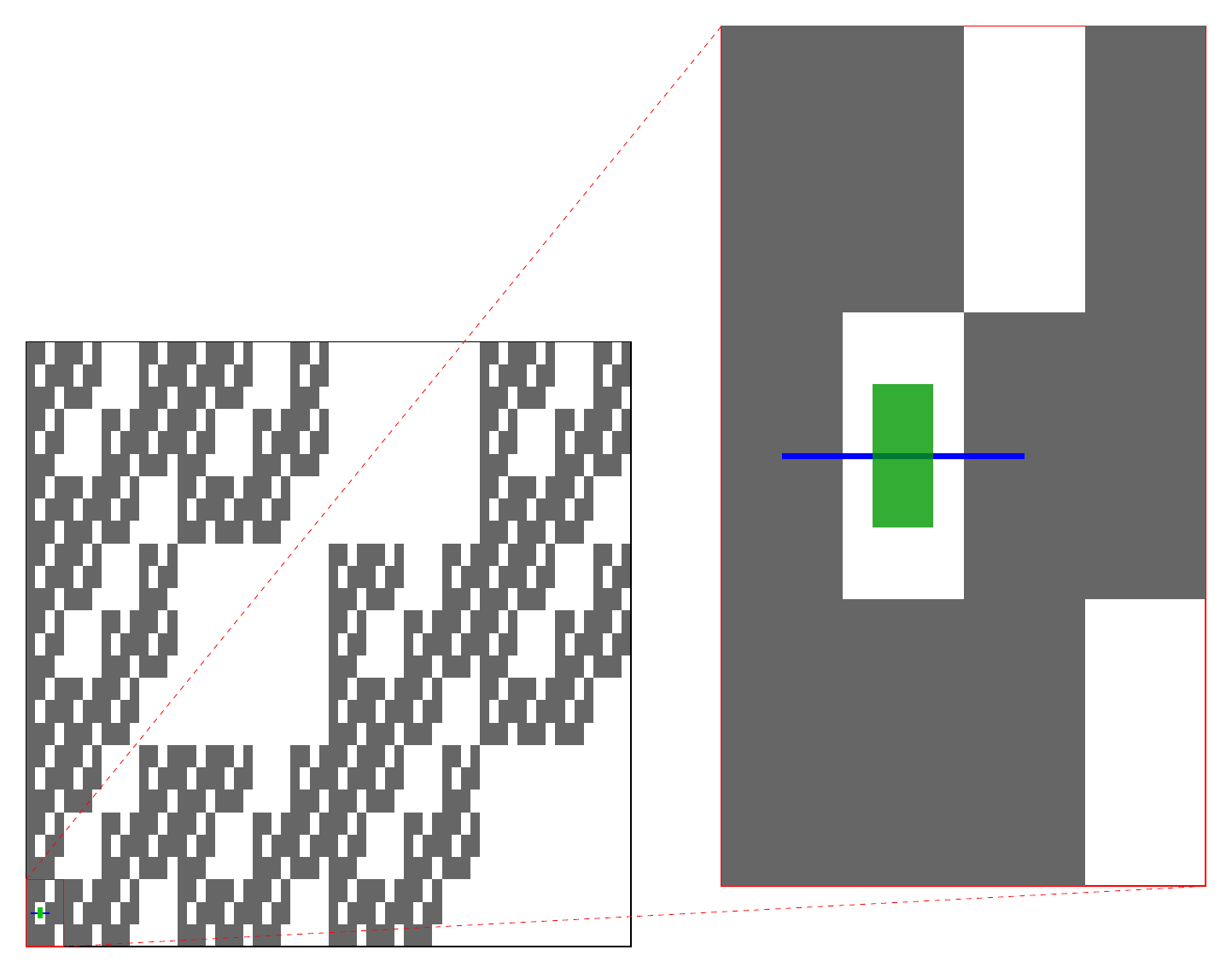}
    \caption{Anisotropic porosity for Bedford-McMullen carpets 
    with $M_1=4,M_2=3$ and $\mathcal{A}=\mathbb{Z}_{M_1}\times \mathbb{Z}_{M_2}
    \setminus\{(3,0),(1,1),(2,2)\}$. The whole blue segment $I$  is of scale $\sim 4^{-2}$ with a sub-segment $I'$ 
lying in the green rectangle of scale $\sim 4^{-3}\times 3^{-3}$.} 
    \label{fig:Anisotropic_porosity_for_BM_carpet}
\end{figure}

\begin{example}[Anisotropic porous BM carpet]\label{ex:BMcarpet-porous}
Let $\mathbf{X}\subset \mathbb{R}^2$ be the Bedford-McMullen carpet iterated from base $\mathbf M$ and alphabet $\mathcal{A}\subset \mathbb Z_{\mathbf M}$,  where $M_2=M_1^\alpha$ for some $\alpha\in(0,1)$.  
Assume that \emph{$\mathcal{A}$ is not full in each row;} that is, for each $j_2\in \mathbb Z_{M_2}$ there exists some $j_1\in \mathbb Z_{M_1}$ such that $(j_1,j_2)\notin \mathcal{A}$. 
Then $\mathbf{X}$ is $\alpha$-anisotropic $\gamma$-porous of scale $(0,1]$ along the directions $(e_1,e_2)$, where $e_1=(1,0)$ and $e_2=(0,1)$ 
are the standard basis vectors of $\mathbb{R}^2$, and $\gamma=\gamma(M_1,M_2)>0$ is a constant depending only on $M_1$ and $M_2$. 
More concretely, consider any segment $I$ parallel to $e_1$ with length $|I|\in (M_1^{-k_0-1}, M_1^{-k_0}]$ 
for some $k_0\in\mathbb N$ and with height $y_0\in(0,1]$. 
Then we can find some integer $a_0$ and define a sub-segment $I'$ as 
\[
    I':=[a_0M_1^{-k_0-2},(a_0+1)M_1^{-k_0-2}]\times \{y_0\}
\]
so that $I'\subset I$. We can also find the integer $b_0$ so that 
\[
    y_0\in [b_0M_1^{-k_0-2},(b_0+1)M_2^{-k_0-2})
\]
After a rescaling $\Phi:(x,y)\mapsto (M_1^{k_0+2},M_2^{k_0+2}y)$, we know 
the set 
\begin{equation}\label{eq:The rescaled set in ex:BMcarpet-porous}
    \Phi_k([a_0M_1^{-k_0-2},(a_0+1)M_1^{-k_0-2}]\times [(b_0-1)M_1^{-k_0-2},(b_0+1)M_2^{-k_0-2}]\cap \mathbf{X})
\end{equation}
is contained in a $1\times 3$ rectangle. This $1\times 3$ rectangle consists of 
three $1\times 1$ cubes, and the intersection with 
the interior of each cube and the set is either empty, or 
equals to $\mathbf{X}$ after some translation. 
By the ``no full row" assumption, it is easy to find a cube $Q$ of 
length $\sim 1$ with the same center as that of $\Phi(I')$, satisfying 
\[  
    \begin{aligned}
        Q\subset \Phi_k([a_0M_1^{-k_0-2},(a_0+1)M_1^{-k_0-2}]\times [(b_0-1)M_1^{-k_0-2},(b_0+1)M_2^{-k_0-2}])
    \end{aligned}
\]
and $Q$ has empty intersection with the set defined in \eqref{eq:The rescaled set in ex:BMcarpet-porous}.
Thus $\Phi_k^{-1}(Q)$ is the desired parallelogram 
in the $\alpha$-anisotropic definition. See Figure \ref{fig:Anisotropic_porosity_for_BM_carpet}.
\end{example}

We also list an interesting example of rescaled carpet which is anisotropic porous in large scale along $(e_2,e_1)$. This will be used in next subsection. 

\begin{example}[Rescaled carpet is macroscopic porous]\label{ex:rescaled BM-carpet 1/alpha porous}
    Let $\mathbf{Y}\subset \mathbb{R}^2$ be the Bedford-McMullen carpet iterated from 
    base $\mathbf M$ and alphabet $\mathcal{B}\subset \mathbb Z_{\mathbf M}$,  where $M_2=M_1^\alpha$ for some $\alpha\in(0,1)$.  
    Assume now that \emph{$\mathcal{B}$ is not full in each column;} 
    that is, for each $j_1\in \mathbb Z_{M_1}$ there exists some $j_2\in \mathbb Z_{M_2}$ such that $(j_1,j_2)\notin \mathcal{B}$.
    Let $\tilde{\mathbf{Y}}\subset \mathbb{R}^2$ defined as 
    \[
        \tilde{\mathbf{Y}}=\{(x/h,y/h^\alpha):(x,y)\in \mathbf{Y}\}
    \]
    Then $\tilde{\mathbf Y}$ is $1/\alpha$ anisotropic $\gamma$-porous of scale $[1,+\infty)$ 
    along the directions $(e_2,e_1)$, where $e_1=(1,0)$ and $e_2=(0,1)$ are the standard basis vectors of $\mathbb{R}^2$, 
    and $\gamma=\gamma(M_1,M_2)>0$ is a constant depending only on $M_1$ and $M_2$. 
    Actually, by the argument in the previous example, there exists 
    some $\gamma>0$ so that, given any segment $I$ with length 
    $l\in [h^{\alpha},+\infty)$ parallel to $e_2$, we know 
    that we can find some $I'\subset I$ of length $\gamma l$ such that the parallelogram
    \[
        \{\mathbf{p}+t\gamma l^{1/\alpha}e_1:\mathbf{p}\in I', t\in [-1/2,1/2]\}
    \]
    has empty intersection with $\mathbf{Y}$. After the rescaling $(x,y)\mapsto (x/h,y/h^\alpha)$, 
    we know this parallelogram becomes another parallelogram of length $\left(\frac{l}{h^\alpha}\right)^{1/\alpha}$
    along $e_1$ and of length $l/h^\alpha$ along $e_2$, which has non-empty intersection with $\tilde{\mathbf{Y}}$.
\end{example}

\begin{remark} 
The anisotropic porosity we introduce, either macroscopic or microscopic,  are  \emph{line} porosities, 
    which are different from Cohen's line porosity and stronger than ball porosity.
    Actually, there is no containment relation between our anisotropic line porosities and Cohen's. 
    Some properties of such anisotropic porous sets are listed in Appendix \ref{app:ani-porous-transf} whose proofs are nearly identical to those for Cohen's line porosity (see \cite[Appendix A.2]{MR4927737}). 
    In particular, we also show in the Appendix \ref{app:porositys} that when the alphabet $\mathcal A$ is nonempty for each row, 
    Cohen's line porosity does not hold for the associated Bedford--McMullen carpet $\mathbf X$; 
    see (a) of Proposition \ref{prop:Failure of line/ball porosity for BM carpet}. 
    Thus, we cannot apply Cohen's FUP \cite[Theorem 1.2]{MR4927737} directly to these sets. We also consider ball porosity and rescaled carpets there.
\end{remark}

\begin{remark}
The anisotropic porosity 
    can be generalized to higher dimensions: one replaces $v_2$ with a hyperplane of 
    codimension-1 transversal to $v_1$, and the parallelogram with a cylinder.
 \end{remark}

The main result of this subsection is the following anisotropic continuous FUP. 
\begin{theorem}[CFUP]\label{thm:cfup}
Let $\gamma\in(0,1)$, $\alpha,\alpha'\in(0,1)$, and $0<h<h(\gamma,\alpha,\alpha')\ll 1$, and let $(v_1,v_2)$ and $(v_1',v_2')$ be two pairs of linearly independent vectors in $\mathbb{R}^2$.
Let $X$ be an $\alpha$-anisotropic $\gamma$-porous set of scale $[h,1]$ along the direction $({v}_1,{v}_2)$, and let $Y$ be an $\alpha'$-anisotropic $\gamma$-porous set of scale $[h,1]$ along the direction $({v}_1',{v}_2')$, such that ${v}_1'\perp {v}_2$ and ${v}_2'\perp {v}_1$. 
Then there exists $\beta>0$ such that 
\begin{equation}\label{eq:cfup}
\|\mathds{1}_X \mathcal{F}_h \mathds{1}_Y\|_{L^2\to L^2}\leq C h^\beta,
\end{equation}
where $\mathcal{F}_h$ is the (isotropic) semiclassical Fourier transform defined by \eqref{eq:semi-F}.
\end{theorem}

\begin{remark} Theorem \ref{thm:cfup-bm} is a special case of Theorem \ref{thm:cfup}. 
By Example \ref{ex:BMcarpet-porous}, the limiting sets $\mathbf X, \mathbf Y$ in Theorem \ref{thm:cfup-bm} are anisotropic porous. Then \eqref{eq:cfup} implies \eqref{eq:cfup-BD}. The non-fullness condition for alphabets in Theorem \ref{thm:cfup-bm} is stated for rows, because the segments $I$ in the porosity condition are parallel to the dominant direction $v_1$.  
\end{remark}

\begin{remark}
This FUP and the following argument can be adapted to more general high-dimensional anisotropic porous sets involving hyperplanes and cylinders, but we do not pursue it here, as we believe the two-dimensional case captures the main idea and is already sufficiently interesting.
\end{remark}

Before the proof, we recall some useful concepts. Let \(a(x,\xi)\) be a function (symbol) defined on the phase space \(\mathbb{R}^{2n}=\{(x,\xi)\}\), usually required to be sufficiently smooth and to satisfy appropriate growth conditions. The \emph{Weyl quantization} \(\operatorname{Op}_h^w(a)\) acts on functions \(u\in\mathcal{S}(\mathbb{R}^n)\) (the Schwartz space) by
\begin{equation}\label{eq:weylquan}
\bigl(\operatorname{Op}_h^w(a) u\bigr)(x) 
= \frac{1}{(2\pi h)^n} \int_{\mathbb{R}^n}\int_{\mathbb{R}^n} 
e^{\frac{i}{h}\,\langle x-y,\,\xi\rangle}\, 
a\!\left(\frac{x+y}{2},\xi\right) u(y) \, dy \, d\xi,
\end{equation}
where \(\langle\cdot,\cdot\rangle\) denotes the standard inner product in \(\mathbb{R}^n\). The integral is understood as an oscillatory integral. The definition can be extended to general symbols 
\[
a\in S(1)=\bigl\{ a\in C^\infty(\mathbb{R}^{2n}) : \sup_{x,\xi} |\partial_{x,\xi}^\alpha a(x,\xi)|<\infty \ \text{ for every multi-index } \alpha \bigr\}.
\]
For any symplectic matrix \(A\in \operatorname{Sp}(2n,\mathbb{R})\) we can define the \emph{metaplectic operator} \(\mathcal{M}(A): L^2(\mathbb{R}^n)\to L^2(\mathbb{R}^n)\) as a unitary transformation satisfying the following exact Egorov property:
\begin{equation}\label{eq:egorov}
\operatorname{Op}_h^w(a)\,\mathcal{M}(A)=\mathcal{M}(A)\,\operatorname{Op}_h^w(a\circ A)\qquad\forall\,a\in S(1).
\end{equation}
Such transformations satisfying the Egorov property exist and are unique up to multiplication by a unit complex number. The collection of all of them, called the \emph{metaplectic group}, forms a subgroup of the group of unitary transformations. For details, see, e.g., the semiclassical monograph by Zworski \cite{MR2952218}.

\begin{proof}[Proof of Theorem \ref{thm:cfup}]
 We can assume $\alpha=\alpha'<1$.
This proof is a modification of 
\cite[Section 4.3]{dyatlov2024semiclassical}.
We first pick $\epsilon=\epsilon(\alpha)>0$ such that
\[
0< \rho:=\Bigl(\frac{1}{2}+\epsilon\Bigr)\,\alpha<\frac12,\qquad \frac12<\rho+\frac{1}{2}+\epsilon<1.
\]

\emph{Step 1. Normalization.} Note that the orthogonality of direction pairs 
implies that we can find an invertible matrix $Q\in \operatorname{GL}(2,\mathbb{R})$ and a 
symplectic matrix $A\in \operatorname{Sp}(4,\mathbb{R})$ of the form
\[
A = \begin{pmatrix} Q & 0 \\ 0 & (Q^{-1})^T \end{pmatrix}
\]
such that for some $c_1,c_2>0$
\[
\begin{aligned}
&A( {v}_1',0,0)=c_1\partial_x,\quad &A( {v}_2',0,0)=c_2\partial_y,\\
&A(0,0, {v}_1)=\partial_\xi,\quad &A(0,0, {v}_2)=\partial_\eta,
\end{aligned}
\]
where $(x,y,\xi,\eta)$ is the standard coordinate on $T^*\mathbb{R}^2\simeq \mathbb{R}^2\times 
\mathbb{R}^2$.    
We view $Y\subset \mathbb{R}^2_{x,y}$ and $X\subset \mathbb{R}^2_{\xi,\eta}$; thus we can write 
\[
\mathds{1}_X\mathcal{F}_h\mathds{1}_Y=
\operatorname{Op}_h^w(\mathds{1}_{\mathbb{R}^2_{x,y}\times X})
\operatorname{Op}_h^w(\mathds{1}_{Y\times \mathbb{R}^2_{\xi,\eta}}),
\]
where $\operatorname{Op}_h^w$ is the Weyl calculus \eqref{eq:weylquan}.
We choose the 
metaplectic operator $\mathcal{M}(A)
:L^2(\mathbb{R}^2)\to L^2(\mathbb{R}^2)$ associated with the symplectic transform $A$, 
which is unitary and satisfies the exact Egorov property \eqref{eq:egorov}.
Thus we have
\[
\begin{aligned}
\|\mathds{1}_X\mathcal{F}_h\mathds{1}_Y\|
&=\|\operatorname{Op}_h^w(\mathds{1}_{\mathbb{R}^2_{x,y}\times X})
\operatorname{Op}_h^w(\mathds{1}_{Y\times \mathbb{R}^2_{\xi,\eta}})\mathcal{M}(A)\|\\
&=\|\operatorname{Op}_h^w(\mathds{1}_{A(\mathbb{R}^2_{x,y}\times X)})
\operatorname{Op}_h^w(\mathds{1}_{A(Y\times \mathbb{R}^2_{\xi,\eta})})\|.
\end{aligned}
\]
By (b) of Proposition \ref{prop:porous small scale}, we can assume that 
\[
({v}_1',{v}_2')=(\partial_x,\partial_y),\qquad ({v}_1,{v}_2)=(\partial_\xi,\partial_\eta).
\]

\emph{Step 2. Decomposition and almost-orthogonality.}   Now we consider the $h$-neighborhoods of $X, Y$:
\[
\begin{aligned}
X_h&:=X+[-h^{1/2+\epsilon},h^{1/2+\epsilon}]_{\xi}\times [-h^{(1/2+\epsilon)\alpha},h^{(1/2+\epsilon)\alpha}]_{\eta},\\
Y_h&:=Y+[-h^{1/2+\epsilon},h^{1/2+\epsilon}]_{x}\times [-h^{(1/2+\epsilon)\alpha},h^{(1/2+\epsilon)\alpha}]_{y}.
\end{aligned}
\]
By (a) of Proposition \ref{prop:porous small scale}, $X_h$ and $Y_h$ are $(\alpha$-anisotropic) $\gamma/5$-porous of 
scale $[(5/\gamma)^{1/\alpha}h^{1/2+\epsilon},1]$ along the directions $(\partial_\xi,\partial_\eta)$ and $(\partial_x,\partial_y)$, respectively.  

Choose $\chi_X\in C_c^\infty(X_h)$ which equals one on $X$ and $\chi_Y\in C_c^\infty(Y_h)$ which equals one on $Y$ 
such that for all $(j_1,j_2)\in \mathbb N^2$,
\begin{equation}\label{eq:estimate_chix_chiy}
\partial_{\xi}^{j_1}\partial_{\eta}^{j_2}\chi_X,\;
\partial_{x}^{j_1}\partial_{y}^{j_2}\chi_Y
= \mathcal{O}_{j_1,j_2}\bigl(h^{-j_2\rho -j_1\rho/\alpha}\bigr).
\end{equation}
Then for \eqref{eq:cfup} it suffices to show that for some $\beta>0$,
\begin{equation}\label{eq:chixchiy-boundbybeta}
\|\chi_X(hD)\,\chi_Y\|
\lesssim h^\beta.
\end{equation}

We next apply a partition of unity along the $y$-direction and 
the $\eta$-direction. Fix a 
function $\tilde{\psi}\in C_c^\infty((-1,1),[0,1])$ such that 
\[
\sum_{j\in \mathbb{Z}} \tilde{\psi}(x-j)\equiv 1,
\]
then define the functions $\psi_j\in C_c^\infty(\mathbb{R}^2)$ as 
\[
\psi_j(x,y):=\tilde{\psi}\Bigl(\frac{y}{h^\rho}-j\Bigr),\qquad j\in \mathbb Z,
\]
and the operators $P_{jl}:L^2(\mathbb{R}^2)\to 
L^2(\mathbb{R}^2)$ as 
\[
P_{jl}:=\chi_X(hD)\,\psi_j(hD)\, \psi_l\, \chi_Y,\qquad j,l\in \mathbb Z.
\]
Then we have 
\[
\chi_X(hD)\,\chi_Y
=\sum_{j,l\in \mathbb{Z}} P_{jl}.
\]
In view of the Cotlar–Stein lemma, for \eqref{eq:chixchiy-boundbybeta} it suffices to show that for some $\beta>0$
\begin{equation}\label{eq:Desired estimate for Pkl}
\|P_{jl}\|_{L^2\to L^2}\leq Ch^\beta,\quad \forall\, j,l\in \mathbb{Z},
\end{equation}
and for $|j_1-j_2|+|l_1-l_2|\geq 100$ and $m\geq 0$,
\begin{equation}\label{eq:Desired estimate for Pk1l1Pk2l2}
\begin{aligned}
\|P_{j_1l_1}^*P_{j_2l_2}\|_{L^2\to L^2}&\leq
C_m(1+|j_1-j_2|+|l_1-l_2|)^{-m}h^m,\\
\|P_{j_1l_1}P_{j_2l_2}^*\|_{L^2\to L^2}&\leq
C_m(1+|j_1-j_2|+|l_1-l_2|)^{-m}h^m.
\end{aligned}
\end{equation}

To prove \eqref{eq:Desired estimate for Pkl}, we claim that for any $j,l$ the projection sets 
\[
\Pi_x(\operatorname{supp}(\chi_Y\psi_l)),\qquad
\Pi_\xi(\operatorname{supp}(\chi_X\psi_j))
\]
are $\gamma'$-porous of scale 
\[[(100/\gamma)^{1/\alpha}h^{1/2+\epsilon},1]\] for some $\gamma'=\gamma/100>0$ independent of $j,l$.  
As $h\to 0$, we can always choose $h$ so small that $(100/\gamma)^{1/\alpha}h^{1/2+\epsilon}\ll 1$.  
Suppose the claim holds.  Since  
\[
P_{jl}=\chi_X(hD)\,\psi_j(hD)\;\mathds{1}_{\Pi_\xi(\operatorname{supp}(\chi_X\psi_j))\times \mathbb{R}_\eta}(hD)\;
\mathds{1}_{\Pi_x(\operatorname{supp}(\chi_Y\psi_l))\times \mathbb{R}_y}\;\psi_l\,\chi_Y,
\]
we can apply the one-dimensional FUP (see e.g. \cite[Proposition 4.2]{dyatlov2024semiclassical} which is derived from \cite{bourgain2018spectral}) to deduce
\[
\begin{aligned}
&\mathds{1}_{\Pi_\xi(\operatorname{supp}(\chi_X\psi_j))\times \mathbb{R}_\eta}(hD)\;
\mathds{1}_{\Pi_x(\operatorname{supp}(\chi_Y\psi_l))\times \mathbb{R}_y}\\
=&\;
\mathds{1}_{\Pi_\xi(\operatorname{supp}(\chi_X\psi_j))}\,\mathcal{F}_{h,\mathbb{R}}\,
\mathds{1}_{\Pi_x(\operatorname{supp}(\chi_Y\psi_l))}\otimes \operatorname{id}_y\\
=&\;\mathcal{O}(h^\beta)_{L^2(\mathbb{R})\to L^2(\mathbb{R})}
\end{aligned}
\]
for some $\beta=\beta(\gamma')>0$, where $\mathcal F_{h,\mathbb R}$ is the one-dimensional semiclassical Fourier transform along the $x$-direction. Then \eqref{eq:Desired estimate for Pkl} follows.

To prove the claim, fix $l$ and consider the line $\ell_l=\{(x,y):y=lh^\rho\}$, parallel to $\partial_x$.  
The anisotropic porosity of $Y_h$ implies that for any line segment 
$I\subset \ell_l$ of length $L\in [(5\gamma)^{1/\alpha}h^{1/2+\epsilon},1]$, we can find 
some $(x_0, lh^\rho)\in I$ such that 
\[
\bigl([x_0-\gamma L/20,x_0+\gamma L/20]_x\times [lh^\rho-\gamma L^\alpha/20,lh^\rho+\gamma L^\alpha/20]_y\bigr)
\cap Y_h=\emptyset.
\]
Thanks to the cutoff $\psi_l$, we know 
\[
[x_0-\gamma L/20,x_0+\gamma L/20]\cap \Pi_x(\operatorname{supp}(\chi_Y\psi_l))=\emptyset
\]
as long as $\gamma L^\alpha/20\geq h^\rho$, i.e. $L\geq (20/\gamma)^{1/\alpha}h^{1/2+\epsilon}$.  
This proves the claim for $\operatorname{supp}(\chi_Y\psi_l)$; a similar argument works
for $\operatorname{supp}(\chi_X\psi_j)$.  Hence the claim holds and \eqref{eq:Desired estimate for Pkl} is proved.

To prove \eqref{eq:Desired estimate for Pk1l1Pk2l2}, 
we note that 
\[
P_{j_1l_1}^*P_{j_2l_2}=
\chi_Y\,\psi_{l_1}\,\psi_{j_1}(hD)\,\chi_X(hD)\,\chi_X(hD)\,\psi_{j_2}(hD)\,\psi_{l_2}\,\chi_Y,
\]
which is nonzero only when $|j_1-j_2|\leq 3$.  Thus our task 
reduces to proving for $|l_1-l_2|\geq 97$,
\begin{equation}\label{eq:Desired estimate for psil psik psil}
\|\psi_{l_1}\,(\psi_{j_1}\psi_{j_2}\chi_X^2)(hD)\,\psi_{l_2}\|_{L^2\to L^2}
\leq C_m(1+|l_1-l_2|)^{-m}h^m.
\end{equation}
This falls into the category of \textit{symbol calculus associated to a coisotropic space}
described in \cite[Section 2.1.4]{dyatlov2024semiclassical}.  
Indeed, define the coisotropic space (its symplectic complement is contained in itself)
\[
L^{c}=\operatorname{Span}_{\mathbb{R}}(\partial_x,\partial_y,\partial_\eta)\subset T(T^*\mathbb{R}^2)\simeq \mathbb{R}^4.
\]
Then the estimate \eqref{eq:estimate_chix_chiy} and our choice of $(\psi_j)_{j\in\mathbb Z}$ imply that the symbols
\[
\psi_{j_1}(\xi,\eta)\psi_{j_2}(\xi,\eta)\chi_X^2(\xi,\eta),\quad 
\psi_{l_1}\otimes 1_{\xi,\eta},\quad 
\psi_{l_2}\otimes 1_{\xi,\eta}
\]
all lie in the class $S_{\frac{1}{2}+\epsilon,\rho}(L^{c}, \mathbb{R}^4)$ with uniform seminorm bounds for all $j,l$; here $S_{\frac{1}{2}+\epsilon,\rho}(L^{c}, \mathbb{R}^4)$ is a smooth $h$-dependent symbol class where derivatives grow by $h^{-\rho}$ along $L^{c}$ and by $h^{-(1/2+\epsilon)}$ in other directions.  See \cite[Definition 2.2]{dyatlov2024semiclassical} for the precise definition.  
Hence the desired estimate \eqref{eq:Desired estimate for psil psik psil}, interpreted as a composition of Weyl quantization operators with symbols in $S_{\frac{1}{2}+\epsilon,\rho}(L^{c}, \mathbb{R}^4)$, follows from \cite[Lemma 2.4]{dyatlov2024semiclassical}.  
The estimate for $P_{k_1l_1}P_{k_2l_2}^*$ follows in the same way.  This proves 
\eqref{eq:Desired estimate for Pk1l1Pk2l2} and completes the proof of the proposition.
\end{proof}

\subsection{Rescaled FUP and FUP involving $\mathcal F_h^{\operatorname{ani}}$}
\label{subsec:rescaled}

Using similar argument as in the proof of 
Theorem \ref{thm:cfup}, we can prove the following rescaled anisotropic continuous FUP.

\begin{theorem}[Rescaled CFUP]\label{thm:cfup-dual}
Let $\gamma\in(0,1)$, $\alpha,\alpha'\in(0,1)$, and $0<h<h(\gamma,\alpha,\alpha')\ll 1$, and let $(v_1,v_2)$ and $(v_1',v_2')$ be two pairs of linearly independent vectors in $\mathbb{R}^2$.
Let $X$ be an $\alpha$-anisotropic $\gamma$-porous set of scale $[h,1]$ along the direction $({v}_1,{v}_2)$, 
and let $Y$ be an $1/\alpha'$-anisotropic $\gamma$-porous set of scale $[1,h^{-1}]$ 
along the direction $({v}_1',{v}_2')$, such that ${v}_1'\perp {v}_2$ and ${v}_2'\perp {v}_1$. 
Then there exist $\beta>0$ and $C>0$ independent of $h$ such that 
\begin{equation}\label{eq:cfup for 1/alpha anisotorpic}
\|\mathds{1}_Y \mathcal{F} \mathds{1}_X\|_{L^2\to L^2}\leq C h^\beta,
\end{equation}
where $\mathcal{F}$ is the standard Fourier transform.
\end{theorem}

\begin{proof}
    As in the proof of Theorem \ref{thm:cfup}, we can 
    assume $\alpha=\alpha'$ and
    \[
        ({v}_1,{v}_2)=(\partial_x,\partial_y),\qquad ({v}_1',{v}_2')=
        (\partial_\xi,\partial_\eta).
    \]
    where $(x,y)$ is the 
    standard coordinate in the physical space, and 
    $(\xi,\eta)$ is the standard coordinate in the frequency space.

    We choose $\epsilon>0$ so small so that
    \[
        \begin{aligned}
            \rho_1:=\left(\frac{1}{2}+\epsilon\right)\alpha,
            \quad &\rho_2:=-1+\left(\frac{1}{2}+\epsilon\right)\alpha^{-1},\quad \rho:=\max(\rho_1,\rho_2)\\
            &\rho<\frac{1}{2},\quad \rho+\frac{1}{2}+\epsilon<1
        \end{aligned}
    \]
    Next we consider the $h$-dependent neighborhoods
    \[
        \begin{aligned}
            X_h&:=X+[-h^{1/2+\epsilon},h^{1/2+\epsilon}]_{x}\times [-h^{(1/2+\epsilon)\alpha},h^{(1/2+\epsilon)\alpha}]_{y},\\
            Y_h&:=Y+[-h^{-1/2+\epsilon},h^{-1/2+\epsilon}]_{\xi}\times 
            [-h^{(-1/2+\epsilon)/\alpha},h^{(-1/2+\epsilon)/\alpha}]_{\eta}.
        \end{aligned}
    \]
    $X_h$ is $\alpha$-anisotropic $\gamma/5$-porous of 
    scale $[(5/\gamma)^{1/\alpha}h^{1/2+\epsilon},1]$ 
    along the directions $(\partial_x,\partial_y)$, and 
    $Y_h$ is $1/\alpha$-anisotropic $\gamma/5$ porous 
    of scale $[(5/\gamma)^{1/\alpha}h^{-1/2+\epsilon},h^{-1}]$.
    We still choose $\chi_X\in C_c^\infty(X_h)$ 
    which equals one on $X$ and $\chi_Y\in C_c^\infty(Y_h)$ which equals one on $Y$ 
    such that for all $(j_1,j_2)\in \mathbb N^2$,
    \begin{equation}\label{eq:estimate_chix in 1/alpha}
        \partial_{x}^{j_1}\partial_{y}^{j_2}\chi_X=\mathcal{O}_{j_1,j_2}\bigl(h^{-j_2\rho -j_1(1/2+\epsilon)}\bigr),
        \quad
        \partial_{\xi}^{j_1}\partial_{\eta}^{j_2}\chi_Y=\mathcal{O}_{j_1,j_2}\bigl(h^{+j_2(1/2-\epsilon)/\alpha +j_1(1/2-\epsilon)}\bigr)
    \end{equation}
    We also consider the rescale function $\tilde{\chi}_Y(\xi,\eta):=\chi_Y(\xi/h,\eta/h)$, it has estimate 
    \begin{equation}\label{eq:estimate_chiy in 1/alpha}
        \partial_{\xi}^{j_1}\partial_{\eta}^{j_2}\tilde{\chi}_Y=
        \mathcal{O}_{j_1,j_2}\bigl(h^{-j_2\rho_2- j_1(1/2+\epsilon)}\bigr)
    \end{equation}
    It is reduced to show that for some $\beta>0$
    \[
        \|\tilde{\chi}_Y(hD)\,\chi_X\|\lesssim h^\beta.
    \]

    We still apply a partition of unity along the $y$-direction and 
    the $\eta$-direction. Fix a function $\tilde{\psi}\in C_c^\infty((-1,1),[0,1])$ such that 
    \[
        \sum_{j\in \mathbb{Z}} \tilde{\psi}(x-j)\equiv 1,
    \]
    then define the functions $\psi_j\in C_c^\infty(\mathbb{R}^2)$ as 
    \[
        \psi_j(x,y):=\tilde{\psi}\Bigl(\frac{y}{h^\rho}-j\Bigr),\qquad j\in \mathbb Z,
    \]
    and the operators $P_{jl}:L^2(\mathbb{R}^2)\to 
    L^2(\mathbb{R}^2)$ as 
    \[
        P_{jl}:=\tilde{\chi}_Y(hD)\,\psi_j(hD)\, \psi_l\, \chi_X,\qquad j,l\in \mathbb Z.
    \]
    Then we know for $|j_1-j_2|+|l_1-l_2|\leq 100$, the operators $P_{j_1l_1}^*P_{j_2l_2}$ 
    and $P_{j_1l_1}P_{j_2l_2}^*$ are of $\mathcal{O}_{L^2\to L^2}(h^\beta)$ for some $\beta>0$, using 
    the one-dimensional FUP after projecting $X$ to the $x$-axis and $Y$ to the $\xi$-axis.
    Actually, the same argument as in the proof of Theorem \ref{thm:cfup} shows that 
    the set
    \[ 
        \{\xi\in \mathbb{R}:\text{there exists } \eta\in \mathbb{R} \text{ so that } (\xi,\eta)\in \operatorname{supp}(\chi_Y) 
        \text{ and } (\xi/h,\eta/h)\in \operatorname{supp} \psi_j\}
    \]
    is of $\gamma/C_{\gamma,\alpha}$-porous from scale $C_{\gamma,\alpha}h^{-1/2+\epsilon}$ to $h^{-1}$, 
    where $C_{\gamma,\alpha}$ is a large constant depending only on $\gamma,\alpha$. And the 
    porous condition for the projection of $X$ to the $x$-axis is the same as in the proof of Theorem \ref{thm:cfup}.
    
    For $|j_1-j_2|+|l_1-l_2|\leq 100$, we still have 
    \[
        \begin{aligned}
        \|P_{j_1l_1}^*P_{j_2l_2}\|_{L^2\to L^2}&\leq
        C_m(1+|j_1-j_2|+|l_1-l_2|)^{-m}h^m,\\
        \|P_{j_1l_1}P_{j_2l_2}^*\|_{L^2\to L^2}&\leq
        C_m(1+|j_1-j_2|+|l_1-l_2|)^{-m}h^m.
        \end{aligned}
    \]
    using the $S_{\frac{1}{2}+\epsilon,\rho}(L^{c}, \mathbb{R}^4)$ calculus for symbols 
    as in the proof of Theorem \ref{thm:cfup}. So \eqref{eq:cfup for 1/alpha anisotorpic} follows from 
    Cotlar-Stein's lemma.
\end{proof}


Combining example \ref{ex:BMcarpet-porous} and example \ref{ex:rescaled BM-carpet 1/alpha porous} and Theorem \ref{thm:cfup-dual}, we have the following corollary for \emph{orthogonal Bedford-McMullen carpets}.      
\begin{corollary}\label{cor:cfup-ani}
    Let $\mathbf{X},\mathbf{Y}\subset \mathbb{R}^2$ be two orthogonal Bedford-McMullen carpets, both not full for each row. Then the neighborhood $\mathbf X_h$ and the rescaled neighborhood $\widetilde {\mathbf Y_h}$ are a dual pair of anisotropic porous sets as in Theorem \ref{thm:cfup-dual}.  Furthermore, there exist $\beta>0$ and $C>0$ independent of $h$ such that
    \[
        ||\mathds{1}_{\mathbf{Y}_h}\mathcal{F}_h^{\operatorname{ani}} \mathds{1}_{\mathbf{X}_h}||\leq Ch^\beta,
    \]
where $\mathcal F_h^{\mathrm{ani}}$ is the anisotropic semiclassical Fourier transform \eqref{eq:anis-semi-F}.  
\end{corollary}

\begin{remark}In particular, we can choose two carpets: the first such that each row is not full, and the second such that each column is not full. Then reflect the second carpet to obtain a reflected carpet. The directions of the two carpets (the first and the reflected second) are orthogonal.  We can then prove Theorem \ref{thm:cfup-bm-ani} for the choosen orthogonal carpets. \end{remark}

\appendix

\section{Submultiplicativity of the discrete FUP}\label{app:submultiplicativity}
We aim to prove the submultiplicative inequality \eqref{eq:submultiplicativity} of the FUP norms under the Cantor iterates.
This is a repetition of Section 5.1 of \cite{Cohen}. We give a detailed 
calculation for the reader's convenience.

For integers $T_1,T_2,S_1,S_2\geq 2$, let $\mathbf T=(T_1, T_2)$ and $\mathbf S=(S_1, S_2)$. We define the following two unitary isomorphisms
\[
\Psi_1,\Psi_2: L^2(\mathbb{Z}_{\mathbf T})\otimes L^2(\mathbb{Z}_{\mathbf S})\to L^2(\mathbb{Z}_{\mathbf T\mathbf S})
\]
by
\[
\Psi_1(u\otimes v)(\mathbf t \mathbf S+ \mathbf s) = \Psi_2(u\otimes v)(\mathbf s \mathbf T+ \mathbf t) = u(\mathbf t)v(\mathbf s)
\]
for $u\in L^2(\mathbb{Z}_{\mathbf T})$, $v\in L^2(\mathbb{Z}_{\mathbf S})$ and $\mathbf t\in \mathbb Z_{\mathbf T}$, $\mathbf s\in \mathbb Z_{\mathbf S}$.
We next define the \emph{phase shift operator} $D^{ps}$ as a multiplication operator on the tensor product space $L^2(\mathbb{Z}_{\mathbf T})\otimes L^2(\mathbb{Z}_{\mathbf S})$ by
\[
D^{ps}f(\mathbf t, \mathbf s)=\exp\left[-2\pi i\, \frac{\mathbf t}{\mathbf T}\cdot \frac{\mathbf s}{\mathbf S}\right]f(\mathbf t, \mathbf s).
\] 

We \emph{claim} that the following tensor decomposition identity holds
\begin{equation}\label{eq:tensor-claim-submul}
    \Psi_1\, (\mathcal{F}_{\mathbf T} \otimes \operatorname{id})\, D^{ps}\, (\operatorname{id} \otimes  \mathcal{F}_{\mathbf S})\, \Psi_2^{-1}=\mathcal{F}_{\mathbf T\mathbf S}.
\end{equation}
Indeed, for any $u\in L^2(\mathbb{Z}_{\mathbf T})$ and 
$v\in L^2(\mathbb{Z}_{\mathbf S})$, we have
\[
    \begin{aligned}
        &(\mathcal{F}_{\mathbf T}\otimes\, \operatorname{id})\, D^{ps}\, (\operatorname{id}\otimes \,\mathcal{F}_{\mathbf S})(u\otimes v)(\mathbf t, \mathbf s)\\
        =&\frac{1}{\sqrt{T_1T_2}}\,\sum_{\mathbf t'\in \mathbb Z_{\mathbf T}}\,
        \exp\left[-2\pi i\,\mathbf t\cdot\frac{\mathbf t'}{\mathbf T}\right]\, \exp\left[-2\pi i\,\frac{\mathbf t'}{\mathbf T} \cdot \frac{\mathbf s}{\mathbf S}\right]\,
        u(\mathbf t')\,(\mathcal{F}_{\mathbf S}v)(\mathbf s)\\ 
        =&\frac{1}{\sqrt{T_1T_2S_1S_2}}\,\sum_{\mathbf t'\in \mathbb Z_{\mathbf T},\mathbf s'\in \mathbb Z_{\mathbf S}}\,
        \exp[-2\pi i\, \Phi(\mathbf t',\mathbf s';\mathbf t,\mathbf s)]\,
        u(\mathbf t')\,v(\mathbf s')
    \end{aligned}
\]
where the phase function $\Phi$ is defined as
\[\Phi(\mathbf t',\mathbf s';\mathbf t,\mathbf s)=\mathbf t\cdot\frac{\mathbf t'}{\mathbf T}+\frac{\mathbf t'}{\mathbf T} \cdot \frac{\mathbf s}{\mathbf S}+\mathbf s\cdot\frac{\mathbf s'}{\mathbf S}.
\]
On the other hand, we know 
\[\Psi_1^{-1}\, \mathcal{F}_{\mathbf T\mathbf S}\Psi_2(u\otimes v) (\mathbf t,\mathbf s)       =
    \frac{1}{\sqrt{T_1T_2S_1S_2}}\sum_{\mathbf t'\in \mathbb Z_{\mathbf T},\mathbf s'\in \mathbb Z_{\mathbf S}}\,
  \exp[-2\pi i\, \tilde \Phi(\mathbf t',\mathbf s';\mathbf t,\mathbf s)]\,
        u(\mathbf t')\,v(\mathbf s')\]
where the phase function $\tilde{\Phi}$ is 
\[
   \tilde \Phi(\mathbf t',\mathbf s';\mathbf t,\mathbf s)=\frac{\mathbf s' \mathbf T+\mathbf t'}{\mathbf T}\cdot 
   \frac{\mathbf t \mathbf S+\mathbf s}{\mathbf S}=(\mathbf s'+\frac{\mathbf t'}{\mathbf T})\cdot(\mathbf t+\frac{\mathbf s}{\mathbf S}) \quad = \Phi(\mathbf t',\mathbf s';\mathbf t,\mathbf s) \quad \text{mod} \quad   \mathbb Z.
\]
The last equality holds after taking modulo $\mathbb{Z}$. This completes the proof of the claimed identity \eqref{eq:tensor-claim-submul}.

Now we are ready to prove the submultiplicativity of the FUP norm. 
Let $\mathbf T=\mathbf M^{k_1}$ and $\mathbf S=\mathbf M^{k_2}$. 
By the self-similarity of Cantor sets, at the operator level we have 
\[
    \Psi_1\, (\mathds{1}_{\mathcal{Y}_{k_1}}\otimes \mathds{1}_{\mathcal{Y}_{k_2}})\, \Psi_1^{-1}=\mathds{1}_{\mathcal{Y}_{k_1+k_2}},
    \quad \Psi_2\, (\mathds{1}_{\mathcal{X}_{k_1}}\otimes \mathds{1}_{\mathcal{X}_{k_2}})\, \Psi_2^{-1}=\mathds{1}_{\mathcal{X}_{k_1+k_2}}.
\] 
Thus by the tensor identity \eqref{eq:tensor-claim-submul} we have 
\[
    \begin{aligned}
        \mathds{1}_{\mathcal{Y}_{k_1+k_2}}\mathcal{F}_{\mathbf T\mathbf S}\mathds{1}_{\mathcal{X}_{k_1+k_2}}
        =&\Psi_1\, (\mathds{1}_{\mathcal{Y}_{k_1}}\otimes \mathds{1}_{\mathcal{Y}_{k_2}})\, \Psi_1^{-1}\, \mathcal{F}_{\mathbf T\mathbf S}\, 
        \Psi_2\, (\mathds{1}_{\mathcal{X}_{k_1}}\otimes \mathds{1}_{\mathcal{X}_{k_2}})\, \Psi_2^{-1}\\
        =&\Psi_1\, (\mathds{1}_{\mathcal{Y}_{k_1}}\otimes \mathds{1}_{\mathcal{Y}_{k_2}})\, (\mathcal{F}_{\mathbf T}\otimes \operatorname{id})\, 
        D^{ps}\, (\operatorname{id}\otimes \mathcal{F}_{\mathbf S})\, (\mathds{1}_{\mathcal{X}_{k_1}}\otimes \mathds{1}_{\mathcal{X}_{k_2}})
        \,\Psi_2^{-1}\\
        =&\Psi_1\,(\mathds{1}_{\mathcal{Y}_{k_1}}\mathcal{F}_{\mathbf T}\mathds{1}_{\mathcal{X}_{k_1}}\otimes \mathds{1}_{\mathcal{Y}_{k_2}})\, 
        D^{ps}\, (\mathds{1}_{\mathcal{X}_{k_1}}\otimes \mathds{1}_{\mathcal{Y}_{k_2}}\mathcal{F}_{\mathbf S}\mathds{1}_{\mathcal{X}_{k_2}})
        \,\Psi_2^{-1}
    \end{aligned}
\]
where in the last equality we used the fact that $D^{ps}$ is a multiplication operator. 
This implies the submultiplicativity inequality \eqref{eq:submultiplicativity}.

\section{Irreducible lines in $\mathbb{Z}_{\mathbf{N}}$}\label{app:irreducible-lines}

In this appendix, we discuss the structure of a \emph{line} in $\mathbb{Z}_{\mathbf{N}}$ defined as \eqref{eq:line}.

Let us consider the non-empty set $\ell\subset \mathbb{Z}_{\mathbf{N}}$ 
defined as 
\[
\ell = \{ (x, y) \in \mathbb{Z}_{\mathbf{N}} : \frac{ax}{N_1} + \frac{by}{N_2} = c \quad (\text{mod}~ 1)\}
\]
for some integers $a, b$ and $c\in {\mathbb{Z}}/{(N_1N_2)}$. 
Define the group homomorphism $\phi:\mathbb{Z}_{\mathbf{N}}\to \mathbb{Q}/\mathbb{Z}$ by 
\[
    \phi(x,y)=\frac{ax}{N_1}+\frac{by}{N_2}.
\]
Then 
$\ell$ is a \emph{coset} in the sense that 
\[
    \ell=\ker(\phi)+\mathbf{p}_0
\]
for some $\mathbf{p}_0\in \ell$. 

The image of $\phi$ is a cyclic group generated by $\frac{1}{M}\in \mathbb{Q}/\mathbb{Z}$, 
where $M$ is the least common multiple 
\begin{equation}\label{eq:M}
    M:=\operatorname{lcm}\left(\frac{N_1}{\gcd(a,N_1)}, \frac{N_2}{\gcd(b,N_2)}\right).
\end{equation}
Set $m_1=\frac{N_1}{\gcd(a,N_1)}$ and $m_2=\frac{N_2}{\gcd(b,N_2)}$. It is easy to check, using Bézout's identity, that the set $\{\frac{ax}{N_1}: x \in \mathbb{Z}_{N_1}\}$ is a cyclic group generated by 
$\frac{1}{m_1}\in \mathbb{Q}/\mathbb{Z}$, and the same 
result holds for $\{\frac{by}{N_2}: y \in \mathbb{Z}_{N_2}\}$ 
with generator $\frac{1}{m_2}$. The sum of these two cyclic groups is a cyclic group generated by $\frac{1}{M}$, 
where $M$ is the least common multiple of $m_1$ and $m_2$.

By the isomorphism $\mathbb{Z}_{\mathbf{N}}/\ker(\phi)\simeq \operatorname{Im}(\phi)$, the kernel $\ker(\phi)$ is a subgroup of $\mathbb{Z}_{\mathbf{N}}$ of index $M$, and thus 
\[
    |\ker(\phi)|=|\ell|=\frac{N_1N_2}{M}.
\]

Next we consider different parameterization of the same line $\ell$, 
i.e., different choices of $(a,b,c)$ that define $\ell$.

Suppose that $\ell$ can also be defined by parameters $(a',b',c')$; then we must have
\[
    \operatorname{ker} \phi'=\operatorname{ker} \phi 
\]
where $\phi'$ is the group homomorphism from $\mathbb{Z}_{\mathbf{N}}$ to $\mathbb{Q}/\mathbb{Z}$ defined 
by $\phi'(x,y)=\frac{a'x}{N_1}+\frac{b'y}{N_2}$. 
We note that $\operatorname{Im}(\phi')$ is the cyclic group generated by $\frac{1}{M'}\in \mathbb{Q}/\mathbb{Z}$, where
\[
    M'=\operatorname{lcm}\left(\frac{N_1}{\gcd(a',N_1)}, \frac{N_2}{\gcd(b',N_2)}\right).
\]
Since $\operatorname{ker} \phi'=\operatorname{ker} \phi$, 
we must have $M'=M$. So if we define the isomorphisms 
$\bar{\phi},\bar{\phi}':\mathbb{Z}_{\mathbf{N}}/\ker(\phi)\to \mathbb{Z}\frac{1}{M}$ 
induced by $\phi$ and $\phi'$ respectively, then the map $\varphi:=\bar{\phi}'\circ \bar{\phi}^{-1}$ is an automorphism of $\mathbb{Z}\frac{1}{M}$, and thus $\varphi(\frac{1}{M})=\frac{k}{M}$ for some $k\in \mathbb{Z}$ 
with $\gcd(k,M)=1$. Now for any $\mathbf{p}\in \mathbb{Z}_{\mathbf{N}}$, we have 
\[
    \phi'(\mathbf{p})=\bar{\phi}'(\mathbf{p}+\ker(\phi))=
    \varphi\circ \bar{\phi}(\mathbf{p}+\ker(\phi))=k\cdot \phi(\mathbf{p}).
\]
This implies that 
\[
    \phi'=k\cdot \phi 
\]
for some $k\in \mathbb{Z}$ with $\gcd(k,M)=1$. 
Therefore 
\begin{equation}\label{eq:line-rep}
    a'\equiv ka \mod N_1,\qquad b'\equiv kb \mod N_2 \qquad \text{and}\qquad c'\equiv kc \mod 1.
\end{equation}

Conversely, given $k\in \mathbb{Z}$ with $\gcd(k,M)=1$, define $\phi'=k\phi$. Then $\operatorname{ker} \phi'=\operatorname{ker} \phi$ 
and thus $(a',b',c')=(ka,kb,kc)$ gives another parameterization 
of the same line $\ell$.

\begin{proposition}
Any irreducible line $\ell$ in $\mathbb Z_{\mathbf N}$ with normal vector $(a,b)\in \mathbb Z_{\mathbf N}$ (see Definition \ref{def:ir-line}) has $N_1N_2/M$ points, where $M$ is the least common multiple defined in \eqref{eq:M}. Moreover, $\ell$ has another representation with normal vector $(a',b')$ if and only if $(a',b') \equiv k\cdot(a,b) \pmod{\mathbf N}$ for some $k\in \mathbb Z$ with $\gcd(k,M)=1$.
\end{proposition}

\begin{remark}
    There might be different choices of $(a,b,c)$ with $(a,b)\in \mathbb Z_{\mathbf N}$ and  
    $\gcd(a,b)=1$ that define the same line $\ell$.
    For example, if $N_1=15$, $N_2=10$, $a=1$, $b=2$, then by \eqref{eq:M} $M=30$ 
    and by \eqref{eq:line-rep} we may take $k=11$; then $a'=11$ and $b'=2$, and 
    the line defined by $(a,b,c=0)$ coincides with that defined by $(a',b',c'=0)$.
\end{remark}

\section{Anisotropic porosity under transformations}\label{app:ani-porous-transf}

\begin{proposition}
    \label{prop:porous small scale}
Let $\alpha\in(0,1)$, $\gamma\in(0,1)$, and $0<h<h(\alpha,\gamma)\ll 1$. Let $X\subset\mathbb R^2$ be an $\alpha$-anisotropic $\gamma$-porous set of scale $[h,1]$ along two independent directions $(v_1,v_2)\in\mathbb R^2\times\mathbb R^2$; see Definition \ref{def:porous}. Then
\begin{enumerate}[label=\emph{(\alph*)}]
\item (Rectangle neighborhood) $X+\{x v_1+y v_2 \mid x\in[-h,h],\, y\in[-h^\alpha,h^\alpha]\}$ is $\alpha$-anisotropic $\gamma/5$-porous of scale $[(5/\gamma)^{1/\alpha}h,1]$ along the same directions.
\item (Linear transformation) Let $A\in\operatorname{GL}(2,\mathbb{R})$. Then $AX$ is $\alpha$-anisotropic $\gamma'$-porous of scale $[\rho h,\rho]$ along the directions $(A v_1,A v_2)$, where
\[
\rho=\frac{|Av_1|}{|v_1|},\qquad \gamma'=\min(\rho^{-\alpha},\rho^{-1})\,\gamma.
\] 
\end{enumerate}
\end{proposition}

\begin{proof}
       
(a) Given a segment $I$ parallel to $v_1$ with length $l\in[(5/\gamma)^{1/\alpha}h,1]$, we can find $x_0\in I$ such that
\[
 \operatorname{P} \cap X=\emptyset,  \qquad \text{where}~ \operatorname{P}=\bigl\{x_0+t_1 v_1+t_2 v_2 \mid t_1\in[-\tfrac{\gamma l}{2},\tfrac{\gamma l}{2}],\; t_2\in[-\tfrac{\gamma l^\alpha}{2},\tfrac{\gamma l^\alpha}{2}]\bigr\}.
\]
Consequently,
\begin{multline*}
\bigl\{x_0+t_1 v_1+t_2 v_2 \mid t_1\in[-\tfrac{\gamma l}{2}+h,\tfrac{\gamma l}{2}-h],\; t_2\in[-\tfrac{\gamma l^\alpha}{2}+h^\alpha,\tfrac{\gamma l^\alpha}{2}-h^\alpha]\bigr\}\\
\cap\bigl(X+\{x v_1+y v_2 \mid x\in[-h,h],\; y\in[-h^\alpha,h^\alpha]\}\bigr)=\emptyset.
\end{multline*}
Hence the condition $l\geq(5/\gamma)^{1/\alpha}h$ implies
\[
\frac{\gamma l}{2}-h \ge \frac{\gamma l}{2}-\frac{\gamma l}{5} = \frac{\gamma l}{5},\qquad
\frac{\gamma l^\alpha}{2}-h^\alpha \ge \frac{\gamma l^\alpha}{2}-\frac{\gamma l^\alpha}{5} = \frac{\gamma l^\alpha}{5}.
\]
This proves (a).

(b) Let $A\in \operatorname{GL}(2,\mathbb{R})$. Consider the same parallelogram $\mathrm{P}$ as above.  Set $u_1=Av_1$ and $u_2=Av_2$. Under the linear map $A$, original parallelogram $\mathrm P$ is sent to the parallelogram along $(u_1,u_2)$ centered at $Ax_0$:
\[
A(\mathrm P) = \bigl\{Ax_0+t_1 u_1+t_2 u_2 \mid t_1\in[-\tfrac{\gamma l}{2},\tfrac{\gamma l}{2}],\; t_2\in[-\tfrac{\gamma l^\alpha}{2},\tfrac{\gamma l^\alpha}{2}]\bigr\}.
\]
Moreover, a segment $I$ parallel to $v_1$ is mapped under $A$ to a segment $I'$ parallel to $u_1$ with length $\frac{|u_1|}{|v_1|}|I|$. Thus (b) follows from the definition of $\alpha$-anisotropic $\gamma$-porous sets.
\end{proof}

Similarly, we can prove the following properties for \emph{macroscopic} anisotropic porous sets without proof.
\begin{proposition}
    Let $\alpha\in(0,1)$, $\gamma\in(0,1)$, and $0<h<h(\alpha,\gamma)\ll 1$. 
    Let $X\subset\mathbb R^2$ be an $1/\alpha$-anisotropic $\gamma$-porous set of 
    scale $[1,h^{-\alpha}]$ along two independent directions $(v_1,v_2)\in\mathbb R^2\times\mathbb R^2$; see Definition \ref{def:porous}. Then
\begin{enumerate}[label=\emph{(\alph*)}]
\item $X+\{x v_1+y v_2 \mid x\in[-1,1],\, y\in[-1,1]\}$ 
is $1/\alpha$-anisotropic $\gamma/5$-porous of scale $[5/\gamma,h^{-\alpha}]$ along the same directions.
\item Let $A\in\operatorname{GL}(2,\mathbb{R})$. Then $AX$ is $1/\alpha$-anisotropic $\gamma'$-porous of scale $[\rho,\rho h^{-\alpha}]$ 
along the directions $(A v_1,A v_2)$, where
\[
\rho=\frac{|Av_1|}{|v_1|},\qquad \gamma'=\min(\rho^{-1/\alpha},\rho^{-1})\,\gamma.
\] 
\end{enumerate}
\end{proposition}

\section{Different porosities of BM carpets}\label{app:porositys}

In this section we review the definitions of ball and line porosity for fractal sets from \cite{{MR4085124}} and \cite{MR4927737}, and we discuss whether, and when, a BM carpet or its anisotropic rescaling exhibits different porosities, and compare the relations among ball, line, and anisotropic porosities.

\begin{definition}
    Let $\mathbf{X}\subset \mathbb{R}^n$, and $\gamma\in (0,1/10)$.
    We say $\mathbf{X}$ is \emph{$\gamma$-porous on balls}
    from scales $l_1$ to $l_2$ if for every ball 
    $\mathrm{B}\subset \mathbb{R}^n$ of diameter $r\in [l_1,l_2]$ 
    we can find some $\mathbf{x}\in \mathrm{B}$ such that 
    $B(\mathbf{x},\gamma r)\cap \mathbf{X}=\emptyset$. 
    We say $\mathbf{X}$ is \emph{$\gamma$-porous on lines}
    from scales $l_1$ to $l_2$ if for every line segment
    $\tau\subset \mathbb{R}^n$ of length $r\in [l_1,l_2]$ 
    we can find some $\mathbf{x}\in \tau$ such that 
    $B(\mathbf{x},\gamma r)\cap \mathbf{X}=\emptyset$. 
\end{definition}

The following proposition shows that Cohen's line porosity can easily \emph{fail} for BM carpets, in particular in the anisotropic porous cases.

\begin{proposition}\label{prop:Failure of line/ball porosity for BM carpet}
    Let $\mathcal{A}\subset \mathbb{Z}_{\mathbf{M}}$ be an alphabet, 
    and let $\mathbf{X}\subset \mathbb{R}^2$ be associated Bedford-McMullen carpet. Then 
    \begin{enumerate}[label=\emph{(\alph*)}]
        \item   (Failure of line porosity) If $\mathcal{A}$ is \textbf{nonempty} for each row, 
                then there is no uniform 
                porous constant $\gamma$ for the 
                line porosity condition at scale $h$ for $\mathbf{X}$ as $h\to 0$.
        \item   (Ball porosity) If some row of $\mathcal{A} $ is \textbf{full}, and every row is \textbf{non-empty}, 
                then there is no uniform 
                porous constant $\gamma$ for the 
                ball porosity condition at scale $h$ for $\mathbf{X}$ as $h\to 0$. This condition is also necessary. That is, $\mathbf X$ is ball porous if and only if $\mathcal A$ is not full for all rows or some row is empty.
    \end{enumerate}
\end{proposition}

\begin{proof}
    
(a)    
Suppose that the alphabet $\mathcal{A}$ is nonempty for each row, we need to show the line porosity condition 
cannot hold. In view of the decimal expansion, the projection of $\mathbf{X}$ onto the $y$-axis is dense in the interval $[0,1]$, 
and actually equals the whole interval $[0,1]$ by compactness. Consider the left boundary of 
any sub-rectangle in the $k$-th generation that is not removed. 
Specifically, consider the line segment $I$ parallel to the $y$-axis defined by
\[
I = \bigl\{ (M_1^{-k}x, M_2^{-k}y) : x = \sum_{j=0}^{k-1} a_1^{(j)} M_1^{j},\ 0 \le y - \sum_{j=0}^{k-1} a_2^{(j)} M_2^{j} \le 1,\quad (a_1^{(j)}, a_2^{(j)}) \in \mathcal{A} \quad \forall\, j\bigr\}.
\]
We have $|I| = M_2^{-k}$. Since this sub-rectangle is not removed and equals a 
translation of $\mathbf{X}$ after rescaling $(x,y) \mapsto (M_1^k x, M_2^k y)$, 
we know that for each $\mathbf{p} \in I$ the line segment parallel to 
the $x$-axis centered at $\mathbf{p}$ with length $2M_1^{-k}$ has nonempty intersection with $\mathbf{X}$.
This forces Cohen's line porosity condition to fail for $\mathbf{X}$, 
because $M_1^{-k} \sim |I|^{\log M_1 / \log M_2}$, 
which is much smaller than $|I|$ as $k \to \infty$. See Figure \ref{fig:non_line_porosity_for_BM_carpet}.

\begin{figure}[htbp]
    \centering
    \includegraphics[width=0.8\textwidth]{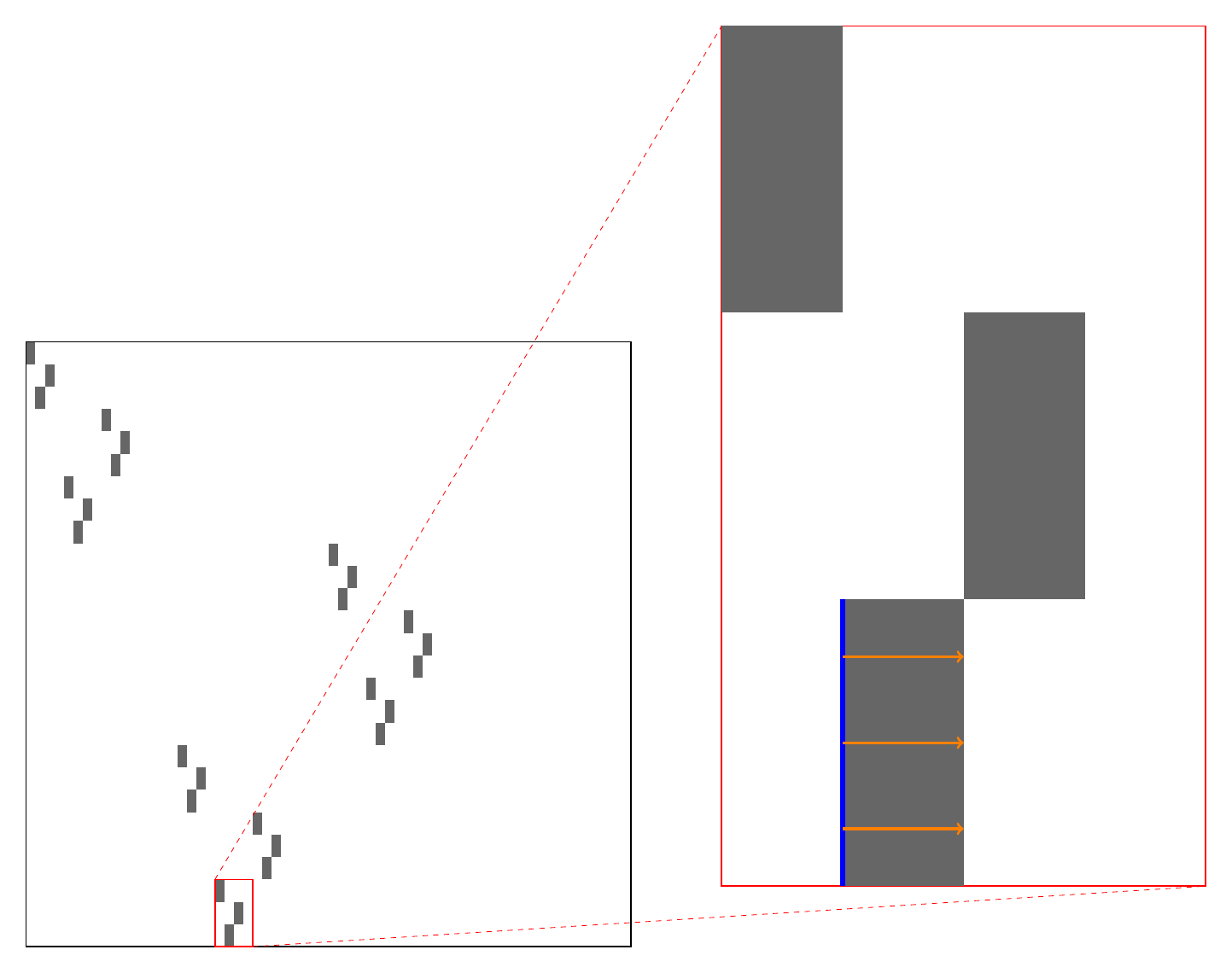}
   \caption{Failure of Cohen's line porosity for the Bedford-McMullen carpets with $M_1=4$, $M_2=3$ and $\mathcal{A}= \{(1,0),(2,1),(0,2)\}$. The black rectangle of size $4^{-3}\times 3^{-3}$ in the right picture has holes after further iteration, but after rescaling it looks identical to the original carpet. Any line segment of length $4^{-3}$ parallel to the $x$-axis, with its left endpoint on the left boundary of this sub-rectangle, has non-empty intersection with the carpet.}    \label{fig:non_line_porosity_for_BM_carpet}
\end{figure}

 (b)   Suppose some row is full, that is, for some $b_0\in \mathbb{Z}_{M_2}$, we know 
    $(a,b_0)\in \mathcal{A}$ for all $a\in \mathbb{Z}_{M_1}$. And we assume every row is non-empty.
    We will show that, for $y_0=\sum_{j=1}^{\infty} b_0M_2^{-j}$ and sufficiently large $k$, 
    any sub-cube of length $5M_2M_1^{-k}$ inside the cube $Q_k$ defined as 
    \[
        Q_k:=[0,M_2^{-k}]\times [y_0,y_0+M_2^{-k}]
    \]
    must contains some point in $\mathbf{X}$, and hence $\mathbf{X}$ is not-box porous. Actually, 
    we will show that for any rectangle $R$ contained in $Q_k$ of the form 
    \[
        R=[\sum_{j=1}^{k} a_jM_1^{-j},
        M_1^{-k}+\sum_{j=1}^{k} a_j]\times 
        [\sum_{j=1}^q b_jM_2^{-j},M_2^{-q}+\sum_{j=1}^q b_jM_2^{-j}]
    \]
    must contain some point in $\mathbf{X}$,
    where $q$ is the unique integer $M_1^{-k}\leq M_2^{-q}<M_1^{-k-1}$, 
    and $a_j\in \mathbb{Z}_{M_1},\,b_j\in \mathbb{Z}_{M_2}$.
    Note that we must have $q>k$ when $k$ is large. 
    Since $R\subset Q_k$, we know $b_j=b_0$ for all $j\leq k$.
    For integer $j\in [k+1,q]$, we further choose $a_j\in \mathbb{Z}_{M_1}$ so that $(a_j,b_j)\in \mathcal{A}$(This is valid 
    since every row is non-empty).
    We define the point $(x^*,y^*)$ as
    \[
        x^*:=\sum_{j=1}^{q} a_jM_1^{-j},\quad y^*=\sum_{j=1}^{q} b_jM_2^{-j}+\sum_{j=q+1}^\infty b_0M_2^{-j}
    \]
    Then it's easy to see that $(x^*,y^*)\in R\cap \mathbf{X}$ by definition, completing 
    the proof. See Figure \ref{fig:no_box_porosity_for_BM_carpet_due_to_full_row}.
\end{proof}

\begin{figure}[htbp]
    \centering
    \includegraphics[width=0.8\textwidth]{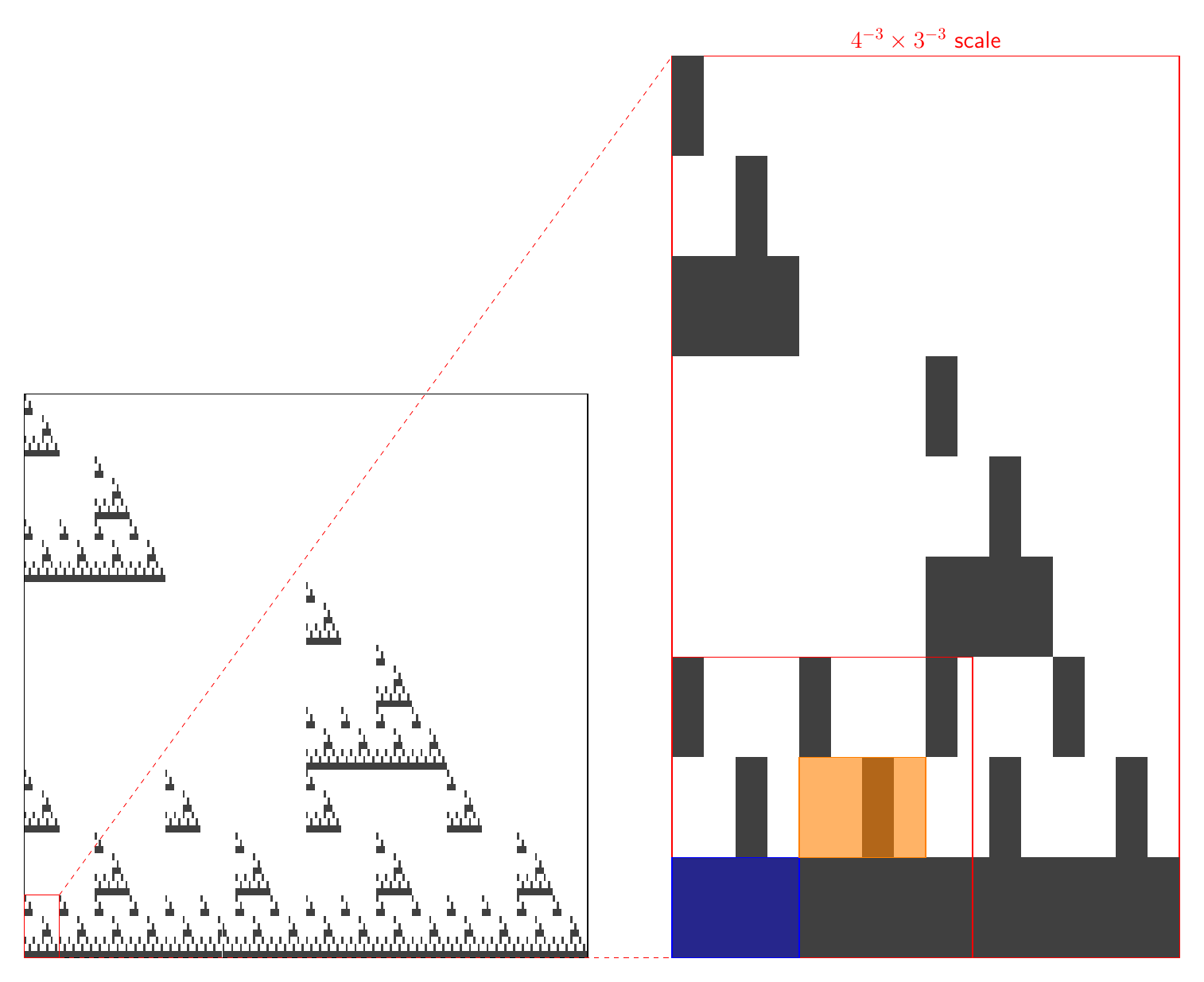}
   \caption{Failure of ball porosity for the Bedford-McMullen carpets
   with $M_1=4$, $M_2=3$ and $\mathcal{A} = \{(0,0), (1,0), (2,0), (3,0), (2,1), (0,2)\}$. 
   In the original coordinate system, the red rectangle corresponds to the region $[0, 3^{-3}] \times [0, 3^{-3}]$. 
   The blue and orange rectangles denote the regions $[0, 4^{-3}] \times [0, 3^{-4}]$ and $[4^{-3}, 2 \cdot 4^{-3}] \times [3^{-4}, 2 \cdot 3^{-4}]$, 
   both the blue and orange rectangles necessarily intersect $\mathbf{X}$. 
   In the $k$-th level iteration, the sub-cube which has empty intersection with $\mathbf{X}$ must have scale $\lesssim M_1^{-k}$, much smaller than $M_2^{-k}$. }    
   \label{fig:no_box_porosity_for_BM_carpet_due_to_full_row}
\end{figure}

For the anisotropic rescaled BM carpet, the same argument shows that 
\begin{proposition}
    Let $\mathbf{Y}$ be the Bedford-McMullen carpet iterated from the alphabet $\mathcal{B}\subset \mathbb{Z}_{\mathbf M}$. 
    Let $\widetilde{\mathbf{Y}}=\{(x/h,y/h^\alpha):(x,y)\in \mathbf{Y}\}$.
    \begin{enumerate}[label=\emph{(\alph*)}]
        \item   If $\mathcal{B}$ is \textbf{nonempty} for each column, then for any 
                $\beta\in (0,\alpha]$, there is no uniform 
                porous constant $\gamma$ for the line porosity condition at scale $h^{-\beta}$ 
                for $\widetilde{\mathbf{Y}}$ as $h\to 0$.
        \item   If some column of $\mathcal{B}$ is \textbf{full}, and every column is \textbf{non-empty,} 
                then for any $\beta\in (0,\alpha]$, there is no uniform 
                porous constant $\gamma$ for the ball porosity condition at scale $h^{-\beta}$ 
                for $\widetilde{\mathbf{Y}}$ as $h\to 0$. This is also equivalent.
\end{enumerate}
\end{proposition}

\begin{remark}
If there is an empty row in $\mathcal A$, then $\mathbf{X}$ 
is trivially ball-porous.
And in this case the projection of $\mathbf{X}$ onto the $y$-axis is contained in a Cantor set, 
so one can easily apply the one-dimensional FUP to obtain the same result for $\mathbf{X}$ as in Theorem \ref{thm:cfup}.
\end{remark}

\begin{remark}
    If every row of $\mathcal{A}$ is not full, then $\mathbf{X}$ is 
    $\alpha$-anisotropic porous as shown in example \ref{ex:BMcarpet-porous}, hence also ball porous by the equivalence condition. If furthermore $\mathcal A$ is non-empty for each row, then $\mathbf X$ is not line porous of Cohen.  Conversely, the line porous set $\mathbf C^2$, where $\mathbf C$ is the one-dimensional $1/3$-Cantor set, is not anisotropic porous for any $\alpha\neq 1$.  
\end{remark}

\vskip2mm
\noindent\textbf{Acknowledgements} 
Long Jin and Hong Zhang are supported by National Key R\&D Program of China No. 2022YFA1007400. Long Jin is also supported by National Natural Science Foundation of China  No. 12525106 and New Cornerstone Investigator Program 100001127.
An Zhang is partially supported by National Key R\&D Program of China No. 2024YFA1015300, Beijing Natural Science Foundation No. 1242009, National Natural Science Foundation of China No. 11801536, the China Scholarship Council No. 202506020208 and the Fundamental Research Funds for the Central Universities. 
The authors would like to thank Stéphane Nonnenmacher, Zuoqin Wang and Jingbo Xu for inspiring discussions.

\noindent\textbf{Statements} 
The authors have no relevant financial or non-financial interests to disclose. Data sharing is not applicable to this article as no datasets were used.
\bibliographystyle{alpha}
\bibliography{ref}

\end{document}